\documentclass[a4paper,12pt,twoside, titlepage]{article}
\usepackage[latin1]{inputenc} % permette di scrivere gli accenti della tastiera su WIN
\usepackage{graphicx}
\usepackage[pdftex,bookmarks]{hyperref}%%indice dinamico
\usepackage{amsmath}    %scrittura caratteri matematici
\usepackage{amssymb}
\usepackage{amsfonts}
\usepackage{amsmath,amsthm, leftindex}
\usepackage{textcomp}
\usepackage{slashed}
\usepackage{color}
\usepackage[all]{xy}
\usepackage{slashed}
\usepackage{faktor}
\usepackage{latexsym}
\usepackage{mathrsfs}
\usepackage{tikz-cd}
\usepackage{enumitem}
\usepackage[titletoc,toc,title]{appendix}
\usepackage{rotating}
\usepackage{color}
\tikzcdset{scale cd/.style={every label/.append style={scale=#1},
		cells={nodes={scale=#1}}}}

\usepackage{fancyhdr}
\usepackage{mathpazo}
\usepackage[scaled=.95]{helvet}
\usepackage{courier}

\numberwithin{equation}{section}
\newtheorem{theorem}{Theorem}[section]
\newtheorem{lemma}[theorem]{Lemma}
\newtheorem{proposition}[theorem]{Proposition}
\newtheorem{corollary}[theorem]{Corollary}

\theoremstyle{definition}
\newtheorem{definition}[theorem]{Definition}
\newtheorem{example}[theorem]{Example}
\newtheorem{examples}[theorem]{Examples}
\newtheorem{remark}[theorem]{Remark}
\theoremstyle{plain}% default

\def\RR{\ensuremath{\mathbb{R}}}

\date{}

\newtheorem{thm}{Theorem}[section]

\newtheorem{prop}{Proposition}[section]
\newtheorem{defn}{Definition}[section]

\newcommand{\thistheoremname}{}
\newtheorem{genericthm}[thm]{\thistheoremname}

\usepackage{rotating}

\tikzset{
	labl/.style={anchor=south, rotate=90, inner sep=.5mm}
}

\begin{document}
	\pagestyle{plain}
	
	\title{The Stolz' positive scalar curvature sequence for $\Gamma$-proper manifolds and depth-1 pseudomanifolds}
	\author{Massimiliano Puglisi}
	\maketitle

	\setcounter{page}{1}
	\tableofcontents
\newpage

	\section{Introduction}

	The Atiyah-Singer theorem, and the consequent index theory, have provided modern mathematics with important tools capable of connecting fields such as analysis, geometry, and global topology. In their original paper (see \cite{AtiyahSinger} and their follow-up works), M.F. Atiyah and I. Singer relate the index of elliptic differential operators (the Fredholm one, defined in terms of the dimensions of Ker and Coker) acting on sections of smooth vector bundles, with purely topological data concerning the operator itself and the underlying space.
	
	One of the main areas where index theory has achieved significant results and has been extensively studied is in the context of metrics with positive scalar curvature.
	If \( M \) is a smooth \( n \)-dimensional manifold and \( g \) is a Riemannian metric on it, recall that the scalar curvature can be defined as the trace of its Ricci tensor or, in even more "geometric" terms, its evaluation at a point \( x \in M \) is the term \( k_g(x) \) that appears in the following expansion in \( \epsilon \):
	
	\[
	\frac{\text{vol}(B_\epsilon(M,x))}{\text{vol}(B_\epsilon(\mathbb{R}^n,o))} = 1 - \frac{k_g(x)}{6(n+2)}\epsilon^2 + O(\epsilon^4),
	\]
	where \( B_\epsilon \) denotes balls of infinitesimal radius \( \epsilon \) and \( \text{vol} \) denotes the volume functions associated with the respective metrics.
	
	At this point, one might ask: "Why positive scalar curvature?".
	The answer lies in the resolution of the "prescribed scalar curvature problem", primarily due to J. Kazdan and F. Warner, who in \cite{KazdanWarner} proved that any smooth closed manifold \( M \) of dimension at least 3 admits a metric with negative scalar curvature. They went further: any smooth function \( f \) that takes negative values somewhere is the scalar curvature of some metric on \( M \).
	
	Regarding the positive case, unfortunately, no similar result exists. Specifically, the problem concerns two aspects: the existence of metrics with positive scalar curvature (often abbreviated as psc-metrics) and, if possible, their classification up to suitable relations. 
	
	Concerning existence, the most powerful tool available in index theory is the use of appropriate \textit{obstructions} to the existence of psc-metrics. In this sense, the most famous result pertains to the so-called Schrödinger-Lichnerowicz formula, established for Dirac operators $D$ on spin manifolds. This, given by:
	
	\[D^2 = \nabla^*\nabla + \frac{1}{4}k_g,\]
	relates the scalar curvature to the so-called Dirac Laplacian $D^2$ and has the significant implication that if the scalar curvature \(k_g\) is (uniformly) positive, then the operator \(D\) is invertible. For a smooth closed manifold \(M\), this implies that if the Fredholm index of the operator \(D\) is non-zero, then \(M\) cannot admit psc-metrics. Thus, this index is precisely an obstruction to its existence.
	
	From here, numerous generalizations and stronger obstructions have been obtained, thereby refining the notion of index. For example, the \textit{Rosenberg index} \(\alpha^\Gamma(M)\) (see \cite{Ros1, Ros3})is the foundation of the Gromov-Lawson-Rosenberg conjecture. This conjecture asserts that for a closed, connected, smooth spin manifold of dimension \(\geq 5\), the index \(\alpha^\Gamma(M)\), defined as an element in $KO$-theory (namely, $K$-theory for real $C^*$-algebras) of the real reduced \(C^*\)-algebra of the fundamental group \(\Gamma=\pi_1(M)\), namely   \(C^{\ast}_{r,\mathbb{R}}\pi_1(M)\), is a primary obstruction to the existence of psc-metrics on \(M\) (i.e., its vanishing is a necessary and sufficient condition for their existence).
	
	This conjecture has been proven true under certain conditions. In his works \cite{Stolz1} and \cite{Stolz2}, Stolz established that it holds if the manifold \(M\) has a trivial fundamental group. More generally, if \(\pi_1(M)\) satisfies the \textit{strong Novikov conjecture}, then the conjecture is verified in its stable formulation. Recall that the strong Novikov conjecture, along with the \textit{Baum-Connes conjecture}, concerns the relationship between the \(K\)-theory of the reduced \(C^*\)-algebra of a group and the \(K\)-homology of an appropriate classifying space of that group. Both can be formulated in terms of the so-called \textit{assembly map} defined from these \(K\)-homology groups to the \(K\)-theory groups. In particular, the strong Novikov conjecture asserts that this map is injective, while the Baum-Connes conjecture asserts that it establishes an isomorphism.
	
	Index theory for Dirac-type operators is generally formulated in terms of operator algebras and \(K\)-theory. The general idea is that such operators define elements in certain \(C^*\)-algebras $A$, and, in particular, using auxiliary structures and tools such as functional calculus, these elements represent classes in the \(K\)-theory of \(C^*\)-algebras. In this context, the condition for an operator to be Fredholm (i.e., invertible modulo compact operators) is replaced by the more general condition of invertibility within the \(C^*\)-algebra modulo one of its ideals.

	Assuming \( I \subset A \) is such an ideal, then considering the short exact sequence:
	
	\[ 0 \rightarrow I \rightarrow A \rightarrow A/I \rightarrow 0 \]
	which induces a long exact sequence in K-theory, then we assoiate to the Dirac operator on an n-dimensional manifold a fundamental class \([M] \in K_{n+1}(A/I)\) ($KO_{n+1}(A/I)$ when dealing with real $C^*$-algebras).
	
	The index is then given by applying to this class the boundary map \(\delta\) associated with the long exact sequence, so \(\text{ind}(D)=\delta([M]) \in K_n(I)\). At this point, the additional geometric condition regarding the uniform positivity of the scalar curvature of a metric \(g\) ensures that the operator is already invertible in \(A\). This implies the existence of a lift of the fundamental class, namely a class \(\rho(M,g) \in K_{n+1}(A)\) and, by exactness of the K-theory sequence, the vanishing of the index class. It is observed that the group \(K_{n+1}(A)\) is often called the structure group and \(\rho\) is an example of a \textit{secondary invariant}.

	In the case of a compact manifold, all this formulation describes the index exactly like the usual Fredholm index. However, an important generalization, where the classical Fredholm condition fails, concerns the case of non-compact manifolds. From here, we move to the so-called \textit{coarse index theory}, based on techniques due to John Roe, which then goes through the \textit{Roe algebras}.
	
	These, introduced in detail in \ref{secroealgebras} in a more general setting, are specific operator algebras defined on spaces endowed with a representation of the \(C^*\)-algebra \(C_0(X)\) of functions vanishing at infinity defined on a certain space \(X\). These fit into a short exact sequence:
	
	\[
	0 \to C^*(X) \to D^*(X) \to D^*(X)/C^*(X) \to 0,
	\]
	which in particular gives rise to a long exact sequence in K-theory:
	
	\[
	\ldots \to K_{n+1}(C^*(X)) \to K_{n+1}(D^*(X)) \to K_{n+1}(D^*(X)/C^*(X)) \xrightarrow{\delta} K_{n}(C^*(X)) \to \ldots
	\]
	known as the \textit{Higson-Roe surgery exact sequence}. This sequence gives rise, in particular, to an index class \( \text{ind}(D) = \delta([M]) \in K_n(C^*(M)) \) and a class $\rho$ in $K_{n+1}(D^*(M))$.
	
	Secondary invariants, of which \(\rho\), as mentioned above, is an example, are useful for classifying metrics with positive scalar curvature on a given manifold \(M\). Certainly, such a classification is useful if it allows distinguishing "classes" of metrics with respect to a certain equivalence relation. This relation will be given by \textit{concordance}, which will be defined later in \ref{stolzsequence}, and in that same section, we will introduce in detail the main tool enabling such classification: the Stolz sequence. This has the following form:
	
	\[
	\ldots \to R^{spin}_{n+1}(M) \to Pos^{spin}_n(M) \to \Omega^{spin}_n(M) \to R^{spin}_n(M) \to \ldots
	\]
	where all of these are bordism groups whose definitions involve the geometric content given by a positive scalar curvature metric.
	The group $Pos^{spin}_n(M)$, in particular, contains information about psc-metrics on \( M \) modulo an appropriate bordism relation. The group $R^{spin}_{n+1}(M) \simeq R^{spin}_{n+1}(B\pi_1(M))$ (this isomorphism will follow from Theorem \ref{G2conn}), on the other hand, as stated later in Theorem \ref{freeactionofR}, acts freely and transitively on the space of concordance classes of psc-metrics on \( M \).
	This has been used, for instance, in \cite{SZ,PSZ} in order to give an estimation of the virtual rank of this affine group and of the moduli space of concordance classes of psc metrics, obtained through the action of the diffeomorphism group of $M$.
	
	A fundamental step in \cite{SZ} uses the fact, that a 2-connected map between CW-complexes, such as the classyfing map $u\colon M\to B\pi_1(M)$, induces an isomorphism between $\mathrm{R}^{spin}_*$ groups. 
	
	Unfortunately, these $R$ bordism groups are difficult to compute. This problem is generally addressed by mapping this sequence into one where the terms are better understood.
	In particular, one such mapping we want to refer to is the one carried out in \cite{PiazzaSchick}, where the Stolz sequence is indeed mapped into the analytical sequence of Higson-Roe (in its equivariant formulation). In their version (\cite[Theorem 1.28]{PiazzaSchick}), valid in the odd-dimensional case (but later proved on each dimension, see \cite{XieYu}), the mapping has the following form:
	
	\begin{equation} 
		\begin{tikzcd}[sep=small, row sep=1.5em]
			\ldots \ar[r] & R^{spin}_{n+1}(M) \ar[r] \ar[d] & Pos^{spin}_{n}(M)\ar[r] \ar[d] & \Omega^{spin}_{n}(M)\ar[r] \ar[d] & R^{spin}_{n}(M)\ar[r] \ar[d] & \ldots \\
			\ldots \ar[r] & K_{n+1}(C^{\ast}_{r}\Gamma) \ar[r] & K_{n+1}(D^{\ast}(\widetilde{M})^\Gamma) \ar[r] & K_n(M) \ar[r, "\mu"] & K_n(C^{\ast}_{r}\Gamma) \ar[r] & \ldots
		\end{tikzcd}	
	\end{equation}
where \(\Gamma = \pi_1(M)\), \(\widetilde{M} \rightarrow M\) is the universal cover of \(M\), and the well-known isomorphisms \(K_*(D^*(\widetilde{M})^\Gamma / C^*(\widetilde{M})^\Gamma) \simeq K_{* - 1}(M)\) and \(K_*(C^*(\widetilde{M})^\Gamma) \simeq K_*(C^*_r \Gamma)\) have been used.
Observe the the map $\mu$, in the universal version of this sequence for the group $\Gamma$, is the assembly map mentioned before involved in the Baum-Connes conjectures.

	\subsection{Outline of main results}
	
	The work contained in this thesis mainly revolves around the Stolz sequence, and in particular its $R$-groups, which will be introduced in more detail in the following section. Specifically, the thesis can be fundamentally considered as composed of two parts.
	
	The first part, meaning Chapter \ref{g,f}, which is based on a joint work with Thomas Schick and Vito Felice Zenobi, addresses the case of $(G, \mathcal{F})$-spaces, i.e., proper $G$-spaces whose isotropy groups lie in a given family $\mathcal{F}$ of subgroups of $G$. 
	
	The second part, essentially given by the rest of the discussion, investigates the case of stratified spaces in their "smooth" version, where singularities, even non-isolated ones, are present. See also \cite{Botvinnik} and \cite{BR} for general references regarding positive scalar curvature metrics in singular cases. However, most of the results in this part regard spaces with particular types of singularities, namely $(L,G)$-singularities, mainly introduced in \cite{BPR2} and \cite{BPR1}.
	
	Now let's summarize the main results contained here. 
	
	In the case of $(G, \mathcal{F})$-spaces, an analogous version of the Stolz sequence will be introduced, whose exactness is derived in a nearly identical manner to the standard case (see Proposition \ref{exactg,fstolz}). Essentially, it requires that cycles be represented by spin $(G, \mathcal{F})$-manifolds and maps to be $G$-equivariant.
	Additionally, the following theorem is introduced, stating the general principle mentioned above, and which still holds true even in this context, providing a more direct proof of the original result.\\

\noindent
{\bf Theorem \ref{G2conn}.} {\it Let $f\colon X \rightarrow Y$ be a continuous, 2-connected $G$-map between $(G,\mathcal{F})$-CW-complexes which induces an isomorphism between the fundamental groups of $X$ and $Y$. Then the functorially induced map $f_{*}\colon R_{n}^{spin}(X)^{(G,\mathcal{F})} \rightarrow R_{n}^{spin}(Y)^{(G,\mathcal{F})}$ is an isomorphism.} \\

The main tools in its proof are provided by Morse theory and the Gromov-Lawson theorem (\cite{GL, SchoenYau}), in an appropriately adapted version for the context.

Next, a realization of a universal space for this category of spaces is introduced. Specifically, using the formalism of groupoids, a generalization of the fundamental group for a CW complex will be presented, namely the \textit{fundamental functor} $\Pi_1(X; G, \mathcal{F})$ associated with a $(G, \mathcal{F})$-CW complex $X$. This is a finer object as it takes into account the fundamental groups associated with all the fixed point spaces with respect to the action of subgroups in the family $\mathcal{F}$.

Associated with this functor, a universal space $B\Pi_1(X; G, \mathcal{F})$ will be introduced, and a concrete realization will be given starting from the complex $X$. Specifically, as stated more precisely in the corollary below, every $(G, \mathcal{F})$-CW complex whose fundamental functor is "isomorphic" to $\Pi_1(X; G, \mathcal{F})$ will have a map to this universal space, which, in particular, will establish an isomorphism between their respective Stolz $R$ groups.\\

\noindent
{\bf Corollary \ref{univ2conn}.} {For all $(G, \mathcal{F})$-CW complex $Y$ such that its fundamental functor is isomorphic to that of $X$, that is, there exists a isomorphism of groupoids:

\[\Phi\colon \Pi_1(Y;G,\mathcal{F})\to \Pi_1(X;G,\mathcal{F}),\]
then there exists a $2$-connected $G$-equivariant cellular map $\varphi$ realizing an isomorphism:

\[\varphi_* \colon R_{n}^{spin}(Y)^{(G,\mathcal{F})} \rightarrow R_{n}^{spin}(B \Pi_1(X;G,\mathcal{F}))^{(G,\mathcal{F})}\]}

After a general introduction to stratified spaces, or more precisely to Thom-Mather stratified spaces, and to the class of wedge metrics that can be defined on them, Chapter \ref{smoothlystrspace} introduces a Stolz sequence, namely the \textit{$L$-fibered Stolz sequence} in the depth 1 case:
	\[
\ldots \to R^{spin, L-fib}_{n+1}(X) \to Pos^{spin, L-fib}_n(X) \to \Omega^{spin, L-fib}_n(X) \to R^{spin, L-fib}_n(X) \to \ldots
\]
whose exactness is proven in detail in Theorem \ref{Lstolzexact}.

In Chapter \ref{secstrmorse}, the problem of establishing an analogue of Morse theory in the context of stratified spaces is addressed. This topic, known as Stratified Morse Theory (for a general treatment, see \cite{GMP}), is introduced by approaching the problem directly to stratified spaces of depth 1, specifically to achieve a decomposition of these spaces as $CW$-complexes.

To establish a result analogous to Theorem \ref{G2conn} mentioned above, it will be necessary to restrict to special case of stratified spaces of depth 1. These are referred to as spaces with $(L,G)$-singularities, where $L$, the link associated with the depth-$1$ stratum of the stratified space, is fixed and has the structure of a homogeneous space, i.e., $L = G/K$. This particular geometry, which was introduced in \cite{BPR2, BPR1}, will be detailed in Section \ref{L,Gsetting}.

In Chapter \ref{mappingstolzlg}, the discussion focuses on how it is possible to establish a mapping between the Stolz sequence (in the $(L,G)$ case mentioned above) and the Higson-Roe surgery exact sequence. In particular, the aim is to prove the following result.\\	
	
\noindent	
{\bf Theorem \ref{mapstolzhigsonthm}.} {Denoting by $B\Gamma=E\Gamma/\Gamma$ the classifying space for Galois $\Gamma$-coverings, then the $(L,G)$-Stolz sequence \ref{l,gstolzseq} with $B\Gamma$ as reference space maps to the universal Higson Roe surgery sequence (\ref{higsonroelocalizedu}).
This means that the following diagram is commutative:

\begin{center}
\begin{tikzcd}[sep=small, row sep=1.5em]
		\ldots \ar[r, "\iota"] & R^{spin, (L,G)}_{n+1}(B\Gamma) \ar[r, "\partial"] \ar[d, "Ind^{\Gamma}_{rel}"] & Pos^{spin, (L,G)}_{n}(B\Gamma)\ar[r, "\varphi"] \ar[d, "\rho^{\Gamma}"] & \Omega^{spin, (L,G)}_{n}(B\Gamma)\ar[r, "\iota"] \ar[d, "Ind^{\Gamma}_L"] & R^{spin, (L,G)}_{n}(B\Gamma)\ar[r, "\partial"] \ar[d, "Ind^{\Gamma}_{rel}"] & \ldots \\
		\ldots \ar[r] & KO_{n+1}(C^{\ast}_{r,\mathbb{R}}\Gamma) \ar[r] & KO_n(C^{\ast}_{L,0;\Gamma}) \ar[r] & KO_n(C^{\ast}_{L;\Gamma}) \ar[r] & KO_n(C^{\ast}_{r,\mathbb{R}}\Gamma) \ar[r] & \ldots
\end{tikzcd}	
\end{center}
}

As can be deduced from the statement, the exact sequence below differs from the one previously proposed in the introductory part. Indeed, two equivalent formulations have been used to describe it. The first involves the use of graded K-theory groups, defined as homotopy classes of appropriate morphisms, rather than the standard approach in terms of projections and unitaries. This approach accounts for the grading as well as the real structure of a $C^*$-algebra, with the consequent advantage of being able to address easily the real case. The second involves the use of localization algebras instead of Roe algebras (which will still be introduced in detail within the chapter).

This result, as previously mentioned, has the direct advantage of providing useful information regarding the bordism groups present in the Stolz sequence, which are otherwise difficult to compute. This becomes even more explicit in Chapter \ref{estimate}, where a lower bound estimate on the rank of the structure group $Pos^{spin, (L,G)}$ is proposed as its application.

At the beginning of this chapter, an overview of some results concerning the spaces of metrics definable on a manifold is provided, and in particular, the definition of the index-difference homomorphism is recalled. This homomorphism provides an example of an obstruction to the existence of a concordance relation between two positive scalar curvature metrics.

Finally, all these constructions are revisited in the case of wedge metrics, particularly in the \textit{wedge well-adapted} case, on spaces with $(L,G)$-singularities.

Reducing this homomorphism to the smooth case, thanks to the particular structure of these well-adapted metrics, will allow us to obtain that, under appropriate conditions on a group $\Gamma$, including the validity of the Baum-Connes conjecture (in a rational version), it is rationally surjective onto the $KO$-theory group of its reduced $C^*$-algebra. This, together with the mapping theorem, will then provide information on the minimum rank of the structure group $Pos^{spin, (L,G)}$ (see Theorem \ref{rankpos}). Specifically, the following result will be demonstrated in the final section of this thesis.\\

\noindent
{\bf Theorem \ref{rankpos}.}{ Let $M$ be a compact, spin stratified pseudomanifold with $(L,G)$-singularities of dimension $n \geq 6$ with link $L$ simply connected. Assume that $\Gamma=\pi_1(M)$ satisfies all the hypothesis of proposition \ref{winddiffsurj} and denote by $f\colon M \to B\Gamma$ the 2-connected classifying map of the universal cover of $M$. 

If:

\[k:=\dim\left(Coker\left(f_{\sharp}\colon KO_{n+1}(M) \otimes \mathbb{Q} \to KO_{n+1}(B\Gamma) \otimes \mathbb{Q}\right)\right),\]

then the following estimate holds:

\[rk\left(Pos^{spin,(L,G)}_n(M)\right)\geq k\]}

	\newpage
	
	\section{The Stolz sequence}\label{stolzsequence}
	
	In \cite{stolz} Stolz introduced the notion of R-groups for compact manifolds with boundary and explained how those fit in a exact sequence, namely the Stolz positive scalar curvature sequence. These groups were defined as bordism groups associated to a \textit{supergroup} $\gamma$ whose cycles were given by compact $\gamma$-manifolds with boundary and with a specific positive scalar curvature on such boundary. 
	
	A supergroup $\gamma$ can be briefly thought as a triple made by a group $\pi$, a $\mathbb{Z}_2$-graduation of $\pi$ (i.e. a group homomorphism $ \pi \to \mathbb{Z}_2$) and a suitable group-extension $\widehat{\pi} \to \pi$. endowed with a canonical $\gamma$-structure. On a vector bundle $E$, one can define a $\gamma$-structure as a reduction of the structure group of $E$ for a specific group defined from $\gamma$. For specific supergroups, this can be reduced to the existence of a spin structure on a manifold, which will be the case we discuss here. 
	
	The main application of the $R$-groups and the Stolz sequence has to be found in the classification of concordance classes of positive scalar metrics. We now recall briefly the main definitions and results.

	First of all, we fix a generic space $X$, called \textit{reference space}. In general, \( X \) can be arbitrary. However, we will discuss later two natural choices for \( X \) that turn out to be convenient for our purposes.
	
	Firstly, a brief recall on the notion of bordism. This can be introduced as an equivalence relation: if \( M_1 \) and \( M_2 \) are two smooth compact manifolds of dimension \( n \), we say that \( M_1 \) and \( M_2 \) are \textit{bordant} if there exists a compact \( (n+1) \)-dimensional manifold $W$ such that its boundary is non-empty and equal to \( M_1 \sqcup M_2 \). 
	In particular, we say that $W$ is a bordism between $M_1$ and $M_2$.
	
	\begin{definition}
		 \( \Omega_n \) is the set of the equivalence classes with respect to the previous equivalence relation.  If we denote by \( [M_i] \) its elements (i.e. its cycles), then an addition operation in \( \Omega_n \) is well-defined, given by \( [M_1] + [M_2] = [M_1 \sqcup M_2] \), which makes \( \Omega_n \) an abelian group, which will be called the \textbf{bordism group} of closed $n$-manifolds.
	\end{definition}
	
	\begin{remark}
		A similar reasoning can be made regarding the Cartesian product of manifolds, which thus defines a product \( \Omega_n \times \Omega_m \to \Omega_{n+m} \). With these two operations, \( \Omega_* \) has a ring structure and is called the \textit{bordism ring}.
	\end{remark}
	By enriching the category on which the equivalence relation is defined, different bordism theories are obtained. Obviously, this relation must be compatible with any enrichments. 
	
	For example, if the manifolds \( M_i \) are all considered to have maps \( f_i:  M_i \to  X \), then the bordism will also have a map to \( X \), denoted by \( f \), and this map must restrict to \( f_i \) on their respective boundary components. This gives rise to the bordism ring \( \Omega_*(X) \). 
	
	Another standard example is when manifolds are considered with orientation: in this case, it will be required that a bordism between \( M_1 \) and \( M_2 \) is also oriented, and that the boundary is \( M_1 \sqcup (-M_2) \), where \(-M_2 \) indicates the manifold taken with the opposite orientation. This gives rise to $\Omega_*^{SO}$. 
	This last example can be further generalized to the case of manifolds equipped with a \( G \)-structure. Here, the bordism will also have to be equipped with a \( G \)-structure that restricts appropriately on the boundary. In this case, we will obtain \( \Omega_*^G \) (or, possibly, \( \Omega_*^G(X) \)).
	
	 Then we introduce the following three bordism groups:
	
	\begin{itemize}
		\item $\Omega^{spin}_{n}(X)$ is the spin bordism group with reference space $X$. This means that a cycle here is given by a pair $(M, f)$, where $M$ is a closed spin manifold of dimension $n$ with fixed spin structure on its tangent bundle and $f\colon M \to X$ is a reference map. Bordisms between two cycles have to extend both the reference maps and the spin structures, meaning that two cycles are equivalent if there exists a $n+1$-dimensional spin manifold with boundary given by the union of the two manifolds representing the cycles, together with a spin structure and a map to $X$ which restrict to those of the cycles on the respective boundary components.
		\item $Pos^{spin}_{n}(X)$ is defined as the spin bordism group, but with the additional geometric information of a Riemannian metric of positive scalar curvature. A cycle is made by a triple $(M,f,g)$ where $M$ and $f$ are as before, while $g$ is a psc metric on $M$. Again, in addition to the requirements already discussed for the spin bordism group, the bordisms between two cycles have to be endowed with psc-metrics wich are product-like on a neighborhood of the boundary and restrict to those of the cycles there.
		\item $R^{spin}_{n}(X)$ is the \textit{Stolz R-group} we referred before. Cycles in this group are made by triples $(M,f,g_{\partial})$, where $M$ is a compact, spin n-manifold with boundary, $f \colon M \to X$ is again a reference map and $g_{\partial}$ is a Riemannian psc-metric on the boundary $\partial M$. Bordisms are then manifolds with corners. In particular, two cycles $(M,f,g)$ and $(M',f',g')$ are equivalent if there is a compact, spin $n+1$-manifold whose boundary decomposes as $\partial M \cup_{\partial M} V \cup_{\partial M'} \partial M'$, where $V$ realizes a bordism between $(\partial M, f|_{\partial M}, g|_{\partial M})$ and $(\partial M', f|_{\partial M'}, g|_{\partial M'})$ in the sense of $Pos^{spin}_{n-1}(X)$. Of course, we require that the spin structures and the reference maps of $M$ and $M'$ are extended along the bordism.
	\end{itemize}

	\begin{remark}  All these bordism groups are covariantly functorial in X: in fact, a map $g:X \rightarrow Y$ induces a mapping $g_{*}: \Omega^{spin}_{n}(X) \rightarrow \Omega^{spin}_{n}(Y)$ (similarly for $Pos^{spin}_{n}(X)$ and $R^{spin}_{n}(X)$) by pushing-forward the reference map. 
	\end{remark}
	
Given the definition of the \( R \)-group, it follows directly that if a manifold \( M \) with boundary admits a metric with positive curvature on its boundary \( g_{\partial} \), defining a class in the \( R \)-group that extends over the entire manifold, then this class is null-bordant, meaning \( [M,h]=0 \).
	
	\begin{proposition}
		The previous abelian groups fit into the so called \textbf{Stolz positive scalar curvature exact sequence}:
		
		\[
		\begin{tikzcd}
			\ldots \arrow[r] & R^{spin}_{n+1}(X) \arrow[r,"\partial"] & Pos^{spin}_{n}(X) \arrow[r, "d"] & \Omega^{spin}_{n}(X) \arrow[r,"i"] & R^{spin}_{n}(X) \arrow[r, "\partial"] & \ldots
		\end{tikzcd}
		\]
		where $\partial$ is the mapping sending a manifold to its boundary, $d$ is the forgetful map (i.e. it forgets about the psc metric) and $i$ is the obvious inclusion.
	\end{proposition}
	
	Stolz's R-groups are typically used for the classification of metrics with positive scalar curvature on a fixed manifold, up to concordance. In particular, by concordance here we mean the following equivalence relation:

		\begin{definition}\label{concordant}
		Two Riemannian metrics $g_{1}$, $g_{2}$ of positive scalar curvature are said \textbf{concordant} if there exists a Riemannian psc-metric $h$ on $M \times [0,1]$ such that it is product like on a neighborhood of the boundary and restricts to $g_{1}$  and $g_{2}$ there. 
	\end{definition}

	 Once the manifold \( M \) is fixed, two natural choices for the reference space \( X \) emerge. The first one is to choose \( X = M \), noting that both \( M \) and \( M \times [0,1] \) trivially possess a map to \( X \). The second choice is to set \( X = B\Gamma \), where \( \Gamma = \pi_1(M) \), observing that \( M \) canonically admits a map to this space, and furthermore, every space with fundamental group \( \Gamma \) admits one.
	 
	 Let's consider the case where \( X = B\Gamma \): using techniques from surgery theory, it can be shown that the null-bordance of the bordism class in \( R^{spin}_n(B\Gamma) \) is not only a necessary condition but also a sufficient one for the existence of an extension to positive curvature of the metric defined on the boundary. Furthermore, in the case where the dimension of the manifold \( M \) is \( \geq 5 \), Stolz exhibited the following result:

	\begin{theorem}\label{freeactionofR}
		Let M be a closed spin manifold, with dimension $n \geq 5$ and fundamental group $\Gamma$. Suppose that M admits a psc metric $g_0$, implying that $[(M,f:M \rightarrow B\Gamma, g_0|_{\partial M})]=0 \in R_{n}^{spin}(B\Gamma)$ vanishes. Then the group $R_{n+1}^{spin}(B\Gamma)$ acts freely and transitively on the set of concordance classes of metrics of positive scalar curvature on M.
	\end{theorem}
	
This theorem, in particular, asserts that, given a psc metric \( g_0 \) on \( M \), there exists a map:

\[ i_{g_0}: C^+(M) \rightarrow R^{spin}_n(B\Gamma) \]
which is bijective for every choice of \( g_0 \), where \( C^+(M) \) is the set of concordance classes of psc metrics on \( M \). This implies that the two are non-canonically isomorphic: in particular, \( C^+(M) \) assumes a structure of \( R^{spin}_n(B\Gamma) \)-torsor, and \( i_{g_0} \) induces a group structure, in which \( g_0 \) plays the role of the identity, isomorphic to \( R^{spin}_n(B\Gamma) \).

In any case, beyond the specific utilities that arise from a given problem, the two choices for the reference spaces we discussed before are equivalent. Indeed, in the case where \( \pi_1(M) = \Gamma \), the existence of a $2$-connected map\footnote{A continuous map $f\colon X \to Y$ between topological spaces is \textit{n-connected} if for all $x \in X$ the induced map $f_*\colon \pi_i(X,x) \to \pi_i(Y,f(x))$ is an isomorphism for all $0<i<n$ and is surjective for $i=n$} \( f: M \to B\Gamma \) is ensured, which, together with the following theorem, establishes the equivalence of such choices.
	
	\begin{theorem}\label{2connthm}
		Let $f: X \rightarrow Y$ be a continuous, 2-connected map. Then the functorially induced map $f_{*}: R_{n}^{spin}(X) \rightarrow R_{n}^{spin}(Y)$ is an isomorphism.
	\end{theorem}

	The proof of this theorem relies on the Gromov-Lawson theorem and basic tools from Morse theory (we will be more detailed later). The objective of the next sections will be to extend this notions and prove this theorem in other contexts, e.g. for $(G, \mathcal{F})$-CW-complexes and some kind of smoothly stratified spaces. In particular, this theorem will follow from these new results.
	
	\newpage
	
	\section{(G,$\mathcal{F}$)-Stolz} \label{g,f}
	
	\subsection{Proper actions of discrete groups}

	In this section, we begin by providing a general introduction to transformation groups and the notion of a \( G \)-space. We will then introduce the concept of a \( G \)-CW-complex for a family of subgroups of \( G \) and observe how the Stolz sequence can be naturally adapted for these spaces.
	
	Suppose \( G \) is a topological group and \( X \) is a topological space. A \textit{left \( G \)-space} (resp. right \( G \)-space) is a pair \( (X,\rho) \) where \( \rho: G \times X \to X \) is a continuous map satisfying the following properties:
	\begin{itemize}
	\item \( \rho(e,x) = x \) for all \( x \in X \), where \( e \) is the identity element of \( G \);
	\item \( \rho(g, \rho(h,x)) = \rho(gh, x) \) for all \( g,h \in G \) and \( x \in X \) for a left \( G \)-space (resp. \( \rho(h, \rho(g,x)) \rho(gh, x)  \) for a right \( G \)-space).
	\end{itemize}

	\begin{remark}
		For a left \( G \)-space, \( \rho \) is called a \textit{left \( G \)-action} and is often abbreviated as \( \rho(g,x) = gx \) (resp. for a right \( G \)-space, the \textit{right \( G \)-action} \( \rho \) is denoted by \( \rho(g,x) = xg \)).
	\end{remark}
	
	In general, we work with left \( G \)-spaces, which is why unless otherwise specified, it is assumed that a \( G \)-space has a left-action. So let \( X \) be a \( G \)-space, and consider the following equivalence relation in \( X \):
	
	\[ x \sim gx, \quad \forall g \in G. \]
	
	The set of its equivalence classes \( X/G \) is called the \textit{orbit space} of \( X \) and is equipped with the quotient topology induced by the quotient map \( X \to X/G \). The class of \( x \in X \), denoted by \( Gx \), is called the \textit{orbit} of \( x \). An action is called \textit{transitive} if $X$ has only one orbit, i.e. for each pair of points $x,y \in X$, there is a $g \in G$ such that $y=gx$. 
	
\begin{definition}\label{homogeneousspace}
	Consider the right action of a subgroup \( H \subseteq G \) on \( G \) itself and the orbit space with respect to this action (whose elements are thus indicated as \( gH \), with \( g \in G \)).
	We can introduce the following left action:
	
	\[ G \times G/H \to G/H, \quad (g',gH) \mapsto g'gH. \]
	
	We define a \textit{homogeneous space} any \( G \)-space \( G/H \) equipped with this action.
\end{definition}

\begin{remark}
	Alternatively, a homogeneous space is any \( G \)-space on which \( G \) acts transitively. In fact, if \( X \) is such a space, let's choose a point \( x \in X \) and consider its \textit{stabilizer} (also called the \textit{isotropy group} of \( x \)), which is the subgroup \( G_x = \{ g \in G | gx = x \} \)\footnote{Observe that $G_{gx}=gG_xg^{-1}$}.Then the points of \( X \) are in correspondence with the classes in \( G/G_x \) (in particular, the point \( x \) corresponds to the class \( eG_x \)). Conversely, a \( G \)-space defined as in the previous definition is a space equipped with a transitive action.
\end{remark}

\begin{example}
	Examples of homogenous spaces are:
	\begin{itemize}
		\item The orthogonal group $O(n+1)$ acts transitively on the $n$-sphere $S^n$. However, the stabilizer subgroup of each point on the sphere is isomorphic to $O(n)$, and then $S^n \simeq O(n+1)/O(n)$.
		\item Consider the \textit{Stiefel manifold} given by the set $V_k(\mathbb{R}^n)$ of all orthonormal $k$-frames in $\mathbb{R}^n$. Of course, orthogonal transformations preserves orthonormality and two frames are obtained one by another by an orthogonal transformation. However, since $k \neq n$, the group $O(n-k)$ keeps fixed the $k$-frames. Then $V_k(\mathbb{R}^n) \simeq O(n)/O(n-k)$. Similarly, one gets $V_k(\mathbb{C}^n) \simeq U(n)/U(n-k)$ and $V_k(\mathbb{H}^n)\simeq Sp(n)/Sp(n-k)$.
	\end{itemize}
\end{example}
	
If $H$ is a subgroup of $G$, then the \textit{$H$-fixed point set} $X^H$ is defined as:

\[X^H:=\{x \in X \ | \ \ hx=x, \ \forall h \in H\}.\]

Suppose \( f: X \to Y \) is a (continuous) map of \( G \)-spaces. \( f \) is called \textit{\( G \)-equivariant} if it preserves the \( G \)-actions, i.e., \( f(gx) = gf(x) \) for every \( g \in G \) and \( x \in X \). The space of \( G \)-equivariant maps from \( X \) to \( Y \), denoted by \( C_G(X,Y) \), can be topologized using the \textit{compact-open topology}, obtained by considering the subbase generated by the subsets \( W(K,\mathcal{U}) = \{ f \in C_G(X,Y) | f(K) \subset \mathcal{U} \} \), where \( K \subset X \) is a compact subset and \( \mathcal{U} \subset Y \) is an open set. 

Two \( G \)-equivariant maps \( f_0, f_1: X \to Y \) are said to be \textit{\( G \)-homotopic} if there exists a \( G \)-equivariant map \( F: X \times [0,1] \to Y \) such that \( F(-,0) = f_0 \) and \( F(-,1) = f_1 \), considering the interval \([0,1]\) equipped with the trivial action of \( G \) and \( X \times [0,1] \) as a \( G \)-space with the diagonal action of \( G \).

\begin{remark}
Observe that for each subgroup $H \subset G$, a map $f \in C_G(X,Y)$ preserves the $H$-fixed point set, i.e $f(X^H)\subseteq Y^H$. 
Clearly, \( f \) maps the orbit of a given point \( x \in X \) to the orbit of its image. For this reason, \( f \) induces a map between the orbit spaces \( f/G: X/G \to Y/G \). Furthermore, observe that for every point \( x \in X \), we have \( G_x \subset G_{f(x)} \).
\end{remark}

Let's recall that a continuous map \( f: X \to Y \) is called \textit{proper} if it is a closed map and the fibers \( f^{-1}(y) \) of each point \( y \in Y \) are compact. If \( X \) and \( Y \) are Hausdorff spaces and \( Y \) is locally compact, then \( f \) is a proper map if for every compact subset \( K \subset Y \), \( f^{-1}(K) \) is compact\footnote{In this case, \( X \) is also locally compact.}. 

If \( X \) is a \( G \)-space, and the action \( \rho \) of \( G \) on \( X \) is called proper (or alternatively, \( X \) is called a \textit{proper \( G \)-space}) if:

\[ \theta: G \times X \to X \times X, \quad (g,x) \mapsto (x,gx), \]
is a proper map.

	\subsection{(G-$\mathcal{F}$)-CW-complex}
	
	Let us fix a discrete group $G$ and a family $\mathcal{F}$ of finite subgroups of $G$, which is closed under conjugation and finite intersection. 
	
	\begin{definition}
		A \textbf{(G,$\mathcal{F}$)-CW-complex} X is a $G$-space together
		with a $G$-invariant filtration
		$$\emptyset = X_{(-1)}\subseteq X_{(0)} \subseteq X_{(1)} \subseteq \dots \subseteq X_{(n)}\subseteq \dots \subseteq\bigcup_{n\geq0}X_{(n)}=X$$
		such that:
		\begin{itemize}
			\item $X$ carries the colimit topology with respect to this filtration, meaning that a set $C \subseteq X$ is closed if and only if $C \cap X_{(n)}$ is closed $\forall n \geq 0$;
			\item $X_{(n)}$
			is obtained from $X_{(n-1)}$ for each $n\geq0$ by attaching equivariant $n$-dimensional
			cells, i.e. there exists a $G$-pushout
			\begin{equation}\label{gcwcomplex}
			\xymatrix{\bigsqcup_{i\in I_n}G/H_i\times S^{n-1}\ar[r]\ar[d]& X_{(n-1)}\ar[d]\\
				\bigsqcup_{i\in I_n}G/H_i\times D^n\ar[r]& X_{(n)}}
			\end{equation}
			where the $H_i$ belong to $\mathcal{F}$.
		\end{itemize}

		\end{definition}

		The space $X_{(n)}$ is called the \textit{n-skeleton} of $X$. An \textit{equivariant open n-dimensional cell} (closed cells are the respective closure) is a $G$-subset of $X_{(n)} - X_{(n-1)}$, namely a preimage of a path component of $G \backslash (X_{(n)} - X_{(n-1)})$.
		
		\begin{remark}
			In the above definiton, only the filtration by skeletons belongs to the $(G,\mathcal{F})$-CW-structure but not the $G$-pushouts, only their existence is required. However, once a $G$-pushout is chosen, then the equivariant closed n-dimensional cells are given explicitly by (\ref{gcwcomplex}) as the image of each $G/H_i \times D^n$ via the lower arrow.
		\end{remark}
		%An equivariant open n-dimensional cell is a G-component
		%of Xn â Xnâ1, i.e. the preimage of a path component of G\(Xn â Xnâ1).
		%The closure of an equivariant open n-dimensional cell is called an equivariant
		%closed n-dimensional cell . If one has chosen the G-pushouts in Definition
		%1.1, then the equivariant open n-dimensional cells are the G-subspaces
		%Qi(G/Hi Ã(Dn âSnâ1)) and the equivariant closed n-dimensional cells are the
		%G-subspaces Qi(G/Hi Ã Dn).

	\begin{remark}
		We assumed that the family of subgroups \( \mathcal{F} \) was composed of finite subgroups so that these would be compact. In this way, \( X \) turns out to be a proper \( (G,\mathcal{F}) \)-CW-complex. Recall that a \( G \)-space is said to be proper if for every pair of points \( x \) and \( y \) in \( X \), there are open neighborhoods \( V_x \) and \( V_y \) such that the closure of the set \( \{ g \in G \ | \ gV_x \cap V_y \neq \emptyset \} \) is compact.
	\end{remark}

	Let us fix some notations. Let $f: X \rightarrow Y$   be a continuous $G$-equivariant map between $(G,\mathcal{F})$-spaces. If $H \subseteq G$ is a subgroup of $G$ and $X^H$ is the $H$-fixed point set, we will denote by $f^H\colon X^H\to Y^H$ the restriction of $f$ to the $H$-fixed point sets, with $H$ a subgroup of $G$. 
	
	\begin{definition} 
		We say that $f$ is \textbf{cellular} if both X and Y are CW-complexes and one has $f(X_{(k)}) \subseteq Y_{(k)}$, i.e. it preserves the filtrations of $X$ and $Y$.
	\end{definition}
	We have a Cellular Approximation Theorem  in the equivariant context too \cite[Theorem 2.1]{ttd}. 
	\begin{theorem} \label{cellular}
		Let $f\colon X\to Y$ be a $G$-map. Then there exists a G-homotopy $h\colon X\times I\to Y$ such that $h_0=f$ and $h_1$ is cellular.
	\end{theorem}
	
	We also have an equivariant version of the Whitehead Theorem for $(G,\mathcal{F})$-spaces. 
	
	\begin{defn}
		Consider a function $\nu\colon \mathcal{F}\to\mathbb{N}$, then we say that $f$ is \textbf{$\nu$-connected} if $f^H$ is $\nu(H)$-connected for all $H\in \mathcal{F}$, namely the induced maps are isomorphisms on the first $\nu(H)-1$ homotopy groups of $X^H$ and $Y^H$ and a surjection on the $\nu(H)$-th one. 
		In particular, we say that it is {$k$-Connected} if $\nu$ is constantly equal to $k\geq 0$. 
		
		Moreover we say that a relative $G$-complex $(X,A)$ has dimension less or equal to $\nu$ if the cells in $X\setminus A$ are of the form $G/H\times D^k$ with $k\leq\nu(H)$. 
	\end{defn}
	\begin{proposition} \label{whitehead}
		Let $f\colon Y\to Z$ be a $\nu$-connected map and $X$ a $(G, \mathcal{F})$-CW-complex. Then $f_*\colon [X,Y]^G\to [X,Z]^G$ is surjective (resp. bijective) if $\dim X\leq\nu$ (resp. $\dim X<\nu$).
	\end{proposition}
	
	As a particular case of this proposition we obtain the following result \cite[Proposition 2.7]{ttd}. 
	
	\begin{proposition}
		Let $f\colon Y\to Z$ be a $G$-map between $(G,\mathcal{F})$-complexes such that $f^H$ is a homotopy equivalence for all $H$. Then $f$ is a $G$-homotopy equivalence.
	\end{proposition}

	\subsection{The (G,$\mathcal{F}$)-Stolz sequence}

	Let $X$ be a $(G,\mathcal{F})$-CW-complex. In the following, by a spin $(G, \mathcal{F})$-manifold we mean a spin manifold with a  $G$-action preserving the spin structure and whose isotropy groups are given by the family $\mathcal{F}$. 
	
	We can then define the following groups.
	
	\begin{itemize}
		\item $\Omega^{spin}_{n}(X)^{(G,\mathcal{F})}$ is the $(G,\mathcal{F})$-equivariant spin bordism group: a cycle here is given by a pair $(M,f)$, where $M$ is a  $n$-dimensional spin $(G,\mathcal{F})$-manifold with a $G$-equivariant reference map $f\colon M\to X$. Two cycles $(M,f)$ and $(M',f')$ are equivalent if there is a spin $(G,\mathcal{F})$-bordism $W$ from $M$ to $M'$ and there exists a $G$-equivariant reference map $F\colon W\to X$ extending $f$ and $f'$.
		\item $Pos^{spin}_{n}(X)^{(G,\mathcal{F})}$ consists of equivalence classes of cycles $(M,f,g)$, where the pair $(M,f)$ is as before and $g$ is a $G$-invariant metric with positive scalar curvature on $M$.  And two cycles $(M,f,g)$ and $(M',f',g')$ are equivalent if there exists a pair $(W,F)$ as before, along with a $G$-invariant metric $g_W$ on $W$ which is of product type near the boundary and restricts to $g$ on $M$ and to $g'$ on $M'$.
		\item $R^{spin}_{n}(X)^{(G,\mathcal{F})}$ is the bordism group of spin $(G,\mathcal{F})$-manifolds with boundary (not necessarily non-empty) of dimension $n$, endowed with a $G$-invariant Riemannian psc-metric on the latter. Bordisms are then manifolds with corners. In particular  $(M,f,g)$ and $(M',f',g')$ are equivalent if there exists a $(G,\mathcal{F})$-bordism $(W,F,\bar{g})$, where $\bar{g}$ is a $G$-invariant psc-metric on  the boundary $(\partial W,\partial F)$, so that it is a bordism between $(\partial M,\partial f,g)$ and $(\partial M',\partial f',g')$ in the sense of $Pos^{spin}_{n-1}(X)^{(G,\mathcal{F})}$.
	\end{itemize}
	
	As before, each of these sets is equipped with an abelian group structure given by disjoint union of manifolds and is covariantly functorial in $X$, meaning that a $G$-equivariant map of $(G,\mathcal{F})$-CW-Complexes $\varphi:X \rightarrow Y$ induces a mapping $\varphi_{*}$ on these groups just by composing it with the reference maps.

	\begin{remark}
		Of course here we assume that all the actions on a manifold with boundary fix the boundary. 
	\end{remark}
	
	\begin{prop}\label{exactg,fstolz}
		The previous abelian groups fit into the following $(G,\mathcal{F})$-equivariant version of the Stolz positive scalar curvature exact sequence:
		
		\[
		\xymatrix{
			\ldots \ar[r] & R^{spin}_{n+1}(X)^{(G,\mathcal{F})} \ar[r] & Pos^{spin}_{n}(X)^{(G,\mathcal{F})} \ar[r] & \Omega^{spin}_{n}(X)^{(G,\mathcal{F})} \ar[r] & R^{spin}_{n}(X)^{(G,\mathcal{F})} \ar[r] & \ldots}
		\]
		where the first map sends  a manifold to its boundary, the second one is the forgetful map (i.e. it forgets about the psc metric) and the last one is the obvious map.
	\end{prop}
	
	Let $X$ be a $(G,\mathcal{F})$-CW-complex with non-trivial fundamental group $\pi_1(X,x)$. Let $\widetilde{X}$ its universal covering, then the proper $(G,\mathcal{F})$-action on $X$ lifts to a $(\widetilde{G}, \widetilde{\mathcal{F}})$-action on $\widetilde{X}$ via the generalized lifting lemma, where: 
	\begin{itemize}
		\item $1\to\pi_1(X,x)\to \widetilde{G}\to G\to 1$ is an extension of discrete groups, which depends on how $G$ moves the base point $x$ chosen to define $\pi_1(X,x)$;
		\item $\widetilde{\mathcal{F}}$ is a family of finite subgroups of $\widetilde{G}$, namely that of the finite isotropy subgroups with respect to the $\tilde{G}$-action, such that its elements are sent (not necessarily surjectively) to elements of $\mathcal{F}$.
	\end{itemize}

\begin{example}
	Consider $S^1\subset \mathbb{R}^2$ whose $\mathbb{Z}_2$ action is given by reflection along the $X$-axis, i.e. $(-1)(x,y)=(x,-y)$. Observe that if $\mathcal{F}=\{e, \mathbb{Z}_2\}$, where $e$ is the identity, \( S^1 \) has the following \( (\mathbb{Z}_2, \mathcal{F}) \)-CW complex structure: its 0-skeleton consists of two equivariant 0-cells of the form \( \mathbb{Z}_2/\mathbb{Z}_2 \times D^0 \), to which an equivariant 1-cell of the form \( \mathbb{Z}_2/\{e\} \times D^1 \) is attached.
	
	Now, $\pi_1(S^1)=\mathbb{Z}$ and the universal cover is realized as $\pi\colon \mathbb{R} \to S^1$ given by $\pi(\theta)=e^{2\pi i \theta}$.
	
	Since $\varphi\colon\mathbb{Z}_2 \to Aut(\mathbb{Z})$ via $\varphi((-1)^j)(m)=(-1)^jm$, then one can form the semidirect product $\mathbb{Z} \rtimes_{\varphi} \mathbb{Z}_2$, whose product law is:
	
	\[\left(k,(-1)^i\right)\cdot\left(m,(-1)^j\right):=\left(k+(-1)^im, (-1)^{i+j}\right).\]
	
	By defining the action of this group (which is an extension of $\mathbb{Z}_2$ via $\mathbb{Z}$) on $\mathbb{R}$ as:
	
	\[\left(k,(-1)\right)\theta:=k-\theta \in \mathbb{R},\]
	then this action covers that of $\mathbb{Z}_2$ on $S^1$, since if $(x,y)=e^{2\pi i \theta}$:
	
	\[\pi(k-\theta)=\pi(-\theta)=e^{-2\pi i \theta}=(x,-y).\]
	
	Note that the action defined on \( \mathbb{R} \) is free, so the family of isotropy subgroups is only given by the trivial group \( \widetilde{\mathcal{F}} = \{e\} \). Therefore, this family maps to \( \mathcal{F} \), but not surjectively.
	
\end{example}

	Then we have the following easy identification. 
	\begin{proposition}\label{extension}
		The $(G,\mathcal{F})$-equivariant Stolz exact sequence associated to $X$ is isomorphic to the $(\widetilde{G}, \widetilde{\mathcal{F}})$-equivariant Stolz exact sequence associated to $\widetilde{X}$.
	\end{proposition}

In particular, the correspondence is given by associating to a cycle whose representative map valued in $X$ is $f$, the cycle represented by the pullback of the universal cover covering $\widetilde{X} \to X$ along the map $f$, namely a total space of a $\pi_1(X,x)$-Galois covering, and lifting the metrics in an equivariant way. The vertical maps are then given by taking quotient with respect to the action of $\pi(X,x)\subset \widetilde{G}$.

We now proceed to state the main result of this section, namely the $(G, \mathcal{F})$-version of Theorem \ref{2connthm} stated before, for which it represents a generalization.

	\begin{theorem}\label{G2conn}
		Let $f\colon X \rightarrow Y$ be a continuous, 2-connected $G$-map between $(G,\mathcal{F})$-CW-complexes which induces an isomorphism between the fundamental groups of $X$ and $Y$. Then the functorially induced map $f_{*}\colon R_{n}^{spin}(X)^{(G,\mathcal{F})} \rightarrow R_{n}^{spin}(Y)^{(G,\mathcal{F})}$ is an isomorphism.
	\end{theorem}
	\begin{proof}
		First of all, observe that Proposition \ref{extension} allows us to assume that $X$ and $Y$ are simply connected, by passing to universal coverings.
		
		%	Furthermore the proof can be discussed, without loss of generality, for connected spaces and $Y=E_{\mathcal{F}}G$ and $f$  the classifying map given by point 2 in Definition \ref{efg}. 
		%	Indeed, for spaces with more than one connected component, one just follows the argument below by considering one component at a time. On the other hand, assume that the theorem is true when $Y=E_{\mathcal{F}}G$ and take a 2-connected $(G,\mathcal{F})$-map $h:X \rightarrow Z$.  There are two classifying maps, unique up to $G$-homotopy, from $X$ and $Z$ to $E_{\mathcal{F}}G$ which we will call $f_{X}$ and $f_{Z}$ respectively. Then, consider the following diagram, which is commutative up to homotopy,
		%	\[
		%	\xymatrix{X\ar[rr]^h\ar[dr]_{f_X}&&Z\ar[dl]^{f_Z}\\
			%	&E_{\mathcal{F}}G&}
		%	\]
		%	by simple considerations, also $f_{Z} \circ h$ is a 2-connected $(G,\mathcal{F})$-map and then induces an isomorphism between the R-groups. By functoriality $(f_{Y} \circ h)_{*}=(f_{Y})_{*} \circ h_{*}$ and, since $f_{Z}$ induces and isomorphism, then also $h$ does. \\
		%	
		\medskip
		
		\textbf{Surjectivity.} Start by showing the surjectivity of the map $f_{*}\colon R_{n}^{spin}(X)^{(G,\mathcal{F})} \rightarrow R_{n}^{spin}(Y)^{(G,\mathcal{F})}$. 
		Let us consider the class $[W,\varphi\colon W \rightarrow Y, g] \in R_{n}^{spin}(Y)^{(G,\mathcal{F})}$, we want to find a \textit{bordant}  cycle whose reference map factors through $f$.

		Consider W as a bordism between its boundary $\partial W$ and the empty set and choose a $G$-invariant Morse function $\alpha\colon W\to \RR$ on it with critical points rearranged as described in \cite[Theorem 4.8]{milnor}, namely for any critical points $p_{i}$ and $p_{j}$ such that $\alpha(p_{i}) < \alpha(p_{j})$, we have that $Ind(p_{i}) < Ind(p_{j})$. Notice that we are going to use the enhanced version of this result to the equivariant setting, see for instance \cite{Mayer, Wassermann}. 
		
		Then there exists a suitable $t\in \RR$ such that the subset $W_1:=\alpha^{-1}([0,t])\subset W$ consists only of $G$-handles of dimensions 0,1 and 2. We immediately obtain a decomposition of $W$ as $W_{1} \cup W_{2}$ such that $W_{1}$ is a bordism from the empty set to $M_{1}:=\alpha^{-1}(t)$ and $W_{2}$ a bordism from $M_{1}$ to $\partial W$. 
		Of course, $W_{2}$ has only critical points $p_{i}$ with $Ind_{\alpha}(p_{i}) \geq 3$. Consider now the function $-\alpha$: this is a Morse function on $W_{2}$ seen as a bordism from $\partial W$ to $M_{1}$ with same critical points $p_{i}$ but with indices now given by $Ind_{-\alpha}(p_{i})=dim(W)-Ind_{\alpha}(p_{i})$. These critical points $p_{i}$ are then associated to $G$-equivariant  $(Ind_{-\alpha}(p_{i})-1)$-surgeries, hence with codimension $Index_{\alpha}(p_{i})+1$ which is $\geq 3$. 
		
		\begin{figure}[h]
			\centering
			\includegraphics[width=
			0.5\textwidth]{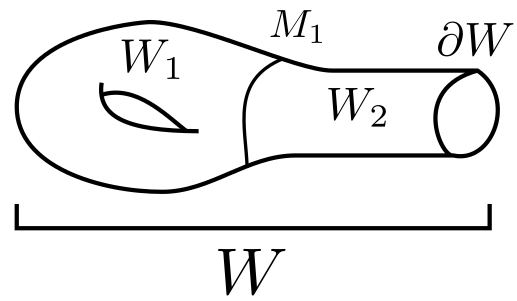}
		\end{figure}
		
		This allows us to apply the Gromov-Lawson Theorem in its $G$-equivariant version as it is proved in \cite{Hanke}. We can then extend the metric with positive scalar curvature $g$ on $\partial W$ to a $G$-invariant metric with positive scalar curvature $\bar{g}$ on $W_2$. Let us denote by $g_{1}$ its restriction to $M_{1}$. 
		Observe that the triad $(W_{1},\varphi_{|W_{1}},g_{1})$ defines a class in $R_{n}^{spin}(Y)^{(G,\mathcal{F})}$ and the manifold $W \times [0,1]$ provides a bordism between $(W_{1},\varphi_{|W_{1}},g_{1})$ and $(W,\varphi, g)$.

		Consider now  the natural $G$-equivariant inclusion $i\colon Y_{(2)}\hookrightarrow Y$, then we have the following facts.
		\begin{itemize}
			\item 	Since the manifold $W_{1}$ is obtained from the empty set by attaching  $(G,\mathcal{F})$-handles of dimension 0,1 and 2, it is homotopy equivalent to a 2-dimensional $(G,\mathcal{F})$-CW-complex. It follows from Theorem \ref{cellular} that the map $\varphi_1:=\varphi_{|W_{1}}$ factors through $i$ up to $G$-homotopy.
			\item Since $f$ is 2-connected,  up to $G$-homotopy we can assume that its restriction to the 2-skeleton  $f_{(2)}\colon  X_{(2)} \rightarrow Y_{(2)}$ has a right inverse, i.e. there exists a $G$-equivariant map $h\colon Y_{(2)} \rightarrow X_{(2)}$ such that $f_{(2)}\circ h=i$.
			To see this, observe that the existence of such a map $h$ is guaranteed, up to $G$-homotopy, by proposition \ref{whitehead}. In fact, since $f_{(2)}$ is $2$-connected and $Y_{(2)}$ has dimension $\leq 2$, it suffices to apply \ref{whitehead} with $X=Y_{(2)}$, $Y=X_{(2)}$, $Z=Y$ and $f=f_{(2)}$ to the map $i \in [X,Z]^G$: the surjectivity of $f_*$ then guarantees the existence of $h$. 
			
		\end{itemize}
		Thus, we obtain the following commutative diagram of $G$-equivariant maps
		\begin{equation}\label{diagram1}
			\xymatrix{W_1 \ar[rr]^{\varphi_1} \ar[d]_{\varphi_1}&& Y &  \\
				Y_{(2)} \ar[d]_h \ar[urr]^i & & \\
				X_{(2)} \ar[uurr]_{f_{(2)}}  \ar[rr]_j  && X\ar[uu]_f}
		\end{equation}
		and, if we set $\psi:= j\circ h\circ \varphi_1\colon W_1\to X$, we obtain  by construction that the following equality holds: $$f_*[W_1,\psi\colon W_1\to X, g_1]=[W,\varphi\colon W\to Y,g]\in R^{spin}_n(Y)^{(G,\mathcal{F})},$$ which proves that  $f_*$ is surjective. \\
		
		\medskip
		
		\noindent \textbf{Injectivity.} In order to prove the injectivity of $f_*\colon R_{n}^{spin}(X)^{(G,\mathcal{F})}\to R_{n}^{spin}(Y)^{(G,\mathcal{F})} $, let us  consider a class $[W,\varphi:W \rightarrow X,g]\in R_{n}^{spin}(X)^{(G,\mathcal{F})}$ such that its image $f_{*}[W,\varphi:W \rightarrow X,g]$ is equal to the trivial element in $R_{n}^{spin}(Y)^{(G,\mathcal{F})}$. This means that there exists:
		\begin{itemize}
			\item a $n+1$-dimensional $(G,\mathcal{F})$-manifold with corners $B$, whose codimension 1 faces are $W$ itself and a bordism $V$ from $\partial W$ to the empty set, which intersect into the only codimension 2 corner  $\partial W=W\cap V$;
			\item a $G$-invariant metric $g_V$ on $V$ of positive scalar curvature of product type near the boundary which restricts to metric $g$ on $\partial W$;
			\item a $G$-equivariant map $\Psi\colon B\to Y$ which restricts to $f\circ\varphi$ on $W$.
		\end{itemize} 
		
		Consider now a $G$-invariant collar neighborhood of $\partial W$ inside $V$ such that the boundary of $B$ is  made of three faces of  codimension 1: $W$ on the bottom, $\partial W \times [0,1]$ vertically and $\bar{V}=V\setminus\partial W \times [0,1)$ on the top.
		
		\begin{figure}[h]\label{B}
			\centering
			\includegraphics[width=
			0.3\textwidth]{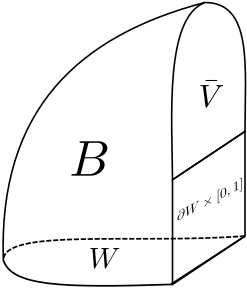}
			
		\end{figure}
		
		We want to split the bordism $B$, as we did to prove the surjectivity, into the composition of two bordisms first from $W$ to a manifold with boundary $W_1$ and then from $W_1$ to $\bar{V}$, such that the first one involves only handle attachments of dimension less or equal than 2 and the second one only of dimension greater or equal than 3.
		Since the vertical boundary face $\partial W \times [0,1]$ is a cylinder, $B$ can be obtained from $W$ by attaching all the handles  to the interior of $W$, away from $\partial W \times [0,1]$. 
		Hence we can find a  Morse function on $B$ which has all critical points there. 
		
		Thus, we can decompose $B$ as desired: $B_{1}$ from $W$ to $W_{1}$ involving only 0,1,2 handle attachments and $B_{2}$ from $W_{1}$ to $\bar{V}$.  We can assume that these two bordisms have vertical boundaries faces equal to $\partial W \times [0,1/2]$ and $\partial W \times [1/2,1]$ respectively and therefore that $W_{1}$ has boundary equal to $\partial W$. 
		
		\begin{figure}[h]\label{B1-B2}
			\centering
			\includegraphics[width=
			0.7\textwidth]{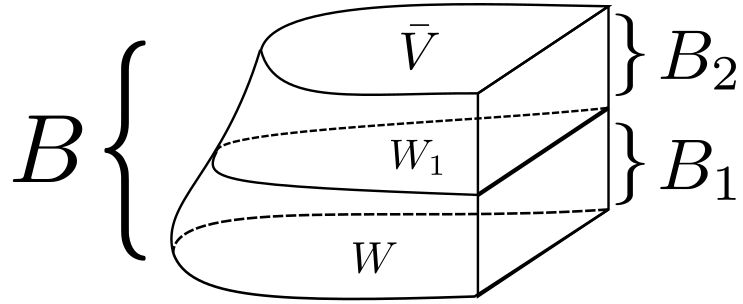}
		\end{figure}
		
		By construction, the bordism $B_{2}$ is the trace of surgeries of codimension $\geq 3$. Therefore, we can apply the Gromov-Lawson Theorem to extend the metric $g_{\bar{V}}$ to a $G$-invariant metric of positive scalar curvature $g_{2}$ on $B_2$. Let us denote by $g_1$ the $G$ -invariant metric of positive scalar curvature obtained by restricting $g_{2}$ to $W_{1}$. 
		
		The last fact to prove is that $\Psi_{|B_1}\colon B_1\to Y$ factors through $f\colon X\to Y$.
		Indeed $B_1$ is obtained form $W$ by attaching, up to homotopy, cells of dimension up to 2, let us call the union of these closed 2-cells $Z$. Up to homotopy, $\Psi_{|B_1}$ maps $Z$ to the 2-skeleton $Y_{(2)}\subset Y$ and,
		as for $W_1$ in \eqref{diagram1}, 
		we obtain a factorization of $\Psi_{|Z}\colon Z\to Y$ as $f\circ \varphi_Z$, by setting $\varphi_Z:= j\circ h\circ \Psi_{|Z}\colon Z\to X$.
		
		Finally, observe that $B_1$ is, up to homotopy, the push-out of the inclusions $W\hookrightarrow B_1$ and $Z\hookrightarrow B_1$ over $\partial Z$. Then, by universality, there exists a unique map $\Phi\colon B_1\to X$ which in particular restricts
		to $\varphi$ on $W$.
		Just by observing that $(B_1, \Phi, g_1)$ is a null-bordism of $(W,\varphi,g)$ in $R^{spin}_n(X)^{(G,\mathcal{F})}$, the injectivity of $f_*$ is proved.

	\end{proof}

	\subsection{A (G,$\mathcal{F}$) universal space}
	
	In this section we are going to define the  \emph{fundamental functor} of a $(G,\mathcal{F})$-CW-complex as a natural generalization of the fundamental group of a CW-complex. This functor will be constructed in such a way that it contains all the information related to the fundamental groups, or rather the fundamental groupoids, of the fixed point spaces of the groups \( H \) belonging to the family \( \mathcal{F} \). To be clearer, let's recall some basic notions about groupoids.
	
	\begin{definition}
		A \textbf{groupoid} is an algebraic structure \( G \rightrightarrows G_0 \), where both \( G \) and \( G_0 \) are sets, and the two arrows represent two surjective maps \( r,s: G \to G_0 \), called respectively the \textit{range} and \textit{source} maps. Furthermore, it is required that:
		
		\begin{itemize}
			\item \( G_0 \) is injectively included in \( G \), and identifying \( G_0 \) with its image, both \( r \) and \( s \) are the identity on \( G_0 \);
			\item defining the set of composable pairs:
			\[ G_2 = \{ (\gamma_1,\gamma_2) \in G \times G | s(\gamma_1) = r(\gamma_2) \}, \]
			on this set there exists a product operation:
			\[ p: G_2 \to G , \quad (\gamma_1,\gamma_2) \mapsto \gamma_1\gamma_2, \]
			such that \( s(\gamma_1\gamma_2) = s(\gamma_2) \) and \( r(\gamma_1\gamma_2) = r(\gamma_1) \);
			\item there is an involution:
			\[ (-)^{-1}: G \to G, \quad \gamma \mapsto \gamma^{-1}, \]
			such that $s(\gamma^{-1})=r(\gamma)$ (implying that $s(\gamma)=r(\gamma^{-1})$). In particular, \( \gamma\gamma^{-1}=r(\gamma) \) and \( \gamma^{-1}\gamma=s(\gamma) \).
		\end{itemize}
	
	Furthermore, it is required that for every \( \gamma \in G \), \( r(\gamma)\gamma = \gamma \) and \( s(\gamma)\gamma = \gamma \), which explains why the set \( G_0 \) is also called the \textit{set of units}, and that the product \( p \) satisfies an associative property, i.e.:
	
	\[ (\gamma_1\gamma_2)\gamma_3 = \gamma_1(\gamma_2\gamma_3), \quad \forall \  (\gamma_1,\gamma_2),(\gamma_2,\gamma_3) \in G_2. \] 
	\end{definition}

\begin{remark}
	A more intuitive way to think about a groupoid is to consider an element \( \gamma \) as an arrow that starts from \( s(\gamma) \) and arrives at \( r(\gamma) \) (which explains the names "range" and "source"). In particular, the product is simply the composition of these arrows, and the inverse is the same arrow but "traversed" in the opposite direction.
	
	From this perspective, it is possible to consider a groupoid as a category, where the objects are given by the elements of \( G_0 \) and the morphisms by the arrows in \( G \).
\end{remark}

From the categorical perspective, a \textit{morphism} of groupoids is simply a functor between them.

\begin{examples} 
		
\begin{enumerate}[label=(\roman*)]
\item Any group \( G \) is a groupoid with \( G_0 = \{ e \} \), where \( e \) is its identity element. 
\item If \( X \) is a \( G \)-space, then we can form the \textit{action groupoid} whose objects are the points of \( X \), and the arrows are elements of \( G \times X \), where \( s(g,x) = x \) and \( r(g,x) = gx \). Clearly, the product is well-defined using the properties of the action of \( G \), thus given by:

\[ (h,y)(g,x) = (hg,x) \quad \text{if} \quad y = gx. \]

\item  The \textit{ fundamental groupoid} of topological space $X$, which constitutes a generalization of the fundamental group, without the choice of a fixed base point.
This is denoted by \( \Pi_1(X) \rightrightarrows X \), where the objects are the points of \( X \), and the elements of \( \Pi_1(X) \) are given by the homotopy classes with fixed endpoints of continuous paths between two points in \( X \). Clearly, if $[\gamma\colon [0,1] \to X]\in \Pi_1(X)$, the source and range maps send $[\gamma]$ to its respective initial and final points, i.e. $\gamma(0)$ and $\gamma(1)$ respectively. The product is the usual composition of paths (wherever this is possible) and the inverse is given by the opposite path, i .e. the path with changed parameter via $t \mapsto 1-t$.
	\begin{remark}
		If $x \in X$, it is clear that:
		
		\[\Pi_1(X)|_x=s^{-1}(x) \cap r^{-1}(x)\simeq \pi_1(X,x),\]
		i.e. it is exactly the fundamental group of $X$ with base point $x$.
	\end{remark}
	
\end{enumerate}
\end{examples}

	Now consider the following \textit{orbit category}: 
	\[
	\mathrm{Orb}_{\mathcal{F}}(G)
	\]
	whose objects are all subgroups in $\mathcal{F}$ and morphisms are sub-conjugations, namely compositions of natural inclusions $H\hookrightarrow K$ as subgroups of $G$ and conjugations 
	$K\rightarrow gKg^{-1}$ by elements $g\in G$.

	\begin{definition}\label{fundamentalfunctor}
		Let $X$ be a  $(G,\mathcal{F})$-CW-complex. The \textbf{fundamental functor} of $X$ is the contravariant functor:
		\[
		\Pi_1(X;G,\mathcal{F})\colon \mathrm{Orb}_{\mathcal{F}}(G)\to Groupoids
		\]
		which associates:
		\begin{itemize}
			\item  to $H\in \mathcal{F}$ the fundamental groupoid of $X^H$ restricted to the $0$-skeleton of  $X^H$, which we denote by $\Pi_1(X^H)_{|X^H_{(0)}}$, 
			\item to a morphism between two objects $H$ and $K$ in  $\mathrm{Orb}_{\mathcal{F}}(G)$ the induced homomorphism of groupoids  between $\Pi_1(X^K)_{|X^K_{(0)}}$ and $\Pi_1(X^H)_{|X^H_{(0)}}$.
		\end{itemize}
	\end{definition}

		We recall the following result \cite[Proposition 3.8]{ttd}, from which it is easy to observe how morphisms in \( \text{Orb}_{\mathcal{F}}(G) \) induce functorially morphisms between the fixed point spaces for the subgroups of the family \( \mathcal{F} \) and consequently between the groupoids of these.
		
		\begin{proposition}	Let $G$ be compact and $H \subset G$ a closed subgroup. If $X$ is a $G$-space, there exists a canonical homeomorphism:
		
		\[X^H \simeq C_G(G/H,X)\]
		\end{proposition}

	The construction of Definition \ref{fundamentalfunctor} is functorial: this means that if $\varphi\colon Y\to X$ is a $G$-equivariant cellular map between $(G,\mathcal{F})$-CW-complexes, then there is an induced natural  transformation $$\varphi_{\#}\colon \Pi_1(Y;G,\mathcal{F})\to \Pi_1(X;G,\mathcal{F})$$ 
	whose component at $H$ is the homomorphism of groupoids 
	$$\varphi_{\#}(H)\colon \Pi_1(Y^H)_{|Y^H_{(0)}}\to \Pi_1(X^H)_{|X^H_{(0)}}$$
	induced by $\varphi_{|Y^H}\colon Y^H\to X^H$.
	
	Now, given a $(G,\mathcal{F})$-CW-complex $X$, consider its $2$-skeleton $X_{(2)}$. Then, for each $H \in \mathcal{F}$ we attach a suitable amount of $G/H$-cells of dimension $k$ to $X^H$ in order to make $\pi_{k-1}(X^H)$ trivial, for all $k \geq 3$. Then, we call the such obtained  $(G,\mathcal{F})$-CW-complex  $B\Pi_1(X;G,\mathcal{F})$. This latter is universal in the following sense. 
	
	\begin{proposition}\label{universalspace}
		For all  $(G,\mathcal{F})$-CW-complex $Y$ and for all natural transformations $$\Phi\colon \Pi_1(Y;G,\mathcal{F})\to \Pi_1(X;G,\mathcal{F}),$$ there exists, unique up to $G$-equivariant homotopy, a $G$-equivariant cellular map $$\varphi\colon Y\to B \Pi_1(X;G,\mathcal{F})$$ such that 
		$\varphi_{\#}=\Phi$.
	\end{proposition}

	\begin{proof}
		We begin by defining the map $\varphi$ on the $G$-equivariant $0$-skeleton of $Y$ by putting 
		$$\varphi_{|Y_{(0)}}:=\Phi(\{e\})_{|Y_{(0)}}.$$ 
		Now, we proceed to define $\varphi$ on the $1$-skeleton of $Y$. Let us fix $H \in \mathcal{F}$, then a cell $c^1$ of dimension $1$ in $Y^H$ defines an element $\gamma \in \Pi_1(Y^H)_{|Y^H_{(0)}}$: set $\varphi(c^1)$ as the unique cell of dimension $1$ in $X^H$ which represents $\Phi(H)(\gamma) \in \Pi_1(X^H)_{|X^H_{(0)}}$ and set $\varphi$ on the $G$-orbit of $c^1$ in such a way that if $g \cdot \gamma$ represents $g\cdot c^1$ in $\Pi_1(X^{gHg^{-1}})_{|X^{gHg^{-1}}_{(0)}}$, then $\varphi(g\cdot c^1)$ is the unique cell representing $\Phi(gHg^{-1})(g \cdot \gamma) \in \Pi_1(X^{gHg^{-1}})_{|X^{gHg^{-1}}_{(0)}}$. 
		In this way we have that $\varphi$ is $G$-equivariant, indeed $g\cdot \varphi(c^1)=\varphi(g\cdot c^1)$, because $\Phi(gHg^{-1})(g\cdot \gamma)=g \cdot \Phi(H)(\gamma) $.
		Observe that since $\Phi$ is a natural transformation, we can do it for every $H \in \mathcal{F}$ in a compatible way. 
		
		Now, let us take a cell $c^2$ of dimension $2$ in $Y^H$. Its attaching map $\psi \colon S^1 \to Y^{H}_{(1)}$ defines a contractible loop and then a unit in the $\Pi_1(Y^H)_{|Y^H_{(0)}}$. Since $\Phi(H)$ is a homomorphism of groupoids, then $\varphi \circ \psi$ has to represent a unit as well in $\Pi_1(X^H)_{|X^H_{(0)}}$, namely there exists a $2$-cell with attaching map $\varphi \circ \psi$ which we define to be $\varphi(c^2)$.

		Finally, by higher contractibility of each fixed point subsets of $B \Pi_1(X;G,\mathcal{F})$, we can extend $\varphi$ from $Y_{(2)}$ to $Y$, in a $G$-equivariant way.
		
		It follows by construction to check that $\varphi_{\#}=\Phi$. Let us check that if there exists an other $G$-equivariant map $\varphi'\colon Y \to B \Pi_1(X;G,\mathcal{F})$ such that $\varphi'_{\#}=\Phi$, then it is $G$-homotopic to $\varphi$.
		It is immediate to check that if $\varphi'_{\#}=\varphi_{\#}$, then their restriction to $Y_{(0)}$ are equal and that their restriction to $Y_{(1)}$ are $G$-homotopic through $\varphi_{(1)}^t$ with $t\in[0,1]$.
		Then we have a $G$-equivariant map 
		\[
		\varphi\cup\varphi'\cup \varphi_{(1)}^t\colon Y\times\{0\}\cup Y\times\{1\}\cup Y_{(1)}\times (0,1)\to B \Pi_1(X;G,\mathcal{F})
		\]
		which we can extend to a $G$-homotopy
		\[
		\varphi^t\colon  Y\times [0,1]\to  B \Pi_1(X;G,\mathcal{F})
		\]
		always because all the fixed point subsets of $B \Pi_1(X;G,\mathcal{F})$ have trivial homotopy groups in dimensions bigger than 1.
	\end{proof}

	\begin{remark}
		Observe that when $G=\{e\}$ is the trivial group, then of course the family $\mathcal{F}$ is trivial and $X$ is a standard $CW$-complex. In particular, if $X$ is connected, the above construction reduces to attach cells to its $2$-skeleton $X_{(2)}$ in order to make its homotopy groups $\pi_k(X)$  trivial for $k \geq 2$. It follows that the space obtained in this way is an Eilenberg-Mac Lane space $K(\pi_1(X),1)$. Since the homotopy type of a $K(\pi_1(X),1)$ $CW$-complex depends only on the group $\pi_1(X)$, it follows that $B \Pi_1(X;G,\mathcal{F})$ is exactly the classifying space $B\pi_1(X)$ up to homotopy (see \cite[Theorem 1B.8.]{hatcher}).
	\end{remark}

Now we combine the results obtained in this section and Theorem \ref{G2conn}, obtaining as a corollary the fact that the Stolz $(G, \mathcal{F})$-equivariant $R$-groups depend only on the isomorphism class of the fundamental functor.

Recall that an isomorphism between two functors $F, G: C \to D$ is a \textit{natural transformation} $\eta$ with a two-sided inverse or, equivalently, a natural transformation such that:

\[\forall c \in \text{Obj}(C),\quad \eta(c): F(c) \to G(c) \ \  \text{is an isomorphism in D.}\]

\begin{corollary}\label{univ2conn}
	Let $Y$ be a $(G, \mathcal{F})$-CW complex such that there exists an isomorphism:

\[\Phi\colon \Pi_1(Y;G,\mathcal{F})\to \Pi_1(X;G,\mathcal{F}),\]
meaning that its fundamental functor is isomorphic to that of $X$.

Then the $G$-equivariant cellular map  $\varphi\colon Y\to B \Pi_1(X;G,\mathcal{F})$ of Proposition \ref{universalspace} is $2$ connected. Moreover, $\varphi$ functorially induces the following isomorphism between the Stolz groups:
 
 \[\varphi_* \colon R_{n}^{spin}(Y)^{(G,\mathcal{F})} \rightarrow R_{n}^{spin}(B \Pi_1(X;G,\mathcal{F}))^{(G,\mathcal{F})}\]
 
\end{corollary}

	\newpage
	
	\section{Smoothly Stratified Spaces}\label{smoothlystrspace}
	
	Let X be a locally compact, second countable, metrizable topological space. 
	\begin{definition} 
		A locally finite family $\mathcal{S}=\{Y_{\alpha}\}$ is a \textbf{stratification} of X if each $Y_{\alpha}$ is a locally closed (i.e. it is open in its closure), smooth manifold without boundary and X decomposes as $X=\bigcup_{\alpha} Y_{\alpha}$. 
	\end{definition}	

		Observe that the dimension of each $Y_\alpha$ is not fixed, but depends on the index $\alpha$.
		In essence, a stratified space will be given by a topological space, its stratification, and compatibility conditions that govern how the stratification is realized. In particular, these conditions are expressed by the use of tubular neighborhoods around each stratum. 
		
		Let's introduce the following \textit{control data}. Given a stratification $\mathcal{S}=\{Y_\alpha\}$ consider a family of tubular neighborhoods $\{T_{\alpha}, \pi_{\alpha}, \rho_{\alpha}\}$, where for each $\alpha$:
	
	\begin{itemize}
		\item $T_{\alpha}$ is an open neighborhood of $Y_{\alpha}$ in X;
		\item $\pi_{\alpha}: T_{\alpha} \rightarrow Y_{\alpha}$ is a continuous retraction;
		\item $\rho_{\alpha}: T_{\alpha} \rightarrow \mathbb{R}_{\geq 0}$ is a continuous map such that $\rho_{\alpha}^{-1}(0) = Y_{\alpha}$.
	
	\end{itemize}

\begin{definition} \label{stratifiedspace}
	Let $X$ be as above. Then X is a \textbf{stratified space} if there exist a stratification $\mathcal{S}=\{Y_{\alpha}\}$ together with a family $\{T_{\alpha}, \pi_{\alpha}, \rho_{\alpha}\}$ satisfying the following:
	
\begin{enumerate}[label=(\roman*)]
	\item  if $Y_{\alpha}, Y_{\beta} \in \mathcal{S}$ are such that $T_{\alpha} \cap Y_{\beta} \neq \emptyset$, then $Y_{\alpha} \subset \overline{Y_{\beta}}$ and we will indicate it as $Y_{\alpha} < Y_{\beta}$, providing a partial order on $\mathcal{S}$;
	\item if $Y_{\alpha} < Y_{\beta}$, then the map:
	\[(\pi_{\alpha},\rho_{\alpha}): T_{\alpha} \cap Y_{\beta} \longrightarrow Y_{\alpha} \times \mathbb{R}_{\geq 0}\]
	is a proper differentiable submersion\footnote{Observe that by Ehresmann's theorem, this guarantees that $(\pi_\alpha,\rho_\alpha)$ is a locally trivial fibration.}.
	\item for each pair of strata $Y_{\alpha} \subset \overline{Y_{\beta}}$ and all $x \in T_{\alpha} \cap T_{\beta}$ such that $\pi_{\beta}(x) \in T_{\alpha} \cap Y_{\beta}$, then:
	\begin{enumerate}
		\item $\pi_{\alpha}\pi_{\beta}(x)=\pi_{\alpha}(x)$;
		\item $\rho_{\alpha}\pi_{\beta}(x)=\rho_{\alpha}(x)$.
	\end{enumerate}
\end{enumerate}
\end{definition}

By \textit{dimension} of the stratified space $X$, we mean the maximal dimension of all of its strata.

\begin{definition}
The \textbf{depth} of a stratum $Y_{\alpha} \in \mathcal{S}$, indicated $d(Y_{\alpha})$, is the maximal length of all ascending chains (with respect to the partial order $<$), with $Y_{\alpha}$ as first element, i.e.:
\[d(Y_{\alpha})=sup\{n \colon Y_{\alpha} < Y_{1} < Y_2 < \ldots < Y_n \}.\]

In particular, the depth of a stratified space $X$ is the maximal depth of any stratum.
\end{definition}

Clearly, a stratified space of depth $0$ is simply a smooth manifold without boundary.

As a consequence of Thom's First Isotopy Lemma (\cite[Theorem 2.6]{Verona}), the retraction $\pi_{\alpha}: T_{\alpha} \rightarrow Y_{\alpha}$ is a locally trivial fibration with fibre the cone $C(L_{\alpha})$ (i.e. the product $L_\alpha \times [0,1]$ with $L_\alpha \times \{0\}$ collapsed to a point which be the vertex of the cone) over some stratified space $L_{\alpha}$ of depth $d(Y_{\alpha})-1$. This means that each point $p \in Y_{\alpha}$ has a neighborhood in X homeomorphic via a strata-preserving homeomorphism to $B_{\epsilon}^{dim(Y_{\alpha})}(p) \times C(L_{\alpha})$ and $L_{\alpha} \simeq (\pi_{\alpha}, \rho_{\alpha})^{-1}(p,\epsilon)$. Observe that $\rho_{\alpha}$ can be identified with the radial coordinate along these cones. From now on, we refer to $L_\alpha$ as the \textit{link} of the stratum $Y_\alpha$.

 \begin{figure}[h]\label{hhhhh}
 	\centering
 	\includegraphics[width=
 	0.6\textwidth]{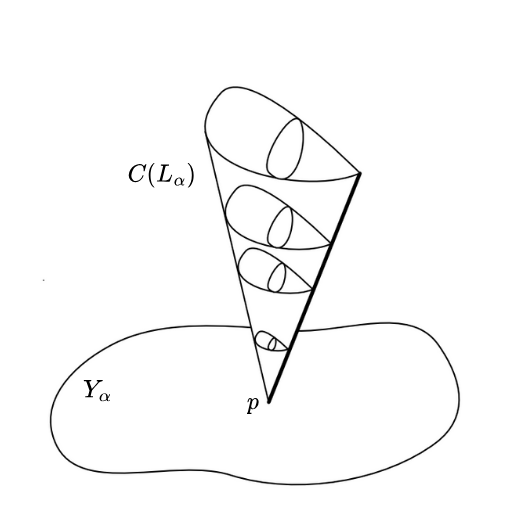}
 	
 \end{figure}

 \begin{definition}
 	Let $X$, $X'$ be two stratified spaces and $F: X \rightarrow X'$ be a continuous map. $F$ is a \textbf{weak morphism} if:
 	
 	\begin{itemize} 
 		\item for any stratum $Y_{\alpha}$ of $X$, there is a stratum $Y'_{\alpha'}$ such that $F(Y_{\alpha})\subseteq Y'_{\alpha'}$;
 		\item for each $\alpha$, $F|_{Y_{\alpha}}$ is smooth;
 		
 		\item $F$ is compatible with the control data in the following sense: 
 		
 		\[ F(T_{\alpha}) \subseteq T_{\alpha'}, \quad  \pi_{\alpha'} \circ F = F \circ \pi_{\alpha}. \]
 	\end{itemize}
 
 		$F$ is a \textbf{morphism} if, in addition, we have:
 		
 		\[\rho_{\alpha'} \circ F = \rho_{\alpha}.\]

 		Moreover, if $F$ maps the open strata of $X$ diffeomorphically to those of $X'$, we say that $F$ is a \textbf{stratified isomorphism}.

 \end{definition}
 
 Observe that in the definition of morphism is not required for the spaces $X$, $X'$ to have same depth or dimension.
 
In order to work in a smooth category, we require that all the fibrations $\pi_{\alpha}$ have transition functions which are stratified isomorphisms. We then obtain the following definition.

\begin{definition}
	Let X be a stratified space as before. We say that X is a \textbf{smoothly stratified space} (or a \textit{stratified pseudomanifold}) if the following hold:
	
	\begin{enumerate}[label=(\roman*)]
		\item for each stratum $Y_{\alpha} \in \mathcal{S}$, the fibration $\pi_{\alpha}$ has transition functions which are stratified isomorphisms, which then preserve the radial variable $\rho_{\alpha}$;
		\item denoting by $X_{k}$ the union of all strata of dimension less or equal than $k$, and by $n$ the largest dimension of all strata, then we get the following filtration:
		
		\[X=X_n \supseteq X_{n-1}=X_{n-2} \supseteq X_{n-3} \supseteq \ldots \supseteq X_{0}, \]
		 and we require that $X_{n} \setminus X_{n-2}$ is an open smooth manifold dense in X.
	\end{enumerate}
\end{definition}

It is common to indicate the union of all strata of depth greater than 0 and its complement as the singular set and regular set respectively of the pseudomanifold, i.e.:

\[sing(X)=\bigcup_{d(Y_{\alpha}) >0} Y_{\alpha},\ \ \ reg(X)=X \setminus sing(X)\]

\begin{remark}
If \( X \) is a stratified space of depth \( k \), then by constructing the cone \( C(X) \) over \( X \), we obtain a stratified space of depth \( k+1 \). In particular, its stratification will be given by that of \( X \) plus the vertex of the cone \( C(X) \), represented by \( X \times \{0\} \), which will be the only stratum at depth \( k+1 \).	
\end{remark}

Defining bordism groups of smoothly stratified space will require the notion of a stratified space with boundary.

\begin{definition}
	By a \textbf{smoothy stratified space with boundary} of dimension $n$ we mean a pair $(X, \partial X)$ such that:
	\begin{enumerate}[label=(\roman*)]
		\item $\partial X$ is a smoothly stratified space of dimension $n-1$;
		\item $X$ satisfies all the requests in the definition of a smoothly stratified space except that it is required that $X_{n} \setminus (X_{n-2} \cup \partial X)$ to be an open smooth, oriented manifold dense in X;
		\item $\partial X$ has a collar neighborhood, meaning that there exists $\mathcal{U} \subset X$ closed with an orientation-preserving, stratified isomorhism $\mathcal{U} \simeq \partial X \times [0,1]$;
		\item if $\mathcal{S}=\{Y_{\alpha}\}$ are the strata of X, then $\{Y_{\alpha} \cap \partial X\}$ are the strata of $\partial X$.
	\end{enumerate}
\end{definition}

An important aspect regarding smoothly stratified spaces is that singularities, which are given by strata of depth $> 0$, can be \textit{resolved}. Specifically, there is an equivalence between the class of smoothly stratified spaces and the class of manifolds with corners with so-called iterated fibration structures. This correspondence, primarily introduced by Melrose and well developed in \cite[Section 2.2, 2.3]{ALMP}, is made explicit through two processes called blowup and blowdown. The transition from a stratified space to the corresponding manifold with corners is useful for applying methods from geometric microlocal analysis.

Recall that a manifold with corners \( X \) of dimension \( n \) is such that every point \( p \in X \) has a neighborhood \( \mathcal{U} \ni p \) diffeomorphic to \( \mathbb{R}_+^l \times \mathbb{R}^{n-l} \), for some \( l \leq n \). In the case where \( l \leq 1 \), we obviously have a manifold with boundary.
Observe that there exists a decomposition of such a space into its interior and the union of all its boundary components of different codimensions \( l \). In particular, all these faces can be obtained by intersecting those of codimension \( 1 \).

Roughly speaking, a manifold is said to have an iterated fibration structure if each of its codimension \( 1 \) faces is equipped with a fibration (in a suitable sense) in which both the base and the fiber are themselves manifolds with corners. Moreover, certain compatibility conditions of these fibrations on the intersections of the faces and how they restrict there are required (for more details, we refer again to \cite[Section 2.2]{ALMP}). 

Now, Propositions 2.3 and 2.5 govern the correspondence discussed above. Starting from a compact manifold with corners with an iterated fibration structure \( \widetilde{X} \), we obtain a smoothly stratified space \( \widehat{X} \) through a process of blowing down, or collapsing, performed for each connected component of the fibers of each boundary hypersurface, resulting in a \textit{blowdown} map \( \beta: \widetilde{X} \to \widehat{X} \). Conversely, given a smoothly stratified space \( \widehat{X} \), we obtain a manifold with corners with an iterated fibration structure \( \widetilde{X} \), called the \textit{resolution} of $\widehat{X}$, and a map \( \beta: \widetilde{X} \to \widehat{X} \) through a \textit{blowup} process. In particular, this process ensures (see \cite[Proposition 2.5]{ALMP}):
\begin{itemize}
\item the existence of a bijective correspondence between the strata of \( \widehat{X} \) and the boundary hypersurfaces of \( \widetilde{X} \);
\item the interior of \( \widetilde{X} \) and the regular set of \( \widehat{X} \) are diffeomorphic via \( \beta \);
\item \( \beta \) is a smooth fibration over the interior of each boundary hypersurface, whose base is the corresponding stratum \( Y \) and the fiber is the regular part of the link of \( Y \) in \( \widehat{X} \).
\end{itemize}

As an example, which will be the case of our interest, consider a manifold \( \widetilde{X} \) with boundary \( \partial \widetilde{X} \) as the total space of a fibration with fiber \( F \) and base space \( Y \), both closed manifolds. Consider a collar neighborhood of the boundary \( \partial \widetilde{X} \times [0,1) \), which is equipped with a retraction onto \( \partial \widetilde{X} \) and a fibration \( \partial \widetilde{X} \times [0,1) \to Y \), whose fiber is \( F \times [0,2) \). If at this point we collapse \( F \times \{0\} \) to a point, we obtain a fibration over \( Y \) with fiber given by the cone \( C(F) \). The space and the quotient map thus obtained are then the blowdown $\widehat{X}$ of \( \widetilde{X} \) and the blowdown map respectively. The image of the collar neighborhood through this map realizes the tubular neighborhood of \( Y \), which will be the singular stratum of depth \( 1 \) of the resulting stratified space.
Conversely, given a smoothly stratified space $\widehat{X}$ of depth $1$, with depth-$1$ stratum $Y$ and control data given by $\{T,\pi,\rho\}$, one simply consider the manifold with boundary $\widetilde{X}=\widehat{X} \setminus \rho^{-1}([0,1))$. The restriction of $\pi$ to gives the desired fibration on the boundary face $\rho^{-1}(1)$ with fiber diffeomorphic to the link of the stratum $Y$.

\subsection{Wedge metrics}\label{secwedge}

Consider a manifold with boundary \( M \) such that its boundary is the total space of a fibration \( \pi: \partial M \to Y \). We have discussed that this can be an example of a space obtained by blowup from a smoothly stratified space \( X \) of depth 1. In particular, the fiber of \( \pi \) will be given by the link \( L \) of the depth-$1$ stratum \( Y \) of $X$.
If the intention is to study differential forms of \( X \), while working on \( M \), indicating by $\iota_\partial \colon \partial M \to M$ the natural inclusion of the boundary, one considers the set:

\begin{equation}\label{wedgesections}
	\{\omega \in \Gamma(T^*M) \,|\, \iota_\partial^*\omega \in \pi^*(\Gamma(T^*Y))\},
\end{equation}
that is, differential 1-forms such that, when restricted to the boundary, they originate from forms defined on the stratum \( Y \).

\begin{definition}
	We call the \textbf{wedge cotangent bundle} $\leftindex^{w}T^*M$ the vector bundle over \( M \) defined through the Serre-Swan theorem such that its sections are given by the finitely generated projective \( C^\infty(M) \)-module as shown in (\ref{wedgesections}).
\end{definition}

Given a point \( p \in \partial M \), it is always possible to find a coordinate system \( (r,y,l) \) around \( p \) such that \( r \) represents a normal coordinate to the boundary, and \( (y,l) \) are coordinates on \( \mathcal{V} \times L \subset Y \times L \), where \( \mathcal{V} \) is a trivialization domain of the fibration \( \pi \) around the point \( \pi(p) \). More precisely, $r$ is a smooth non negative function on $M$ such that \( r=0 \) on \( |\text{d}r| \neq 0 \) on \( \partial M \) and  is often also called a \textit{boundary defining function} on M. In these coordinates, $\leftindex^{w}T^*M$ has a local basic of 1-forms given by:

\[\{dr,rdl_j, dy_k\}, \quad j=1,\ldots, dim(L), \quad k=1, \ldots, dim(Y).\]

Observe that the basis elements $rdl_j$ vanish as sections of $T^*M$, but not as elements of $\Gamma(\leftindex^{w}T^*M)$.
Moreover, the map which sends a section of the wedge cotangent bundle to the exact 1-form in $T^*M$ is an isomorphism on the interior of $M$. 

By duality we can obtain the \textit{wedge tangent bundle} \(\leftindex^{w}TM\) as a vector bundle over \(M\). This, using the same coordinate system as above, is locally spanned by the following vector fields:

\begin{equation}\label{wedgeframe}
\{\partial_r, \frac{1}{r}\partial_{l_j}, \partial_{y_k}\}, \quad j=1, \ldots, \text{dim}(L), \quad k=1,\ldots, \text{dim}(Y).
\end{equation}

Note that even in this case, there exists an isomorphism between the wedge tangent bundle and $TM$ over the interior of $M$, and that the frame (\ref{wedgeframe}) is globally well defined when viewed as a frame in $\leftindex^{w}T^*M$ over $M$.

Now let's proceed to introduce a new class of vector fields on M. Consider the following set:

\begin{equation}\label{edgesections}
	\{V \in \Gamma(TM) \,|\, V|_{\partial M} \text{ is tangent to the fibers of } \pi\} 
\end{equation}	

In particular, if we introduce the following subspace of smooth functions on M, corresponding to those functions that are continuous on the corresponding stratified space X:

\[ C^\infty_b(M) := \{f \in C^\infty(M) \,|\, f|_{\partial M} \in \pi^*C^\infty(Y)\}, \]
then (\ref{edgesections}) corresponds exactly to those vector fields that, when applied to elements of \( C^\infty_b(M) \), produce functions that vanish on the boundary.

Similarly as before, using the Serre-Swan theorem, we can define the \textit{edge tangent bundle} $\leftindex^{e}TM$ as the vector bundle over M such that its sections are given by (\ref{edgesections}). Its local basis is then given by:

\[ \{r\partial_r, \partial_{l_j}, r\partial_{y_k}\}, \quad j=1, \ldots, \text{dim}(L), \quad k=1,\ldots, \text{dim}(Y). \]

The universal enveloping algebra of (\ref{edgesections}) is the ring of edge differential operators (see \cite{Mazzeo}), denoted as $Diff_e^*(M)$, composed of all operators that can be locally written as polynomials in elements of (\ref{edgesections}).
This means that if we consider two vector bundles \(E\) and \(F\) over \(M\), then, in the usual local coordinate system near the boundary, an edge differential operator of order \(k\) in \(Diff_e^k(M; E, F)\) is locally written as:

\[P = \sum_{|\alpha| + |\gamma| + j \leq k} a_{j,\alpha,\gamma}(r,y,l) (r\partial_r)^j(r\partial_y)^\alpha(\partial_l)^\gamma,\]
where each $a_{j,\alpha,\gamma}$ denotes a local section of $Hom(E,F)$.

\begin{definition}
  A differential operator $P\colon \Gamma(E) \to \Gamma(F)$ of order $k$ is a \textbf{wedge differential operator} if there exists an operator $P' \in Diff^k_e(M;E,F)$ such that:
  
  \[P=r^{-k}P',\]
  where $r$ is a boundary defining function. The set of wedge differential operators between $E$ and $F$ is then called $Diff^k_w(M;E,F)$.
\end{definition}

Similarly to what happens for differential operators, it is possible to define a principal symbol for this class of operators. In particular, the \textit{principal wedge symbol} of \(P \in Diff^k_w(M)\) is defined as a smooth section of the bundle \(p^*\left(Hom(E,F)\right) \to \leftindex^{w}T^*M\), where \(p^*(Hom(E,F))\) is the pullback bundle via the projection map of the wedge cotangent bundle \(p\colon \leftindex^{w}T^*M \to M\). Locally, this means:

\[ \sigma_w(P)(r,y,l,\xi) = \sum_{|\alpha| + |\gamma| + j = k} a_{j,\alpha,\gamma}(r,y,l) (\xi_r)^j(\xi_y)^\alpha(\xi_l)^\gamma. \]

Here, \( \xi = (\xi_r, \xi_y, \xi_l) \) are coordinates on the fiber of the wedge cotangent bundle \( \leftindex^{w}T^*M \). Such operator is said to be \textit{elliptic} if its principal wedge symbol is invertible whenever $\xi \neq 0$.

Keeping in mind the example just described for depth 1, we will now introduce a special class of metrics, which will turn out to be simply metrics on the wedge tangent bundle.

Let X be an smoothly stratified space of arbitrary depth. A Riemannian metric on X is, by definition, a metric on its regular set, $reg(X)$. To define this special class of metrics, it is observed that it is possible to construct an open covering of $reg(X)$ by recalling that each point \( p_\alpha \in Y_\alpha \) admits a neighborhood \( \mathcal{U}_\alpha \) such that \( \pi_\alpha^{-1}(\mathcal{U}_\alpha)\simeq\mathcal{U}_\alpha \times c(L_\alpha) \). Since $L_\alpha$ is again a smoothly stratified space of depth equal to $d(Y_\alpha)-1$, then a point $p_{\alpha_1}$ of a stratum $Y_{\alpha_1}$ of $L_\alpha$ has again a neighborhood such that $\pi^{-1}_{\alpha_1}(\mathcal{U}_{\alpha_1}) \simeq\mathcal{U}_{\alpha_1} \times C(L_{\alpha_1})$. 
Continuining in this way, one can obtain open sets of the form:

\[\mathcal{U}_\alpha \times C(\mathcal{U}_{\alpha_1} \times C(\mathcal{U}_{\alpha_2} \times \ldots \times C(\mathcal{U}_s)),\]

for $s \leq d(Y_\alpha)$. Then one can choose coordinates $y_{\alpha_i}$ on each $\mathcal{U}_{\alpha_i}$ and radial coordinates $r_{\alpha_i}$ on the cones $C(L_{\alpha_i})$. Of course, these constitute a covering of $X$ and of $reg(X)$.

\begin{definition}
	A Riemannian metric \( g \) on X is called a \textbf{wedge metric} if there exists an open covering of the form just introduced such that \( g \) in each of such sets:
	
	\[g=dr_\alpha^2+h_\alpha+r^2_\alpha(dr_{\alpha_2}^2+h_{\alpha_2}+r^2_{\alpha_2}(dr_{\alpha_3}^2+h_{\alpha_3}+r^2_{\alpha_3}(\ldots+r^2_{s-1}h_{\alpha_s})),\]
	
	where $h_{\alpha_j}$ are metrics on $\mathcal{U}_{\alpha_j}$.
	Moreover, it is required that the tangent spaces of each stratum $Y_{\alpha_j}$ are lifted horizontally as subbundles of the tangent bundles of the total spaces of the cone bundles (for example, by choosing some connections). This implies that the metrics $h_{\alpha_j}$ depends only on $y_{\alpha_i}$, for $i \leq j$.
	
\end{definition}

\begin{proposition}[\cite{ALMP}, Proposition 3.1, 3.2]
	Let $X$ be a smoothly stratified space, then:
	
	\begin{itemize}
		\item there always exists a wedge metric $g$ on $X$;
		\item any two wedge metrics $g$, $g'$ on $X$ are homotopic within the class of wedge metrics.
	\end{itemize}
\end{proposition}

Let's now consider again a stratified pseudomanifold $X$ of depth-$1$. Recall that this consists of a space $X$ with a depth-$0$ stratum $X_{reg}$ (dense in X) and a depth-$1$ stratum $Y$. Moreover, supposed be fixed a closed manifold $L$ (a depth-0 stratified space) as the link of the stratum $Y$. We will indicate as $T_Y$ the tubular neighborhood associated to the stratum $Y$ and $\pi: T_Y \rightarrow Y$, $\rho: T_Y \rightarrow \mathbb{R}_{\geq 0}$ the associated mappings. Once $L$ is fixed, we refer to $X$ as a \textit{pseudomanifold with fibered $L$-singularities}.

To such a space there is associated its \textit{resolution}, meaning a smooth manifold with boundary $X_{r} :=X \setminus \rho^{-1}([0,1))$ obtained by the blowup process described before: observe that there is an obvious diffeomorphism between its interior and the smooth stratum $X_{reg}$. By construction, the boundary $\partial X_{r} := \rho^{-1}(1)$ is endowed with a fibration $\pi_{r} \colon \partial X_{r} \rightarrow Y$, which is the restriction of the fibration $\pi$ and whose fibers are diffeomorphic to the link $L$.

Now we describe more precisely wedge metrics in this setting.
Fix a Riemannian metric on both $\partial X_{r}$ and $Y$, which are smooth manifolds, and call them $g_{\partial X_{r}}$ and $g_{Y}$ respectively. We ask that the fibration $\pi_{r}: (\partial X_{r}, g_{\partial X_{r}}) \rightarrow (Y, g_{Y})$ is a Riemannian submersion, i.e. $(\pi_{r})_{*} \colon (Ker(\pi_{r*}))^{\perp} \rightarrow T(Y)$ is an isometric isomorphism, meaning that we have chosen a connection on the tangent bundle $T(\partial X_{r})$ inducing a splitting $T(\partial X_{r})=T_{H}(\partial X_{r}) \oplus T_{V}(\partial X_{r})$, i.e. in its horizontal and vertical part. In particular, we obtain the following identifications for any $y \in Y$:
 \[T_{H}(\partial X_{r}) \simeq (\pi_{r})^*T(Y), \quad  T_{V}(\partial X_{r})|_{\pi_r^{-1}(y)} \simeq TL.\]

 Then, consider a Riemannian metric on $\partial X_{r}$ of the following form:

\[
	g_{\partial X_{r}}=(\pi_{r})^*g_{Y} \oplus g_{\partial X_r / Y},
\]
where $g_{\partial X_r / Y}$ means a metric on the vertical tangent bundle of $\pi_r$, which restricts to a Riemannian metric of $L$ on each fiber.
Similarly, by choosing a connection for the restriction to $T_Y \setminus Y$ of the bundle $\pi$, which therefore becomes a Riemannian submersion, with open cones as fibers, we have that a wedge metric on $X$ is a Riemannian metric on $X_{reg}$ such that on the tubular neighborhood $T_Y$ the metric locally takes the form:

\[dr^2+\pi^*(g_{Y})+r^2g_{\partial X_r / Y}+O(r),\]
where $r$ stays for the radial coordinate on the cone and $O(r)$ indicates a term vanishing as $r \rightarrow 0$ (i.e. while approaching the singular stratum $Y$). Again, $g_{\partial X_r / Y}$ indicated a metric on the vertical tangent bundle.

\begin{remark}\label{wedgemetricresolution}
	Note that, by construction, the regular set of the stratified space \(X\) is diffeomorphic to the interior of its resolution \(X_r\). We can therefore use this diffeomorphism to induce a metric on the interior of \(X_r\), suitably renaming the coordinate \(r\) so that it corresponds to the boundary defining function for \(\partial X_r\). In this way, we obtain a metric that extends, in a non-degenerate manner, to a metric on the wedge tangent bundle $\leftindex^{w}TX_r \to X_r$. Observe that this metric will have a different behaviour around the boundary $\partial X_r$. In particular, note that (\ref{wedgeframe}) constitutes an orthonormal frame for it.
	
	The metric obtained is clearly incomplete: however, for a wedge metric $g$, one can always associate a metric:
	
	\[\widetilde{g}=r^{-2}g,\]
	which is an example of an edge metric and, in particular, turns out to be a complete metric.
\end{remark}

We say that a wedge metric is of positive scalar curvature if it has positive scalar curvature as a Riemannian metric on the regular smooth stratum $X_{reg}$ or, equivalently, using what said in the above remark, on the interior of the resolution $X_r$.

The above definition extends similarly to the case of a smoothly stratified space with boundary $(X, \partial X)$. In that case, one requires that on a collar neighborhood $\mathcal{U}=\partial X \times [0,1]$ the metric is of the product-like form $g=dx^2+g_{\partial X
}$, where $x$ is the obvious normal coordinate to the boundary and $g_{\partial X}$ is an wedge metric of the stratified space $\partial X$. \\

\subsection{The L-bordism groups}\label{secbordismgroups}

Now we will proceed to introduce some bordism groups for smoothly stratified spaces of depth 1 endowed with wedge metrics. These, in particular, will play a key role in the next section where the Stolz sequence in this setting will be introduced (recall section \ref{stolzsequence}).

First of all, we will need to introduce the concept of a spin pseudomanifold. Let \(M\) be a smoothly stratified space of depth 1, and denote, changing a bit the notations used before, its regular and singular strata by \(M_{\text{reg}}\) and \(\beta M\), respectively. Furthermore, assume that \(g\) is a wedge metric on \(M\).

\begin{definition}
	We say that $(M,g)$ is a \textbf{spin stratified} (or simply a \textit{spin pseudomanifold}) if both its resolution $M_r$ and the stratum $\beta M$ are spin. 
\end{definition}

\begin{remark}
	Fixing a spin structure on the resolution $M_r$ fixes a a spin structure on $\partial M_r$ too. Moreover, fixed a spin structure on the stratum $\beta M$ fixes a spin structure on the vertical tangent bundle of the fibration $\partial M_r \to \beta M$ (see \cite[Proposition 1.15]{LM}).
\end{remark}

Fix a link $L$, which in our case will be a closed manifold, and let $M,M'$ be two compact spin pseudomanifolds with fibered $L$-singularities with strata $\{M_{reg}, \beta M\}$ and $\{M'_{reg}, \beta M'\}$ respectively. Moreover, suppose that both have dimension $n$. We introduce the following bordism groups.

\begin{itemize}
	
	\item$\Omega^{spin,L-fib}_n$. 
	
	Cycles are given by compact $n$-dimensional spin pseudomanifolds with fibered $L$-singularities.
	 We say that $M$ and $M'$ are equivalent if there exists a compact $n+1$-dimensional smoothly spin stratified space with fibered $L$-singularities with boundary $(W, \partial W)$, where $\partial W = M \sqcup M'$. This means that W consists in two strata $\{W_{reg}, \beta W\}$ such that:
	 
	 \[W_{reg} \cap \partial W = M_{reg} \sqcup M'_{reg} \quad \quad \beta W \cap \partial W = \beta M \sqcup \beta M'.
	 \] 
	 
	 Moreover, the depth-$1$ stratum $\beta W$ has a tubular neighborhood $N(\beta W)$ and a fibration $N(\beta W) \rightarrow \beta W$, whose fibers are cones $C(L)$, which restricts to those of $\beta M$ and $\beta M'$ on the boundary. 
	 Finally, it is required $W$ to be spin and that the spin structures of its strata extend those of the strata of $M$ and $M'$. \\
	  Note that a spin pseudomanifold with fibered $L$-singularities with boundary $(X, \partial X)$ represents a bordism in this group from its boundary to the empty set. Thus, in particular $\partial X$, i.e. a boundary of a smoothly spin stratified space with fibered $L$-singularities, represents the zero class, hence the identity element in $\Omega^{spin, L-fib}_n$.

	 \item $Pos^{spin, L-fib}_n$.
	 
	 In this group, cycles are required in addition to be endowed with wedge metrics of positive scalar curvature. \\
	 Let $M$ and $M'$ be as above, and suppose they are equipped with wedge metrics of positive scalar curvature $g$ and $g'$ respectively whose associated metrics on their depth-$1$ stratum are $g_{\beta M}$ and $g'_{\beta M'}$. \\
	 We say that $M$ and $M'$ are bordant if there exists a compact spin pseudomanifold with fibered $L$-singularities with boundary $(W, \partial W)$ such that $\partial W= M \sqcup M'$ and $W$ is endowed with a psc wedge metric $h$. In particular $h$ is asked to be of product type on a collar neighborhood of the boundary $\partial W$ and to restrict there to $g$ and $g'$.
	 Denoting by $\pi: N(\beta W) \rightarrow \beta W$ the fibration of the depth-$1$ stratum $\beta W$ of $W$, the above request implies the following data: 
	 \begin{itemize}
	 	\item a Riemannian metric $g_{\beta W}$ on its depth-$1$ stratum $\beta W$ which is of product type near $\partial(\beta W)= \beta M \sqcup \beta M'$ and restricting to $g_{\beta M}$ and $g'_{\beta M'}$;
	 	\item a connection, say $\nabla_W$, on $T(\partial W_r)$ inducing a splitting in a vertical and horizontal part restricting to the respective connections on the two boundary components;
	 	\item on the tubular neighborhood $N(\beta W)$, denoting by $r$ the radial coordinates on the cones $C(L)$, the metric $h$ takes locally the form:
	 	
	 	\[dr^2+\pi^*g_{\beta W}+ r^2g_{\partial W_r / \beta W}+O(r).\]
	 	
	 	 \end{itemize}	
	 
	Observe that two equivalent pseudomanifolds in such group are also equivalent in $\Omega_{n}^{spin, L-fib}$.

	 \item $R^{spin, L-fib}_n$. \\
	 Assume now that $M$ and $M'$ are with boundary $\partial M$ and $\partial M'$ respectively. Suppose such boundaries to be endowed with wedge metrics of positive scalar curvature and call them $g_{\partial M}$ and $g_{\partial M'}$. Cycles in this group are given by spaces like these.\\
	  A bordism between $M$ and $M'$ will be represented by a compact spin pseudomanifold with fibered L-singularities with corners $(W, \partial W, \partial^2 W)$, meaning that $W$ has two strata $\{W_{reg}, \beta W\}$ and boundary:
	 
	 \[ \partial W = \partial^0 W \bigcup_{\partial^2 W} \partial^1 W,\] 
	 where:
	 
	\begin{itemize} 
	 \item $\partial^0 W= M \sqcup M'$, i.e. it is a pseudomanifold with fibered $L$-singularities with boundary $\partial^2 W:=\partial M \sqcup \partial M'$;
	 \item $\partial^1 W$ is a pseudomanifold with fibered $L$-singularities with same boundary $\partial^2 W$. In particular, we ask that $\partial^1 W$ realizes a bordism between $\partial M$ and $\partial M'$, in the sense of $Pos^{spin, L-fib}_{n-1}$ introduced before. This means that on $\partial^1 W$ there is a positive scalar curvature wedge metric which is of product type in a collar neighborhood of the boundary and restricting to $g_{\partial M}$ and $g_{\partial M'}$.
	 \item $\partial^2 W$ represents the corner locus of the pseudomanifold. It is itself a stratified pseudomanifold with $L$-fibered singularities of dimension $n-1$, endowed with a neighborhood $\mathcal{V} \subset W$, which is stratified isomorphic to $[0,1]^2 \times \left(\partial X \sqcup \partial X'\right)$ and its strata are then $\{(M_{reg} \cap \partial M) \sqcup (M'_{reg} \cap \partial M'), (\beta M \cap \partial M) \sqcup (\beta M' \cap \partial M')\}$.
\end{itemize}
 The bordism group $R^{spin, L-fib}_n$ associated to this relation will be called the n-dimensional \textit{$L$-fibered Stolz group}.

\end{itemize}

Observe that in all three of the groups described above, it is required that the space \( W \) realizing the bordism be spin stratified. In particular, its strata must be endowed with spin structures that extend those of the strata of \( M \) and \( M' \).

\begin{remark}
	Observe that the $L$-bordism groups for which $n \leq dim(L)$, meaning that they are represented by pseudomanifolds of dimension less than that of the link $L$, it is assumed that the singular stratum is empty. 
\end{remark}

Now, if X is a generic topological space, then we can extend all the above bordism theories by requiring that all the spaces are endowed with continuous reference maps with values in $X$. Then, a class in each of the above theories is given by specifying a pair $(M, f: M \rightarrow X)$, with $M$ as above, and in the equivalence relations one simply requires that such reference maps are interpolated along the bordism. The above bordism groups are then simply bordism groups whose reference space is a point.

\subsection{The L-fibered Stolz sequence} \label{LStolzseq}

Let us now proceed to explain in detail the maps that will constitute the Stolz sequence. Recall from the previous section that a cycle $[M]$ means:

\begin{itemize}
	 \item a pair $(M,f\colon M \to X)$, if $[M] \in \Omega_n^{spin, L-fib}(X)$;
	 \item a tuple of the form $(M, g, g_{\beta M}, \nabla_M, f: M \rightarrow X)$, if $[M] \in Pos^{spin, L-fib}_n(X)$;
	 \item $(M, \partial M, g_{\partial M}, g_{\beta(\partial M)}, \nabla_{\partial M}, f: M \rightarrow X)$, when $[M] \in  R^{spin, L-fib}_n(X)$.
	 
\end{itemize}	 

Then, define the following maps:

\begin{enumerate}[label=(\roman*)]
	\item The \textit{forgetful map}: 
	
	\[\varphi: Pos^{spin, L-fib}_*(X) \longrightarrow \Omega^{spin, L-fib}_*(X),\]
	sends a class represented by $(M, g, g_{\beta M}, \nabla_M, f: M \rightarrow X)$ to the class represented by $M$ and $f$. This map simply forgets about the wedge psc metric.
	
	\item The \textit{inclusion map}:
	
	\[\iota: \Omega^{spin, L-fib}_*(X) \longrightarrow R^{spin, L-fib}_*(X),\]
	sends the class of $(M,f)$ to the class represented by $M$ and $f$ seen as a pseudomanifold with $L$-fibered singularities with empty boundary (and then with no psc wedge metric metric).
	
	\item The \textit{boundary map}:
	
	\[\partial: R^{spin, L-fib}_*(X) \longrightarrow Pos^{spin, L-fib}_{*-1}(X),\]
	sends a class $[(M, \partial M, g_{\partial M}, g_{\beta(\partial M)}, \nabla_{\partial M}, f: M \rightarrow X)]$ to the class represented by $(\partial M, g_{\partial M}, g_{\beta(\partial M)}, \nabla_{\partial M}, f|_{\partial M}: \partial M \rightarrow X))$, i.e. acts by a restriction to the boundary. \\
\end{enumerate}

We observe that the boundary map is well defined: in fact, consider $[M] \sim [M'] \in R^{spin, L-fib}_*(X)$ and $W$ the $n+1$-dimensional space which realizes the bordism. By the definition of the equivalence relation in the $R$-group, the boundary component $\partial^1 W$ of $W$ realizes a bordism between the boundaries of $M$ and $M'$. Then we have that:
 \[\partial[M]=[\partial M] \sim [\partial M']=\partial[M'] \in Pos^{spin, L-fib}_{*-1}(X).\]

\begin{definition}
	Let X be a topological space and $L$ be a fixed compact smooth manifold, then we have the following sequence of bordism groups:
	
\begin{equation} \label{stolzseq}
\xymatrix{
	\ldots \ar[r] & R^{spin, L-fib}_{n+1}(X) \ar[r]^{\partial} & Pos^{spin, L-fib}_{n}(X)\ar[r]^{\varphi} & \Omega^{spin, L-fib}_{n}(X)\ar[r]^{\iota} & R^{spin, L-fib}_{n}(X)\ar[r] & \ldots},
\end{equation}

which we call the \textbf{L-stratified Stolz sequence}.
	
\end{definition}

It is not difficult to show that this sequence is in particular a complex of abelian groups. In fact:
\begin{itemize}
	\item $\varphi \circ \partial = 0$ \\
	Consider a class $[M, f \colon M \to X] \in R^{spin, L-fib}_{n+1}(X)$, then its image under the map $\varphi \circ \partial$ is represented by $(\partial M, f|_{\partial M})$ in $\Omega^{spin,L-fib}_n(X)$. Then  $M$ itself trivially represents a bordism between $\partial M$ and the empty set. Hence we have that: 
	\[\varphi \circ \partial([M]) \sim [\emptyset] = 0\].
	
	\item $\partial \circ \iota=0$ \\
	This is straightforward since the image of $\iota$ is given by all pseudomanifolds with empty boundary. Then:
	 \[\partial \circ \iota[M]= [\partial M] = [\emptyset] = 0, \quad \forall \ [M] \in \Omega^{spin, L-fib}_n(X).\]
	
	\item $\iota \circ \varphi=0$ \\
	Let $[M, f\colon M \to X] \in Pos^{spin, L-fib}_n(X)$, then the composition $\iota \circ \varphi$ first forgets about the metric and then consider the space as a pseudomanifold with empty boundary. Let $W:=M \times [0,1]$ (whose stratification is obtained from that of \(X\) by simply taking the direct products of the strata of \(X\) with the closed interval \([0,1]\)) and $\tilde{f}: W \rightarrow X$ be the trivial extension of $f$.\\
	 $W$ is a pseudomanifold with fibered $L$-singularities with boundary $\partial W= M \sqcup M$ and realizes a bordism in the $R$-group between $M$ and the empty set. In fact, using the same notation of the definition of the $R$-groups, we identify a copy of $M$ as $\partial^0 W$ and the other one as $\partial^1 W$, with no intersection since both $M$ and the empty set have empty boundary. If \(\partial^1 W = M\) is equipped with the wedge metric of positive scalar curvature of \(M\) defining \([M]\), then this is a stratified space without boundary, thereby realizing a bordism in \( Pos^{spin, L-fib}_n(X) \) between the empty set and the boundary of \(M\).
	
\end{itemize}

\begin{theorem}\label{Lstolzexact}
	The $L$-stratified Stolz sequence (\ref{stolzseq}) is exact.
\end{theorem}
 
 \begin{proof}
 	We have just shown that the sequence constitutes a complex, so what remains is to show that the images of each map are included in the kernels of the subsequent ones.
 	\begin{enumerate}
 		
 	\item $Ker(\varphi) \subseteq Im(\partial)$\\
 	Let $[M, f \colon M \to X] \in Pos^{spin, L-fib}_{n}(X)$ be such that $\varphi([M])=0$. Then there exists a bordism between $M$ and the emptyset, i.e. a spin pseudomanifold with fibered $L$-singularities $W$ with boundary $\partial W=M$ and a mapping $f_W: W 	\rightarrow X$ extending $f$. Of course, $M$ admits a wedge metric of positive scalar curvature, and then $W$ on its boundary too, hence $W$ represents the required class in $R^{spin, L-fib}_{n+1}(X)$ such that $\partial([W])=[M]$.
 	
 	\item $Ker(\partial) \subseteq Im(\iota)$\\
 	Let $[M, f \colon M \to X] \in R^{spin, L-fib}_{n}(X)$ be such that $\partial([M]) = 0$. Then the boundary $\partial M$ is bordant to the empty set, meaning that there exists a pseudomanifold with fibered $L$-singularities $W$ such that $\partial W = \partial M$, a wedge psc metric on $W$ extending suitably that of $\partial M$ and a mapping $f_W: W \rightarrow X$ extending the restriction $f|_{\partial M}$. \\
 	Now, construct the following: 
 	\[N:=M \cup_{\partial M} W,\]
 	with the obvious map to $X$. This is a spin pseudomanifold with fibered $L$-singularities without boundary whose stratification, following the previous notations, is given by:
   \[N_{reg}:=M_{reg} \bigcup_{\partial M_{reg}} W_{reg}, \quad \quad \beta N:=\beta M \bigcup_{\beta(\partial M)} \beta W. \]
   
  Now, the product $N \times [0,1]$, whose stratification is given as before, has boundary $N \sqcup N$ decomposable as:
 \[N \sqcup N=(M \cup_{\partial M} W) \sqcup N =(M \sqcup N) \cup_{\partial M} W. \]
 
 Identify $\partial^0 N=M \sqcup N$ and $\partial^1 N=W$ and observe then that, by construction, $\partial^1 N$ provides by construction a bordism in $Pos^{spin, L-fib}_{n-1}(X)$ between $\partial M$ and $\partial N=\emptyset$. But then $N \times [0,1]$ represents a bordism between $M$ and $N$ in $R^{spin, L-fib}_{n}(X)$ and since $N$ has no boundary, we have that $[M]=\iota([X])$.
 
 \item $Ker(\iota) \subseteq Im(\varphi)$ \\
 Let $[M, f \colon M \to X] \in \Omega^{spin, L-fib}_n(X)$ be such that $\iota([M])=0$. This means that there exists a spin pseudomanifold with fibered $L$-singularities $W$ with empty corners, boundary $\partial W=M \sqcup \partial^1 W$ (we have identified $\partial^0W$ with $M$) and with a mapping $f_W\colon W \to X$ extending $f$. Note that both \( M \) and \( \partial^1 W \) have no boundary, then $W$ can be also seen as a bordism in $\Omega^{spin, L-fib}_n(X)$ between them, i.e. $[M] \sim [\partial^1W]$. However, by construction, $\partial^1 W$ realizes a bordism in $Pos^{spin, L-fib}_{n-1}(X)$ between the emptyset and itself, hence an n-dimensional spin pseudomanifold with fibered $L$-singularities without boundary endowed with a wedge metric of positive scalar curvature. Then $[M] \sim \varphi([\partial^1 W])$, with $[\partial^1 W] \in Pos^{spin, L-fib}_n(X)$ and

 	\end{enumerate}	
 \end{proof}

\newpage

\section{Stratified Morse theory in the (L,G)-setting}\label{secstrmorse}

\subsection{Controlled vector fields}

Observe that a generic smooth manifold $M$ can be endowed with a (trivial) stratified space structure with the only stratum $M$ itself and control data given by:

 \[\{T_M=M, \pi_{M}=Id_M, \rho_M=0\}\]

\begin{definition}
	Let $X$ be a stratified space and $M$ be a smooth manifold. Endow $M$ with the above trivial stratification, then a weak morphism $F: X \rightarrow M$ is called \textbf{controlled}.
\end{definition}

In particular, if the above weak morphism is such that its restrictions $F|_{Y_\alpha}$ to the smooth strata of $X$ are submersions, then we say that $F$ is a \textit{controlled submersion}.

\begin{example}\label{excontrolledsubm} Consider a stratum $Y_{\alpha}$ of a stratified space $X$, then both the restrictions to $K:=T_{\alpha} \setminus Y_{\alpha}$:
	\[(\pi_{\alpha}, \rho_{\alpha})|_K\colon K \to Y_\alpha \times \mathbb{R}_{>0}, \quad \rho_{\alpha}|_K\colon K \to \mathbb{R}_{>0}\] 
	are controlled submersions.
\end{example}	

This class of maps is of particular importance because the following result can be proven, which allows, as in the smooth case, the definition of global objects by working locally.

\begin{lemma}[\cite{Verona}, Lemma 1.3]\label{partunity}
	Let $X$ be a stratified space. For each open covering $\mathcal{U}$ of $X$, there exists a controlled partition of unity subordinated to $\mathcal{U}$, i.e. made by controlled mappings.
\end{lemma}

As for the smooth case, the above lemma holds similarly in the case of stratified spaces with boundary.

\begin{definition}
	A \textbf{stratified vector field} $\xi$ on a stratified space $X$ with strata $\{Y_{\alpha}\}$ is a collection:
	
	\[\xi=\{\xi(x) \in T_xY_\alpha: \quad x \in  Y_{\alpha}\},\]
	such that for each stratum $Y_\alpha$, the map:
	\[\xi_\alpha\colon Y_\alpha \to T_xY_\alpha, \quad   x \mapsto \xi(x),\]
	defines a vector field on $Y_\alpha$. If the family $\{\xi_\alpha\}$ is made by smooth (resp. $C^k$, with $k \in \mathbb{N}$) vector fields, then $\xi$ is said smooth (resp. of class $C^k$).
\end{definition}

In the following, we will not make any distinction between $\xi_{\alpha}$ and $\xi|_{\alpha}$.

Now we will introduce compatibility conditions for the vector fields defining a stratified vector field. Indeed, the given definition does not guarantee anything about the behavior of these fields in the neighborhoods of a stratum, and therefore, a priori, no condition of "continuity" can be assumed when transitioning from one vector field of the family to another. These conditions, in particular, will allow us to conclude that the flow generated by the vector field is locally continuous.

\begin{definition}
	A stratified vector field $\xi$ is said \textbf{controlled} if for each pair of strata $Y_{\alpha} < Y_{\beta}$ and $x \in T_{\alpha} \cap Y_{\beta}$, it satisfies the following:
	
	\begin{itemize}
		\item $(\pi_{\alpha})_{*}(\xi_{\beta})(x)=\xi_{\alpha}(\pi_{\alpha}(x))$;
		\item $(\rho_{\alpha})_{*}(\xi_{\beta})(x)=0$.
	\end{itemize}
\end{definition}

\begin{proposition}\label{controlledlift}
	Let $X$ and $M$ be as above. If $F\colon X \to M$ is a controlled submersion, then for each vector field $\chi\colon M \to TM$, there exists a controlled stratified vector field $\xi$ on $X$ such that on each stratum $Y_\alpha$:
	
	\[\eta \circ F|_{Y_\alpha}= (F|_{Y_\alpha})_*(\xi_\alpha).\]
	
	\end{proposition}

In particular, we call the stratified vector field $\xi$ of \ref{controlledlift} a \textit{controlled lift} of the vector field $\chi$.

\begin{remark}
	Observe that if one has a stratified space $X$ with stratification given by:
	 \[X_1 < X_2 < \ldots < X_n, \quad dim(X_i) < dim(X_{i+1}),\] 
	 then a vector field on the stratum $X_1$ has a controlled lift on $X$ (recall Example \ref{excontrolledsubm}).
\end{remark}

We now observe that the properties we will use later also hold for a weaker class of stratified vector fields. In particular, we introduce the following two conditions for a stratified vector field \(\xi = \{\xi_\alpha\}\), valid for suitable pairs of strata $Y_\alpha < Y_\beta$ and $x \in T_\alpha \cap Y_\beta$:

\begin{itemize}
	\item \begin{equation} \label{weakcontrol1} (\pi_{\alpha})_{*}(\xi_{\beta})(x)=\xi_{\alpha}(\pi_{\alpha}(x))+\rho^2_{\alpha}(x)\chi(\pi_{\alpha}(x)), \end{equation}
	where $\chi$ is a bounded stratified vector field;
	\item \begin{equation} \label{weakcontrol2}
	|(\rho_{\alpha})_{*}(\xi_{\beta})(x)|<A\rho_{\alpha}(x),
\end{equation}
for some positive constant $A$.
\end{itemize}

With these weaker conditions, it can be shown that the stratified vector field still produces a locally continuous flow, i.e. it is locally integrable (see \cite[Proposition 2.5.1]{Duplessis}).

\begin{definition}
      Let $Y_{\alpha}$ be a stratum of a stratified space $X$. A stratified vector field $\xi_r$ is called \textbf{radial} with respect to $Y_{\alpha}$ if:
      
      \begin{itemize}
      	\item $\xi_r|_{Y_{\alpha}}=0$;
      	\item $\xi_r|_{T_{\alpha} \setminus Y_{\alpha}}$ is a controlled lift of the vector field $-t \frac{\partial}{\partial t}$ of $\mathbb{R}$ along the controlled submersion $\rho_{\alpha}|_{T_{\alpha} \setminus Y_{\alpha}}$ in the sense of Proposition \ref{controlledlift}.
      \end{itemize}
\end{definition}

Notice that in the above definition, in particular for the second point, the weaker version of controlled vector field is needed. 

\begin{definition}
	Let $\xi_{\alpha}$ be a vector field of a smooth stratum $Y_{\alpha}$ of $X$, then we say that a stratified vector field $\xi$ on $T_{\alpha}$  is a \textbf{radial extension} of $\xi_{\alpha}$ if: 
	
	\[\xi(x)=\eta(x)+\xi_r(x), \ \ \forall x \in T_{\alpha}\]
	
	where $\xi_r$ is radial with respect to the stratum $Y_{\alpha}$ while $\eta$ is the sum of a controlled lift of $\xi_{\alpha}$ using the controlled submersion $\rho_{\alpha}$ and a $O(\rho^2_{\alpha}$) term.
\end{definition}

\subsection{Morse pairs}

We now introduce the concept of a Morse pair for a stratified space. In particular, this will represent the analog of choosing a Morse function on a smooth manifold, with which it can be established that the homotopy type of a stratified space is that of a CW-complex.

From now on, we furthermore assume that the stratified space $X$ is smoothly stratified, compact and of depth 1.

\begin{definition}
	  A \textbf{smooth function} $f\colon X \to \mathbb{R}$ (resp. $C^k$, $k\geq0$) on a stratified space $X$ is a continuous function on \( X \) such that its restrictions to all strata \( Y_\alpha \) of \( X \) are smooth (resp. $C^k$) functions, i.e.:
	  
	  \[f=\{f_{\alpha}=f|_{Y_{\alpha}} \colon f_\alpha \in C^{\infty}(Y_{\alpha}) \ (C^k(Y_{\alpha})) \}. \]
\end{definition}

It is clear that many concepts regarding smooth functions on smooth manifolds (in particular, in this case, regarding Morse functions) can thus be adapted to the stratified context by utilizing their respective restrictions on the smooth strata. 

\begin{definition} We say that \( x \in X \) is a \textbf{critical point} for a smooth function \( f \) on \( X \) if it is a critical point for the restriction \( f_\alpha \), where \( Y_\alpha \) is the stratum containing \( x \).
\end{definition}
	
\begin{definition}
Consider a pair $(f,g)$ consisting of a smooth function and a wedge metric on $X$. The \textbf{gradient vector field} $\nabla_gf$ of $f$ is defined as the stratified vector field obtained by the following family:
	
	\[ \{ \nabla_{g_{\alpha}}f_{\alpha} \colon g_\alpha(\nabla_{g_{\alpha}}f_{\alpha},v)(x)=(df_{\alpha})(v)(x), \quad \forall x \in Y_{\alpha}, \quad \forall v \in T_{x}Y_{\alpha}\},\]
	where $g_\alpha$ denote the Riemannian metric on the stratum $Y_\alpha$ induced by $g$.
\end{definition}	

We now need to introduce a suitable notion of \textit{non-degeneracy} for a critical point $p \in X$. In particular, an extra condition on the behaviour of the gradient vector field in the normal direction to the stratum containing a critical point along its tubular neighborhood $T_{\alpha}$ is requested.
 
\begin{definition}
	Let $(f,g)$ be a pair consisting of a smooth function $f$ and a wedge metric $g$ on a stratified space $X$. Then a critical point $p \in Y_{\alpha}$ is \textbf{non-degenerate} if:
	
	\begin{itemize}
		\item $p$ is a non degenerate critical point for the smooth function $f_{\alpha}$;
		\item on a neighborhood of $p$, the gradient vector field $\nabla_{g}f$ is a radial extension of $\nabla_{g_{\alpha}}f_{\alpha}$.
	\end{itemize}
\end{definition}

\begin{definition} \label{morsepair}
	The pair $(f,g)$ is said to be a \textbf{Morse pair} if all the critical points are non-degenerate and the negative gradient vector field $-\nabla_gf$ satisfies the weak control conditions (\ref{weakcontrol1}) and (\ref{weakcontrol2}) and:
	
	 \[(\rho_{\alpha})_*(-\nabla_gf)<0, \quad \forall \ \alpha.\]
	
\end{definition}

A direct consequences of the definition is that the flow lines of $-\nabla_gf$ go from larger strata to smaller strata. Moreover, the condition (\ref{weakcontrol2}) implies that they cannot reach the smaller strata in finite time, as \( \rho_\alpha \) tends to zero as they approach them.
\begin{definition} The \textbf{index} of a critical point $p$ of a Morse pair $(f,g)$ as the index of the restriction $f_{\alpha}$ in the usual sense if $p$ is contained in $Y_{\alpha}$, i.e. as the number of the negative eigenvalues of the non singular Hessian of $f_{\alpha}$ at $p$.
\end{definition}

\begin{lemma}\label{exstncmorsepair}
	On a compact stratified space $X$ of depth $1$ there always exists a Morse pair $(f,g)$.
\end{lemma}

\begin{proof}
	Call $X_{reg}$ and $Y$ the strata of $X$ of depth $0$ and $1$ respectively. T: Given the existence of a partition of unity, as mentioned in Lemma \ref{partunity}, it is sufficient to reason locally and use this to define the object globally. Since both $X_{reg}$ and $Y$ are compact smooth manifolds, they admit a Morse function: indicate by  $f_Y$ that on $Y$. Then, denoting by $(\pi, \rho): T_Y \rightarrow Y \times \mathbb{R}_{\geq 0}$ the mappings of its tubular neighborhood, we extend $f_Y$ on $T_1$ as the following:
	
	\[f'_Y(x)= f_Y \circ \pi (x) + \rho^2 (x), \quad \forall x \in T_Y.\]
	
	Each critical point of $f'_Y$ lies in $Y$ and is a critical point of $f_Y$. It is easy to show that the negative gradient vector field of $f'_Y$ satisfies the control conditions requested in Definition \ref{morsepair}.
\end{proof}

Observe that by construction, no critical points of such a Morse pair $(f,g)$ are contained in $T_Y \setminus Y$. Moreover, analogously to the smooth case, we can suppose that $f$ is \textit{self-indexing}, i.e. such that for each critical point $p$ of index $k$ one has that $f(p)=k$.

\begin{remark}Lemma \ref{exstncmorsepair} can be generalized obtaining a pair $(f,\xi)$ of a smooth function on $X$ and a stratified vector field satisfying suitble control conditions and with particular classes of singular points. In that context, our case will thus represent an example where \(\xi = -\nabla_g f\). See \cite[Proposition 6.2]{Ludwig} for details.
\end{remark}
Now, we want to extend this notion to the case of a space with non empty boundary. In particular, we consider spaces which may represent a bordism between stratified space.  

Let $(W, X_1, X_2)$ be a triple made by a compact depth-$1$ stratified space with boundary $(W,\partial W)$ such that $\partial W= X_1 \sqcup X_2$.

\begin{definition}\label{specialmorse}
	Given a triple $(W,X_1,X_2)$ as above, we say that a pair $(f,g)$, with $f$ a smooth function and $g$ a wedge metric on $W$ respectively, is a \textbf{special Morse pair} if:
	
	\begin{itemize}
		\item $f: W \rightarrow [a,b]\subset \mathbb{R}$ and $f^{-1}(a) = X_1$, $f^{-1}(b)=X_2$;
		\item all critical points of $f$ are non degenerate and lying in the interior of $W$;
		\item $-\nabla_gf$ satisfies the same conditions of Definition \ref{morsepair}, except that $(\rho_{1})_*(-\nabla_gf)\leq0$ and it vanishes on a neighborhood of the boundary.
	\end{itemize}
\end{definition}

\begin{lemma} \label{existencespecialmorsepair}
	Let $(W,X_1,X_2)$ be a triple as above, then there always exists a special Morse pair $(f,g)$.
\end{lemma}

\begin{proof}
	Let us call by $\{T_{Y}, \pi,\rho\}$ its control datum associated to its stratification $\{W_{reg}, Y\}$. 
	Recall that the intersections of the strata of $W$ with the boundary $\partial W$ realize the stratifications of both the two boundary components $X_1$ and $X_2$. By construction, since $W_{reg}$ and $Y$ are compact manifold with boundary $W_{reg} \cap \partial W$ and $Y \cap \partial W$ respectively, then they both admit special Morse functions (see \cite[Theorem 2.5]{milnor}) such that:
	\begin{enumerate}
		\item they take arbitrary constant values on the two boundary components respectively;
		\item elsewhere, they take on values that are between those assumed on the boundary;
		\item they admit critical points only in the interiors.
	\end{enumerate}

 Observe that such functions can be defined by using a collar neighborhood of the boundary and its defining function. We then consider a special Morse function on $Y$ of this kind which takes values $a$ and $b$ on its boundary components $Y \cap X_1$ and $Y \cap X_2$ respectively. We now need to find a suitable extension of the special Morse function $f_Y$ defined on $Y$ along the tubular neighborhood $T_Y$. 
%	 Suppose that an open set $U_i$ intersects one of the two boundary components and is a neighborhood of a point of $W_0$. There we extend the function on $T_0$ using the submersion $\pi_0$, i.e. we declare that
%	
%	\[f(x)=f_0 \circ \pi_0 (x), \forall x \in T_0 \cap \mathcal{U}_i \]
%	
%	If the open neighborhood does not intersect any boundary component, then we extend the function as done in the previous lemma. Finally, using the controlled partition of unity we get the requested function $f$.

Again, we know that there is a neighborhood of the boundary $\partial Y$ containing no critical points of $f_Y$. Now, take a non negative function $\phi \in C^{\infty}(Y)$ such that the complement of its support is contained in such a neighborhood. In particular, all critical points of $\phi$ are contained in its support.

Then, extend the function $f_Y$ on $T_Y$ as follows:

\begin{equation}f(x)=f_Y(\pi(x))+\phi(\pi(x))\rho^2(x), \quad \forall x \in T_Y \label{f}
\end{equation}	

Such extension is compatible with the weak control conditions requested in the above definition and the flow lines of the negative gradient vector field goes from larger to smaller stratum except for the neighborhood in which $\phi=0$, where these stay on a level set of the function $\rho$ and then there:
 \[(\rho)_*(-\nabla_gf)=0.\]
Now, by taking into account the special Morse function on the stratum $W_{reg}$ with same values of $f_Y$ taken on the boundary components, and a controlled partition of unity Lemma \ref{partunity}, we obtain the requested special Morse pair.
\end{proof}

\begin{remark}\label{levelsurface}
	Observe that in the above construction the function $\phi$ can be chosen with an additional property, i.e. that it has non vanishing gradient only in the neighborhood of the boundary specified in the above proof. Moreover, one can suppose that there the level surfaces of $\phi$ are exactly the ones of $f$.
\end{remark}

Now, let $(f,g)$ be a Morse pair on a pseudomanifold with fibered $L$-singularities $X$. Suppose $f(X) \subseteq [a,b]$ and take a regular value $c \in [a,b]$. Then it is known that $f^{-1}(c)$ is a stratified space whose stratification is given by the intersections of the strata of $X$ with $f^{-1}(c)$. In particular, the proof makes use of the weak control condition of the gradient vector field of $f$ and hence the result is also valid in the case of a special Morse pair on a compact stratified space with boundary. 

In the following proposition we want to extend this by saying more about the regular level set $f^{-1}(c)$ in the case of a pseudomanifold with fibered $L$-singularities.

\begin{proposition}
	Let $(W,X_1,X_2)$ be a triple as above with $W$ a pseudomanifold with fibered $L$-singularities. Then if $c \in [a,b]$ is a regular value of a special Morse pair $(f,g)$ on $W$, $f^{-1}(c)$ is a pseudomanifold with fibered $L$-singularities.
\end{proposition} 

\begin{proof}
	Call $W_{reg}$ and $Y$ the depth-$0$ and depth-$1$ strata of $W$. As said before, $f^{-1}(c)$ is a compact stratified space of depth $1$ and its stratification is given by $\{f^{-1}(c) \cap W_{reg}, f^{-1}(c) \cap Y\}$. Moreover, let us denote its control datum by $\{T_Y,\pi,\rho\}$.  We need now to prove that $f^{-1}(c)$ is smoothly stratified and that its depth-1 stratum has link diffeomorphic to $L$.
	
	 Assume $(f,g)$ has been  extended along the tubular neighborhood $T_Y$ of $Y$ as in (\ref{f}). In a region $\mathcal{U} \subset Y \cap f^{-1}(c)$ such that $\pi(\mathcal{U})$ belongs to the zero locus of $\phi$, $f$ is trivially extended, meaning that the $L$-cones over $\mathcal{U}$ still belongs to the level set $f^{-1}(c)$, then the control datum is obtained by restriction.
	 
	However, the observation made in Remark \ref{levelsurface} ensures that the level surfaces of $f$ can be characterized by being completely contained in the support of $\phi$ or in its complement. Assume we are in the first, and hence non trivial, case. 
	
	By (\ref{f}), for a fixed value of the radial map $\rho$, the function depends only on $\pi(x)$: in particular, it is constant on the link $L$. Moreover the function $f$ increases while moving from the vertex of the cone, since $\phi$ is positive there. Then if $p \in f^{-1}(c) \subset Y$, except for its vertex, the fiber $c(L)$ over $p$ will not lie in $f^{-1}(c)$.
	
	 Since $f_Y$ is a smooth Morse function on $Y$ and $c$ is a regular value for $f$, the gradient vector field $\xi:=-\nabla_{g_Y}f_Y$ is non vanishing on $f_Y^{-1}(c)$ and normal to it. This means that the flow line of $\xi$ passing through $p$ on time $t_0=0$:
	 
	 \[\{p_t:=\Phi_{\xi, p}(t), \ t \in [-\delta,\delta],\  \delta >0\} \subset Y,\]
	 wil intersect $f_Y^{-1}(c)$ transversally. Clearly $f(p_t)=f_Y(p_t)<c$, for $t>0$.
	 
	By what discussed in section \ref{smoothlystrspace}, there is a open neighborhood $\mathcal{U}_p \subset Y$ around $p$ whose preimage along $\pi$ is domain of a trivialisation, meaning that:
	
	\[\pi^{-1}(\mathcal{U}_p) \simeq \mathcal{U}_p \times c(L), \]
	via a stratified isomorphism $\psi_p$.
	
	Therefore, let $t_p>0$ such that all the points $p_t$ are contained in such $\mathcal{U}_p$ for $0 \leq t < t_p$. From (\ref{f}) it is clear that for each $t \in [0,t_p)$ there exist a unique $\rho_t \geq 0$ and a set:
	
	\[ K_{p_t}:=\{ y\in T_Y \colon \pi(y)=p_t, \  \rho(y)=\rho_t\},\]
	both depending on $t$ such that:
	\[f(x)=c, \quad \forall x \in K_{p_t}.\]
	
	In particular, $\rho_t$ is zero if and only if $t=0$ (for which $K_{p_0}$ is simply given by $x=p_0=p$). 
	
	Using the above trivialisation one can observe that for a fixed value of $t>0$, since $\rho_t$ is a fixed value, the set $K_{p_t}$ is diffeomorphic to the link of the stratum $Y$, i.e. to $L$:
	
	\begin{equation} \label{trivialisation}
		(\pi,\rho)^{-1}(p_t,\rho_t) \simeq L, \ \  \forall t \in (0,t_p].
	 \end{equation}

	Conversely, consider a point $x \in T_Y \cap f^{-1}(c)$ such that $\pi(x)$ lies in a sufficiently small neighborhood of $f_Y^{-1}(c)$ (we will be more precised later). Then we know that $\pi(x)$ must lie in a unique flow line for $\xi$, i.e. $\pi(x)=p_t$, for some point $p \in f_Y^{-1}(c)$ and positive time $t$. Of course, if $x$ lies in the stratum $Y$, and $x$ coincides with $p$ and $t=0$.
	
	Define the following:
	
	\[\pi_{c}(x):= \Phi_{\xi,\pi(x)}(-t), \quad \quad \rho_{c}(x):=t.\]
	
	These are well defined continuous map and in particular $\pi_{c}$ is a continuous retraction from an open set of $T_Y \cap f^{-1}(c)$ to $Y \cap f^{-1}(c)$. 
	
	In particular, for a fixed a point $p \in f_Y^{-1}(c) \subset Y$, the set $\pi_{c}^{-1}(p) \subset T_Y$ consists of all points $x \in T_Y$ such that $f(x)=c$, $\pi(x)=p_t$ and $\rho(x)=\rho_t$, for some $t$, i.e. $K_{p_t}$.

	 In particular, if $0<t<t_p$ we can restrict the trivialisation $\psi_p$ defined on $\pi^{-1}(\mathcal{U}_p)$ and from (\ref{trivialisation}) have that:
	 \[ (\pi_{c},\rho_{c})^{-1}(p,t)=(\pi,\rho)^{-1}(p_t,\rho_t)\simeq L, \]
	 while for $t=0$ it is just the point $p$. 
	 
	 Then the fibre $\pi_{c}^{-1}(p)$ is identified to a cone $c(L)$ with radial coordinate given by the parameter $t$. In fact:
	 \[\psi_p|_{\pi_{c}^{-1}(p)}\colon \pi_{c}^{-1}(p) \rightarrow \faktor{p_t \times L}{\{p_0\} \times L} \simeq \faktor{[0,t_p) \times L}{ \{0\} \times L} \simeq \{p\} \times c(L)\]
	 
	 To each point $p \in f_Y^{-1}(c)$ we associated a $t_{p}>0$ such that the flow line $\{p_t\colon t \in [0,t_p)\}$ is contained in a trivialisation domain. However, we would like to find a unique  $t_c > 0$ for which the above is valid for each $p$. Consider a map $\lambda: f_Y^{-1}(c) \rightarrow \mathbb{R}_{\geq 0}$ which assign to each $p$ the largest $t_p$. Of course $\lambda$ is strictly positive and in order to find such $t_c$ it is sufficient to show the existence of $\epsilon > 0$ such that $\lambda(p)>\epsilon$ for all $p \in f_Y^{-1}(c)$. In that case, it suffices to fix $t_c := \epsilon$.
	 
	 Assume that such $\epsilon$ does not exist: this means that:
	 \[\forall \delta > 0, \quad  \exists \  q_{\delta} \  \colon \  \lambda(q_{\delta}) < \delta.\]
	 
	 From this, we can then obtain a sequence $\{q_n\} \subset f_1^{-1}(c)$ associated to $\{\delta_n=1/n\}$. 
	 
	 By compactness of $f_Y^{-1}(c)$, then such a sequence admits a subsequence convergent to a point, say $q$. Take a trivialisation domain $\mathcal{U}_q$ around $q$: there exists a subset $K_q \subset \mathcal{U}_q$ and a positive $t_q$ such that the evolution of $K_q$ along the flow of $\xi$ for the time $0 \leq t < t_q$ is completely contained in $\mathcal{U}_q$. But then, there exists $\overline{n} \in \mathbb{N}$ such that $\lambda(q_n)>t_q$, for all elements of the subsequence with $n \geq \overline{n}$ and this makes a contradiction.

\end{proof}

\begin{corollary}
	Let $X$ be a spin pseudomanifold with fibered $L$-singularities, and let $(f,g)$ be a Morse pair defined on it. Assume that $a,b$ are two regular values of $f$ with $a<b$, then the preimage $f^{-1}([a,b])$ is a spin pseudomanifold with fibered $L$-singularities with boundary $f^{-1}(a) \sqcup f^{-1}(b)$. 
	
	In particular, it realizes a bordism in $\Omega_*^{spin, L-fib}$ between the two level surfaces.
\end{corollary}

Let us consider a compact stratified space $W$ of depth $1$ and denote by $W_a$ and $W_{[a,b]}$ the sets of all points $p$ of $W$ are such that $f(p)\leq a$ and $f(p) \in [a,b]$ respectively. (see \cite[Lemma 8.1]{Ludwig})

\begin{proposition}
	Let $(f,g)$ be a Morse pair on $W$. 
	\begin{itemize}
		\item If the interval $[a,b]$ does not contain any critical value for $f$, $W_a$ and $W_b$ are homotopy equivalent.
		\item If $p$ is a critical point of index $k$ and the only one with critical value $c$, then $W_{c+\delta}$ has the homotopy type of $W_{c-\epsilon}$ with a $k-cell$ attached, where $\epsilon, \delta >0$ are chosen in orden to have an interval $[c-\epsilon, c+\delta]$ with no critical values other than $c$.
	\end{itemize}
\end{proposition}

\begin{corollary} A compact stratified space $X$ has the homotopy type of a $CW$-complex, whose cells are in correspondence with the critical points of a Morse pair $(f,g)$ defined on $X$. 
In particular, to a $k$-cell corresponds a critical point of index $k$.
	
	Moreover, if $f$ is self-indexing, one has the correspondence between its $k$-cells and the preimage of the critical value $k$.
\end{corollary}	

\newpage

\section{(L,G)-singularities}\label{L,Gsetting}

Let us proceed by introducing a class of stratified spaces that we will work with in the following sections. These are introduced and discussed in detail in \cite{BPR1} and \cite{BPR2}. 

The idea is that, while working on a smoothly spin stratified space of depth $1$, imposing extra conditions on the geometry of the link $L$ associated to its depth-$1$ stratum guarantees nice properties for the index associated to a Dirac operator similar to those of the closed smooth case. For example, one obtains well defined mappings from suitable bordism groups to $KO$-homology groups, i.e. involving suitable fundamental classes of Dirac operators (see \cite[Proposition 5.1]{BPR1}) or can extend metrics of positive scalar curvature along spaces realizing such bordisms (\cite[Theorem 4.5]{BPR1}). In particular, in this section, this latter property will play a relevant role. Furthermore, alongside these geometric conditions, an additional class of metrics adapted to this specific context will be introduced. Here we will provide a summary of the main results contained there.

Let us fix a compact, connected, semisimple Lie group $G$ with an $Ad$-invariant metric on its Lie algebra $\mathfrak{g}$. In this way, a bi-invariant metric on the whole group $G$ is obtained. 

\begin{remark}\label{killing}
	When $G$ is a simple Lie group, an $Ad$-invariant metric on $\mathfrak{g}$ is necessarily a multiple of its Killing form. In particular, this metric has constant Ricci curvature (meaning that the Ricci curvature is a multiple of the metric). In particular, up to scale by a constant, one can obtain a metric with constant positive scalar curvature equal to a fixed value (see \cite[Lemma 7.6]{Milnor2}). 
\end{remark}

Then, we ask $L$ to be a homogeneous space $G/K$, as defined in \ref{homogeneousspace}.
The tangent bundle of $L$ can be identified to $G \times_K \mathfrak{g}/\mathfrak{k}$, where $K$ acts on $\mathfrak{g}$ via the adjoint action and $\mathfrak{k}$ is the Lie algebra of $K$ (see, for example, \cite[Section 18.16]{Michor}).

 A metric on \(\mathfrak{g}\) thus induces a metric on \(\mathfrak{g}/\mathfrak{k}\) and, consequently, on \(L\). In particular, the latter will clearly be \(G\)-invariant and have positive scalar curvature.

We now follow the notations already introduced in Section \ref{LStolzseq}. Let $M$ be a spin stratified pseudomanifold of depth $1$ with link $L\simeq G/K$. If $\beta M$ is its compact depth-$1$ stratum, we assume that the fibration of the boundary of its resolution $\pi_r\colon \partial M_r \to \beta M$, with fibre diffeomorphic to $L$, comes from a principal $G$-bundle $P$ over $\beta M$. More specifically, assume that \(\pi_r\) can be described as an associated bundle to \(P\), that is:

\[
\pi_r\colon\partial M_r = P \times_G (G/K) \to \beta M
\]

A principal connection induces an associated connection on $\pi_r$, which induces a splitting of the tangent bundle of $\partial M_r$ as:

\[T(\partial M_r)=T_V(\partial M_r) \oplus \pi_r^*(T(\beta M)),\]
where the vertical tangent bundle is given by:

\[T_V(\partial M_r) = P \times_G (TL) = P \times_G (G \times_K \mathfrak{g}/\mathfrak{k})=P \times_K \mathfrak{g}/\mathfrak{k}.\]

Similarly, the fibration of the tubular neighborhood \(\pi\colon N(\beta M) \to \beta M\), with fiber diffeomorphic to the cone \(c(L)\), can be described using the action of \(G\) on the cone that preserves the radial coordinate on it. In this way, a connection is also associated with this bundle.

With such connection, if $g_L$ is a fixed Riemannian metric on $L$ built as described before, then one obtain on $N(\beta M)$ a metric of the form $(dr^2+r^2g_{\partial M_r / \beta M})\oplus \pi^*(g_{\beta M})$, with $g_{\beta M}$ a Riemannian metric on $\beta M$. Observe that $g_{\partial M_r / \beta M}$, as already remarked in Section \ref{secwedge}, is a metric in the vertical tangent bundle of $\pi_r$. However, in this case, we assume that on each fiber it restricts to the same metric and then we are denoting it with a small abuse of notation simply by $g_L$. In this way, each vertical fiber of $\pi$ is totally geodesic and with scalar curvature:
\[ k_{c(L)} = \frac{k_L - k_l}{r^2}, \quad k_l = l(l-1), \quad l = \dim(L), \]
where $r$ is a radial coordinate along the cone and $k_l$ is the scalar curvature of the standard $l$-sphere of radius $1$.

\begin{definition}
	A smoothly stratified space $M$ of depth-$1$ is a \textbf{pseudomanifold with (L,G)-singularities} if the link $L$ is given by homogeneous space $G/K$ and the fibration on the boundary of its resolution $\pi_r$ is an associated bundle to a a principal bundle as described above.
\end{definition}

\begin{definition}\label{welladaptedwedgemetric}
	Let $M$ be a pseudomanifold with $(L,G)$-singularities, then a \textbf{well-adapted wedge metric} $g$ is the following datum.
	
	\begin{itemize}
		\item A Riemannian metric $g_r$ on the resolution $M_r$ which is of product type on a collar neighborhood of the boundary;
		\item On the tubular neighborhood has the form described before, meaning that the vertical metric $g_{\partial M_r/\beta M}$ induces a fixed $g_L$, suitably scaled such that it has constant scalar curvature equal to $k_l=l(l-1)$, with $l=dim(L)$ (recall Remark \ref{killing});
		\item On a neighborhood of the boundary $\partial M_r$ in $N(\beta M)$, the metric transitions smoothly to a product type metric: 
		\[
			dr^2 + Cg_{\partial M_r/\beta M} \oplus \pi_r^*(g_{\beta M}),\]
		with $C$ a positive constant.
	\end{itemize} 
\end{definition}

If the second condition of the definition is omitted, thereby allowing the metric $g_{\partial M_r/\beta M}$ on the vertical tangent bundle to depend on the point of the stratum \(\beta M\), then \(g\) is simply said to be an \textit{adapted wedge metric}.

\begin{remark}\label{transition}
	The smooth transition of the above definition can be described as follows. Let $\partial M_r \times [0,\epsilon]$ be the neighborhood in $N(\beta M)$, where $\partial M_r$ is identified with $\partial M_r \times \{\epsilon\}$ and $\partial M_r \times \{0\}$ corresponds to those values in the cone bundle with radial coordinate equal to $R$. Then we ask the metric to be locally expressed by: 
	 \[dr^2+f(r)^2g_{\partial M_r/\beta M}+\pi^*g_{\beta M},\]
	 for a suitable smooth function $f$ defined on $r \in [0,\epsilon]$ such that around $0$ it has the form $R+r$, while it is equal to the constant $R + \epsilon/2$ on a neighborhood of $\epsilon$.
\end{remark}

Given these choices, in \cite[Proposition 3.4]{BPR1} is described a result concerning the scalar curvature of a space \(X\), equipped with a Riemannian submersion \(X \to B\) with fiber a totally geodesic space \(F\), in terms of the scalar curvature of \(B\), \(F\), and the O'Neill \(T\)-tensor of the fibers. Since this tensor, under our hypotheses, vanishes, the following result is proven.

\begin{theorem}[\cite{BPR1}, Theorem 3.5] \label{wedgethm}
	Let $L=G/K$ be a homogeneous space, where $G$ is a connected, semisimple Lie group, and equip it with a $G$-invariant metric with scalar curvature $k_L=l(l-1)$, where $l=dim(L)$. Then, if $M$ is a stratified pseudomanifold with $(L,G)$-singularities, the conical fibres $c(L)$ are scalar flat. Moreover:
	
	\begin{itemize}
		\item $\partial M_r$ always admits a well-adapted wedge metric of positive scalar curvature;
		\item if its depth-$1$ stratum $\beta M$ admits a metric with positive scalar curvature, then its tubular neighborhood $N(\beta M)$ has a well-adapted wedge metric of positive scalar curvature;
		\item it $N(\beta M)$ has a well-adapted wedge metric of positive scalar curvature, then $\beta M$ has a metric of positive scalar curvature.
	\end{itemize}
\end{theorem}

Observe, in particular, that the normalization chosen for the metric \(g_L\) is necessary in order to prove the second point of this theorem, while the first point is proven by making a choice for the value of the radius of cones $R$ (see Remark \ref{transition}).

\subsection{The (L,G)-fibered Stolz groups}

In this section, we proceed to adapt the contents of the previous chapter to the newly introduced context and thus define a bordism theory for spaces with \((L, G)\)-singularities. To do this, we will make some modifications to the definitions given in Sections \ref{secbordismgroups} and \ref{LStolzseq}, requiring that the links be homogeneous spaces of the form just discussed and that the wedge metrics be well-adapted.

Firstly, all cycles are given by compact spin pseudomanifolds with \((L, G)\)-singularities, i.e., the link \(L=G/K\) is fixed. Referring to the bordism group \(\Omega^{spin, L-fib}_*\) introduced in \ref{secbordismgroups}, what needs to be added is that the fibrations of the tubular neighborhoods of the strata of depth 1 must be described as associated bundles. We make this requirement by recalling that any \(G\)-principal bundle $P$ can be described via a classifying map $f$ from the base space to a universal space \(BG\): then the bundle will be isomorphic to the pullback bundle of a universal bundle $EG \to BG$, i.e. $P \simeq f^*(EG)$. Similarly, the associated bundles with fiber $L$ have a universal bundle. In particular, any associated bundle $E \to B$ with fiber $L$ is simply described via a map $f\colon B \to BG$ as:

\begin{equation} 
	\begin{tikzcd}
		f^*(EG\times_G L) \simeq E  \arrow[r]  \arrow[d] & EG \times_G L \arrow[d] \\
		B \arrow[r, "f"] & BG
	\end{tikzcd}
\end{equation}

Observe that if, instead of considering fibers \(L\), we consider the cones \(C(L)\) with a \(G\)-action that preserves the radial coordinate, then associated bundles are obtained by pulling back from the universal bundle \(EG \times_G C(L) \to BG\). 

In the cases of our interest, that is, those where the fibrations with fibers \(C(L)\) and \(L\) are that of the tubular neighborhood and its restriction to the boundary of the resolution, it is clear that they determine each other reciprocally.

Consequently, we will define the group \(\Omega^{spin,(L,G)}_\ast\) simply by adding the condition that a bordism between \(M\) and \(M'\) is realized by a space \(W\) whose tubular neighborhood \(N(\beta W) \to \beta W\) is described by a map \(f_W \colon \beta W \to BG\) such that this restricts on the boundary (given by the disjoint union of the two depth 1 strata of $M$ and $M'$) to the maps that classify the fibrations of the tubular neighborhoods of \(M\) and \(M'\).

Now we move on to the case of the bordism groups $Pos$ and $R$, which involve the metric structure in their description. In this case, we want the wedge metrics to be well-adapted, and thus, we first fix the metric \( g_L \) on the link \( L \) so that it has scalar curvature equal to that of the \( l \)-dimensional sphere of radius 1, with \( l = \text{dim}(L) \).

In this situation, a well-adapted wedge metric is thus fully described by a metric \( g_{M_r} \) on the resolution, which we recall to be product-like in a neighborhood of its boundary, a metric \( g_{\beta M} \) on the depth 1 stratum, and a principal connection. In this sense, in order to define the groups $Pos^{spin,(L,G)}_*$ and $R^{spin,(L,G)}_*$,  it is sufficient to additionally require that a bordism between \( M \) and \( M' \) is realized through a space equipped with a well-adapted wedge metric such that these three objects determining the metric restrict on the boundary to those of \( M \) and \( M' \). Thus, we obtain the following result, which is a straightforward adaption of Theorem \ref{Lstolzexact}.

\begin{theorem}\label{l,gstolzseq}
	Let X be a topological space and $L$ be as before. Then the \textbf{(L,G)-stratified Stolz sequence}:
	
	\begin{equation}
		\xymatrix{
			\ldots \ar[r] & R^{spin, (L,G)}_{n+1}(X) \ar[r]^{\partial} & Pos^{spin, (L,G)}_{n}(X)\ar[r]^{\varphi} & \Omega^{spin, (L,G)}_{n}(X)\ar[r]^{\iota} & R^{spin, (L,G)}_{n}(X)\ar[r] & \ldots},
	\end{equation}
	
	is exact.
	
\end{theorem}

\subsection{Invariance of the (L,G)-R-groups under 2-connected maps}

We now are going to state a result similar to Theorem \ref{G2conn} for the just introduced $R$-groups. In its proof, we will use the results concerning the Morse theory for stratified spaces obtained in the previous sections (in particular regarding the CW-complex structure that follows) and the Gromov-Lawson theorem stated below in a formulation that, along with Theorem \ref{wedgethm} on well-adapted wedge metrics, will allow us to extend a metric along a bordism (see \cite{EbertFrenck}).

\begin{theorem}\label{gromovlawson}
	Considering a Riemannian manifold $(M,g)$, where $g$ is a positive scalar curvature metric, then if $M'$ is another manifold obtained from M by surgeries of codimension $\geq 3$, then:
	\begin{itemize}
		
		\item one can construct a psc metric on M';
		\item in particular, one can construct a psc metric on the bordism between M and M' (the trace of the surgery) which reduces to a product metric on the boundary.
	\end{itemize}
\end{theorem}

Recall that a continuous function \( f: X \to Y \) is called \( n \)-connected if it induces isomorphisms for all homotopy groups of order \( k < n \), while it induces a surjective map for the \( n \)-th homotopy group. Furthermore, remember that the introduced bordism groups behave functorially with respect to their reference spaces, in the sense that if \( f: X \to Y \), then there exists a map:

\[ f_*: \Omega^{\text{spin},(L,G)}_n(X) \to \Omega^{\text{spin}, (L,G)}_n(Y), \]

obtained simply by composing the reference map of a bordism class with \( f \) (the same for $Pos$ and $R$ groups).

\begin{theorem} \label{2conrgroup}
	Let $X, Y$ be two CW-complexes such that $f: X \to Y$ is a continuous, 2-connected map. Then the functorially induced map on the $R$-groups:
	
	\[f_*\colon R_n^{spin, (L,G)}(X) \to R_n^{spin, (L,G)}(Y),\]
	
	is an isomorphism.
\end{theorem}

\begin{proof}
	The proof will proceed in a manner very similar to that of Theorem \ref{G2conn}.
	
	Assume, without loss of generality, that both $X$ and $Y$ are connected and $Y=B\pi_1(X)$. We recall that, since $\pi_1(X)$ is a discrete group, $B\pi_1(X)$ has the homotopy type of an Eilenberg-MacLane space $K(\pi_1(X), 1)$, i.e.:
	\begin{equation}
	 \begin{cases}
	 	\pi_k(B\pi_1(X))\simeq 0, \quad k\neq 1 \\
	 	\pi_1(B\pi_1(X))\simeq \pi_1(X)
	 \end{cases}
	 \end{equation}
	 
	Moreover, $B\pi_1(X)$ is a final object, in the sense that for each space $Z$ with fundamental group $\pi_1(X)$, it is the target of a unique $2$-connected map $f_Z\colon Z \to B\pi_{1}(X)$ up to homotopy.
	
	Now we prove the surjectivity of $(f_X)_*$.
	
	Let $[M]=(M, \partial M, g_{(\partial M)_r}, g_{\beta(\partial M)}, \nabla_{\partial M}, f: M \rightarrow B\pi_1(X)) \in R_n^{spin, (L,G)}(B\pi_1(X))$.
	Proving the surjectivity of $(f_X)_*$ is equivalent to find a map $h\colon M \to X$ giving the following factorization:
    
    \[
    \xymatrix{M \ar[r]^{f} \ar[d]^{h}& B\pi_1(X) \\ X \ar[ur]^{f_X}}
    \]
    
    Let us consider $M$ as a bordism in $R_{n-1}^{spin, (L,G)}(B\pi_1(X))$ between $\partial M$ and the empty set: then take a self-indexing, special Morse pair $(\varphi, h)$ (recall Definition \ref{specialmorse}) on $M$ seen as a triple $(M, \partial M, \emptyset)$. A regular value $t_0 \in \mathbb{R}$ allows us to split $M$ in two bordisms:
    
     \[M_1:=\varphi^{-1}(\{t \leq t_0\}), \quad M_2:=\varphi^{-1}(\{t \geq t_0\}),\]
     such that $M= M_1 \cup M_2$ and $S=M_1 \cap M_2 \cong \varphi^{-1}(t_0)$.
     In particular, since $\varphi$ is self-indexing, we can choose $t_0$ suitably such that $M_1$ consists only on cells up to dimension $2$.
     
	In this way, we can apply the second point of the Gromov-Lawson Theorem $\ref{gromovlawson}$, which allows us to extend the psc metric $g_{\beta(\partial M)}$ along the depth-1 stratum of $M_2$ and then obtaining a well adapted wedge metric of positive scalar curvature on its tubular neighborhood using Theorem \ref{wedgethm}. Then one can extend the psc-metric on the interior of the resolution $(\partial M)_r$, obtaining a well-adapted wedge metric of positive scalar curvature on $M_2$.
	
	It follows that $S=\partial M_1$ gets a psc well-adapted wedge metric and that $M \times [0,1]$ provides a bordism between $M$ and $M_1$ in $R_{n}^{spin, (L,G)}(B\pi_1(X))$.

 Consider then $f|_{M_1}$, i.e. the reference map which defines the class $[M_1]$. $f$, by construction is a mapping from a space consisting only on cells up to dimension $2$ attached to the empty set, i.e. homotopy equivalent to a $2$-dimensional $CW$-complex.   
 
 Then, we can apply the arguments cited in the proof of \ref{G2conn}, simply adapting them to our situation, thereby first obtaining the following factorization:

	\[
\begin{tikzcd}
	M_1 \arrow[r,"f|_{M_1}"] \arrow[d, "k"] & B\pi_1(X)   \\
	B\pi_1(X)_{(2)}  \arrow[ur,"i"]
\end{tikzcd}
\]
and then, since $f_X$ is a $2$-connected map, we obtain the following:

\[
\begin{tikzcd}
	M_1 \arrow[r,"f|_{M_1}"] \arrow[d, swap, "k"] & B \pi_1(X)  \\
	(B\pi_1(X))_{(2)} \arrow[r,"V"] \arrow[ur, "i"] &  X_{(2)} \arrow[u, swap, "f_X|_{X_{(2)}}"]
\end{tikzcd}
\]

Finally, the composition $V \circ k$ provides the required factorization proving the surjectivity of $(f_X)_*$.

Now let's proceed to prove the injectivity of \((f_X)_*\).

Let us take a class $[N]=(N, \partial N, g_{(\partial N)_r}, g_{\beta(\partial N)}, \nabla_{\partial N}, f: N \rightarrow X) \in R_n^{spin, (L,G)}(X)$ such that $f_*([N])=0$ in $R_n^{spin, (L,G)}(B\pi_1(X))$. This means that there exists a stratified $(n+1)$-pseudomanifold with $(L,G)$-singularities with corners $(W, \partial W, \partial^2 W)$ with:

\[ \partial W = N \cup N', \quad \partial^2 W= \partial N.\] 

 In particular, $N'$ realizes a bordism in $Pos_{n-1}^{spin, (L,G)}$ between $\partial N$ and the empty set. Moreover, a well-adapted wedge metric of positive scalar curvature is defined on the whole $N'$ extending that of $\partial N$. 
 
 On $W$ there is a mapping $F\colon W \to B\pi_1(X)$ which restricts to $f$ on $N$. Therefore, $F$ factorizes to $X$ on $N$. What will allow us to prove the injectivity of \((f_X)_*\) is finding a bordism of \(N\) with the empty set or, equivalently, to a space representing the zero element in $R_n^{spin, (L,G)}(B\pi_1(X))$, realized by a space in which this factorization is achieved everywhere.

Let us now consider the space $W$ as a triple $(W, N, N')$ with vertical boundary given by a collar neighborhood $\partial N \times [0,1]$.
Then, a self-indexing, special Morse pair $(\varphi,h)$ defined on it will have only critical points defined on the interior, meaning that each regular level set will be a stratified pseudomanifold with $(L,G)$-singularities with boundary $\partial N$.

Similarly as before, choose a suitable regular value $t_0$ of $\varphi$ which splits the bordism $W$ in two bordism:
 \begin{itemize}
 	\item $W_1$, from $N$ to $S:=\varphi^{-1}(t_0)$;
 	\item $W_2$, from $S$ to $N'$.
 \end{itemize} 

In particular, the latter will have critical points associated to surgeries of codimension $\geq 3$. Therefore, we can apply the second point of the Gromov-Lawson theorem \ref{gromovlawson} to push the well-adapted wedge metric of positive scalar curvature to $S$.

Then, we observe that $W_1$ is realized as the pushout of the inclusions of $N$ and the union of all its closed cells up to dimension $2$ attached on its interior, say $Z$. 
Since, by construction, $Z$ is homotopy equivalent to a $2$-dimensional CW-complex, we can apply to $F|_Z$ the same argument used to prove the surjectivity obtaining a factorization:

	\[
\begin{tikzcd}
	Z \arrow[r,"F|_Z"] \arrow[d, ""] & B\pi_1(X)   \\
	X  \arrow[ur,"f_X"]
\end{tikzcd}
\]

Finally, by the universal property of the pushout, there exists a unique map on $W_1$ which factorizes over $f_X\colon X \to B\pi_1(X)$.

 Then, $W_1$, with this map, realizes a bordism to the zero class in $R_n^{L,G-fib}(X)$ and then $[N]=0 \in R_n^{spin, (L,G)}(X)$ proving the injectivity of $(f_X)_*$

\end{proof}

	\newpage

\section{Mapping (L,G)-Stolz to Higson-Roe} \label{mappingstolzlg}

\subsection{Dirac operators}\label{secwedgediracoperator}

Let us now recall, following \cite{LM}, the basic notions about $Cl_n$-linear Dirac operators, and subsequently propose an extension of these to the case of smoothly stratified spaces.

Firstly, if \((V, q)\) is a pair consisting of a vector space \(V\) over a commutative field and its quadratic form \(q\), then the \textit{Clifford algebra} \(Cl(V, \eta)\) is given by the quotient of the tensor algebra \(\mathcal{T}_V=\sum_k \bigotimes^k V\) with respect to an ideal, that is:

\[
Cl(V, q) := \mathcal{T}_V / \langle v \otimes v + \eta(v)1,  v \in V \rangle
\]

This algebra has a natural \(\mathbb{Z}_2\)-graded algebra structure:
\[Cl(V,q)=Cl^0(V,q)\oplus Cl^1(V,q),\] 
where $Cl^i(V,q)$ are given by the eigenspaces of the involutive automorphism in \(Cl(V, \eta)\) induced by the sign change in \(V\).

\begin{remark}
In the case where \(V = \mathbb{R}^{p+q}\) and \(\eta\) is in standard form with signature \((p, q)\), we will refer to the corresponding Clifford algebras as \(Cl_{p, q}\). In this situation, \(Cl_{p, q}\) is the algebra generated by an orthonormal frame \(\{e_1, \ldots, e_{p+q}\}\) subject to the relations:

\[
e_i e_j + e_j e_i = 
\begin{cases} 
	-2\delta_{ij}, & \text{if } i \leq p \\ 
	2\delta_{ij}, & \text{if } i > p 
\end{cases}
\]
\end{remark}

Within the Clifford algebra, it is possible to define its \textit{spin group} \( Spin(V, q) \) as the subgroup of the group of invertible elements with respect to the product of the algebra, consisting of elements in \( Cl^0(V, q) \) and generated by \( v \in V \) such that \( q(v) = \pm 1 \).

The spin group can be also defined as the double universal cover of the special orthogonal group \( SO(V) \). This introduces the definitions of a spin structure associated to the oriented orthonormal frame bundle of a (Riemannian, oriented) vector bundle $E\to M$ as an equivariant lift of it via the covering map.

In particular, if \( P_{spin}(E) \to M \) denotes such a spin structure, then the \textit{Clifford bundle} \( Cl(E) \) can be defined as:

\[ Cl(E) := P_{spin}(E) \times_{Ad} Cl_n \to M \]

where \( n \) is the rank of \( E \) and $Ad$ denotes the action of the spin group $Spin_n = Spin(\mathbb{R}^n)$ on $Cl_n$ made by $Ad_g(\varphi)=g\varphi g^{-1}$. Observe that the grading of $Cl_n$ induces one on $Cl(E)$.

In case $M$ is a Riemannian manifold, then the Clifford bundle associated to its tangent bundle is also denoted by $Cl(M)$. Then, any bundle of left module over $Cl(M)$ (whose action is fiberwise) endowed with a Riemannian metric and a connection $\nabla$ such that the action of unit vectors in $TM$ (seen as sections of $Cl(M)$) is orthogonal and its covariant derivative is a module derivation, is called a \textit{Dirac bundle}. The action of $Cl(M)$ on a Dirac bundle is called \textit{Clifford multiplication}, denoted by $c$.

\begin{definition}
	Consider a Dirac bundle $S$ over a Riemannian manifold $M$, then we define the \textbf{Dirac operator} as the operator $D\colon \Gamma(S) \to \Gamma(S)$ given by:
	
	\[D\sigma:=\sum_i c(e_i) \cdot \nabla_{e_i}\sigma,\]
	
	where $\{e_i\}$ is an orthonormal frame of $TM$, $\cdot$ denotes the Clifford multiplication.
\end{definition}

Now, consider an n-dimensional spin manifold $M$ and $P_{spin}(M)$ its spin structure. Then, given $l\colon Spin_n \to Hom(Cl_n)$ defined by left multiplication, the following associated bundle is defined:

\[\slashed{\mathfrak{S}}(M):=P_{spin}(M) \times_l Cl_n \to M\]

This has a structure of a right module bundle of rank 1 with respect to $Cl_n$, with the operation of obvious right multiplication, since it commutes with $l$. Moreover, an obvious action of $Cl(M)$ commuting with that of $Cl_n$ is defined on $\slashed{\mathfrak{S}}(M)$. Since it is an associated bundle of $P_{spin}(M)$, it carries a canonical Riemannian connection: it can be proven that $\slashed{\mathfrak{S}}(M)$ is then an example of Dirac bundle over $M$.

The action of $Cl_n$ is parallel with respect to the connection, so in particular, the $\mathbb{Z}_2$-grading of $Cl_n$ induces one on $\slashed{\mathfrak{S}}(M)$, namely $\slashed{\mathfrak{S}}(M)=\slashed{\mathfrak{S}}^0(M) \oplus \slashed{\mathfrak{S}}^1(M)$ such that:

\[ \slashed{\mathfrak{S}}^i(M) \cdot Cl^j_n \subseteq \slashed{\mathfrak{S}}^{i+j}(M), \quad Cl(M)^i \cdot \slashed{\mathfrak{S}}^{j}(M) \subseteq \slashed{\mathfrak{S}}^{i+j}(M)\]

The Dirac operator on $\slashed{\mathfrak{S}}(M)$ is, by construction, an odd operator with respect to such grading which commutes with the right $Cl_n$-action.

\begin{definition}
	The Dirac operator just discussed on the Dirac bundle $\slashed{\mathfrak{S}}(M)$, denoted $\slashed{\mathfrak{D}}$, is called the \textbf{Cl$_n$-linear Atiyah-Singer operator} of $M$.
\end{definition}

\begin{remark}
	It can be shown that, with respect to the inner product on $\Gamma(S)$ induced by the Riemannian metric on $S$, the Dirac operator is formally self-adjoint for sections with compact support. If we consider the space of $L^2$ sections of $S$, denoted by $L^2(S)$ and defined as the completion, with respect to the above mentioned inner product, of the space $\Gamma_c(S)$ of compactly supported sections of S, we can initially consider $D$ as a symmetric operator on $\Gamma_c(S)$, and then take its closure in $L^2(S)$. This results in an unbounded operator on $L^2(S)$. It is a standard fact that if $M$ is a complete Riemannian manifold, then the closure of the operator $D$ in $L^2(S)$ is self-adjoint, i.e. $D$ is essentialy self-adjoint (see \cite[Theorem 5.7]{LM}).
\end{remark}	

It is a well-known fact that for operators on Hilbert spaces $H$, it is possible to define the so-called \textit{Borel functional calculus}. Given an unbounded self-adjoint operator on $H$, this allows for the definition of a bounded operator $f(T)$ on $H$ for any bounded Borel function $f$ on $\mathbb{R}$, such that the correspondence $f \to f(T)$ is a ring homomorphism that respects the involutions. In particular, by elliptic regularity, the following holds.

\begin{proposition} \label{compactfunctcalc}
	For a Dirac operator $D$ on a compact (and therefore complete) manifold, if $f$ is a continuous function vanishing at infinity, that is $f \in C_0(\mathbb{R})$, then the operator $f(D)$ is compact.
\end{proposition}	

This allows us to define the index of a Dirac operator. The notion of index is abstractly defined for Fredholm operators, which are bounded operators on Hilbert spaces that are invertible modulo compact operators. For a Fredholm operator $T$, both $Ker(T)$ and $Coker(T)$ are finite-dimensional, and thus the index is defined as 
\[ind(T) = \dim(Ker(T)) - \dim(Coker(T)) \in \mathbb{Z}\]

Choosing an appropriate continuous function $\chi$, called a \textit{chopping function}, it is shown that on a compact manifold, $\chi(D)$ is a Fredholm operator, and this is independent of the choice of $\chi$ up to compact operators. However, since $\chi(D)$ is a self-adjoint operator and its index would consequently be zero (in fact, $\dim(Coker(T))=\dim(Ker(T^*))$), the grading of the bundle on which it is defined and the fact that $D$ is an odd operator with respect to this grading are used. Therefore, the index of $D$ is defined as the index of the restriction of the operator $\chi(D)$ to the even component.

Unfortunately, and this will be our case, on a non-compact manifold, Proposition \ref{compactfunctcalc} does not hold. However, it can be shown that the operator $f(D)$ is locally compact. Moreover, in our context, we cannot even be sure of dealing with essentially self-adjoint operators due to the presence of incomplete metrics.

\subsection{Roe algebras}\label{secroealgebras}

Let $X$ be a smoothly stratified space of depth $1$ and $g$ be a wedge metric on $X$. Recall that such metric is Riemannian in the regular set of $X$: however, each stratum has a metric and all those fit together continuously. In particular, $X$ can be topologized as a metric space with distance given by the infimum over all rectificable curves joining two points. With such distance, $X$ becomes a complete proper metric space in the sense of Pflaum (\cite[Theorem 2.4.17]{Pflaum}).

Now, let $\Gamma$ be a discrete group and consider an arbitrary Galois cover $\pi \colon X_{\Gamma} \to X$. Since $\pi$ is a local homeomorphism, on $X_{\Gamma}$ a stratification is induced by the covering considering the preimages of the strata of $X$. In particular, $X_{\Gamma}$ becomes a smoothly stratified space of the same depth, whose link over a point $p \in X_{\Gamma}$ is diffeomorphic to the one over $\pi(p) \in X$.

With the aim of introducing the Higson-Roe sequence, we will define some C*-algebras of operators, specifically the Roe algebras, defined on suitable modules with respect to functions on a proper metric space.

Let $C_0(Y)$ be the $C^*$-algebra of continuous functions on a locally compact space which vanish at infinity. Of course, we can consider $C_0(X_{\Gamma})$ since $\pi$ is a local homeomorphism and then locally compactness is preserved. 

\begin{definition}
	A \textbf{covariant $X_{\Gamma}$-module} is a Hilbert space $H$ with a $\ast$-representation of $C_{0}(X_{\Gamma})$, i.e. a $\ast$-homomorphism $\rho \colon C_{0}(X_{\Gamma}) \to \mathcal{L}(H)$ and a representation of $\Gamma$ by unitaries on $H$ compatible with $\rho$ in the following sense:
	
	\[ \rho(f \circ \gamma)=\gamma \rho(f) \gamma^{-1}, \quad \forall \gamma \in \Gamma, \quad \forall f \in C_0(X_{\Gamma})\]
\end{definition}

Moreover, we require that the module $H$ is \textit{ample}, in the sense that the $\ast$-representation $\rho$ is non degenerate and $\rho(f)$ is not a compact operator except for $f=0 \in C_0(X_\Gamma)$. In the following definitions, we will simply denote by $f$ the operator induced by the representation $\rho$, omitting the notation $\rho(f)$.

\begin{definition}
Let $T \in \mathcal{L}(H_{X},H_{Y})$ be a linear, bounder operator between an $X_{\Gamma}$-module $H_X$ and a $Y_{\Gamma}$-module $H_Y$.
\begin{enumerate}[label=(\roman*)]
	\item The \textbf{support} of $T$ is the subset $supp(T) \subset X_{\Gamma} \times Y_{\Gamma}$ such that:
	 \[(x,y)\notin supp(T) \Longleftrightarrow  gTf=0 \in \mathcal{L}(H_X,H_Y),\] 
	 for some $f \in C_0(X_{\Gamma})$ and $g \in C_0(Y_{\Gamma})$ such that $f(x)\neq0$ and $g(y)\neq0$.

\item The \textbf{propagation} of an operator $T\in \mathcal{L}(H_X)$ is the minimum value $R\geq0$ such that $supp(T)$ is contained in an $R$-neighborhood of the diagonal in $X_{\Gamma} \times X_{\Gamma}$.

\item $T \in \mathcal{L}(H_X)$ is said \textbf{locally compact} if the operators $fT$ and $Tf$ are compact operators, while it is \textbf{pseudolocal} if $[f,T]$ is compact, for any $f \in C_0(X_{\Gamma})$.

\item Given a closed subset $Z \subset X_{\Gamma}$, we say that $T \in \mathcal{L}(H_X)$ is \textbf{supported near Z} if $supp(T) \subseteq B_R(Z) \times B_R(Z)$, where $B_R(Z)$ denotes an $R$-neighborhood of $Z$ in $X_{\Gamma}$, for some $R \geq 0$.

\end{enumerate}
\end{definition}

Having introduced the necessary nomenclature, we will now proceed to define the Roe algebras.

\begin{definition}
	Consider an ample, covariant $X_{\Gamma}$-module $H_X$ and a closed, subset $Z \subset X_\Gamma$, then we define:
	\begin{itemize} 
		\item $D^{\ast}_c(X_{\Gamma}, H_X)$ as the $C^{\ast}$-subalgebra of $\mathcal{L}(H_X)$ made by pseudo-local operators of finite propagation;
		\item $C^{\ast}_c(X_{\Gamma}, H_X)$ as the $C^{\ast}$-algebraic ideal in $D^{\ast}_c(X_{\Gamma}, H_X)$ of operators which are locally compact. 
		\item $D^{\ast}_c(Z\subset X_{\Gamma}, H_X)$ and $C^{\ast}_c(Z \subset X_{\Gamma}, H_X)$ as the respective ideals in the above $C^{\ast}$-algebras given by operators which are supported near $Z$.
	
	\end{itemize}	

These algebras are not complete; therefore, their norm closures are considered. These closures are denoted by omitting the subscript "c", thus $D^{\ast}(X_{\Gamma}, H_X)$, $C^{\ast}(X_{\Gamma}, H_X)$, $D^{\ast}(Z\subset X_{\Gamma}, H_X)$ and $C^{\ast}(Z \subset X_{\Gamma}, H_X)$
\end{definition}

\begin{remark}	
Certainly, all these definitions are valid for any $X$-modules with respect to proper metric spaces $X$ endowed with a proper, free action of a discrete group $\Gamma$ by isometries. Moreover, observe that we haven't used yet the $\Gamma$-action in the above definitions. In fact, the above can be defined also when $\Gamma$ is the trival group. 

However, it is possible to define the corresponding equivariant algebras. The unitary representation of $\Gamma$ on $H_X$ induces an adjoint representation on both $D^{\ast}_c(X_{\Gamma}, H_X)$ and $C^{\ast}_c(X_{\Gamma}, H_X)$. Then, we define $D^{\ast}(X_{\Gamma}, H_X)^{\Gamma}$, $C^{\ast}(X_{\Gamma}, H_X)^{\Gamma}$, $D^{\ast}(Z\subset X_{\Gamma}, H_X)^{\Gamma}$ and $C^{\ast}(Z\subset X_{\Gamma}, H_X)^{\Gamma}$ (where now $Z$ is required to be also $\Gamma$-invariant) as the closure of the $\Gamma$-invariant parts of the above algebras with respect to this action.
\end{remark}

These algebras are functorial with respect to the class of continuous and coarse functions. Recall that a function $f: X \to Y$ between two proper metric spaces is called \textit{coarse} if for every $R > 0$ there exists an $S > 0$ such that the image of every $R$-ball is contained in an $S$-ball, and furthermore, the preimage of any bounded set is bounded.

Functoriality on algebras is given by conjugation with respect to an isometry which suitably covers the map between the two metric spaces. 

Refering to the $\Gamma$-equivariant case, a continuous, $\Gamma$-equivariant (for the standard case simply consider $\Gamma=\{e\}$), coarse map $f\colon X_{\Gamma} \to Y_{\Gamma}$ induces mappings:

 \[D^{\ast}(X_{\Gamma}, H_X)^\Gamma \to D^{\ast}(Y_{\Gamma}, H_Y)^\Gamma, \quad C^{\ast}(X_{\Gamma}, H_X)^\Gamma \to C^{\ast}(Y_{\Gamma}, H_Y)^\Gamma,\] 
 for $H_X$ and $H_Y$ ample modules as above. In particular, such maps become canonical on the $K$-theory level and then the $K$-theory groups do not depend on the choice of the modules (see \cite[Lemma 3]{HRY}). Because of this, we will denote the above algebras omitting the module $H$.

Another application of functoriality appears when $Z \subset X_\Gamma$ is a closed $\Gamma$-subset, and then the inclusion $Z \hookrightarrow X_\Gamma$ induces the following isomorphisms.

\begin{lemma}[\cite{HRY}, Section 5, Lemma 1]\label{inclisomk}
	In the above hypothesis, the inclusion functorially induces isomorphisms on the $K$-theory level, i.e.:
	\[K_{\ast}(D^{\ast}(Z)^{\Gamma}) \simeq K_{\ast}(D^{\ast}(Z \subset X_\Gamma)^{\Gamma}), \quad K_{\ast}(C^{\ast}(Z)^{\Gamma}) \simeq K_{\ast}(C^{\ast}(Z \subset X_\Gamma)^{\Gamma})\]
\end{lemma}

Since $C^{\ast}(X_{\Gamma})^{\Gamma}$ is by construction an ideal, we obtain the following short exact sequence:

	\[
	\begin{tikzcd}
		0 \arrow[r] & C^{\ast}(X_{\Gamma})^{\Gamma} \arrow[r] & D^{\ast}(X_{\Gamma})^{\Gamma} \arrow[r] & D^{\ast}(X_{\Gamma})^{\Gamma} / C^{\ast}(X_{\Gamma})^{\Gamma} \arrow[r] & 0
	\end{tikzcd}
	\]
	
The associated long exact sequence in $K$-theory is the well-known \textit{Higson-Roe surgery sequence}. However, assuming that the action of $\Gamma$ is free, then $K_{\ast}(D^{\ast}(X_{\Gamma})^{\Gamma} / C^{\ast}(X_{\Gamma})^{\Gamma}) \simeq K^{\Gamma}_{\ast-1}(X_{\Gamma})$, where $K^{\Gamma}_{\ast-1}(X_\Gamma)$ is the equivariant $K$-homology of $X_{\Gamma}$ (see \cite[Lemma 5.15]{Roeindextheory}) and one can obtains the following.

	\begin{equation}\label{higsonroeseq}
	\begin{tikzcd}[sep=small, row sep=1.5em]
		\ldots \arrow[r] & K_n(C^{\ast}(X_{\Gamma})^{\Gamma}) \arrow[r] & K_n(D^{\ast}(X_{\Gamma})^{\Gamma}) \arrow[r] & K^{\Gamma}_{n-1}(X_{\Gamma}) \arrow[r, "\mu"] & K_{n-1}(C^{\ast}(X_{\Gamma})^{\Gamma}) \arrow[r] & \ldots
	\end{tikzcd}	
	\end{equation}
	
where $\mu \colon K^{\Gamma}_{n-1}(X_{\Gamma})  \to K_{n-1}(C^{\ast}(X_{\Gamma})^{\Gamma})$ is the \textit{coarse assembly map}. 

If $X/\Gamma$ is a finite complex, i.e. the quotient is compact, one can use the $K$-homology of $X/\Gamma$ instead of $K^{\Gamma}_{\ast-1}(X)$.

\begin{remark} \label{roecoefficients}
	All the above can certainly be generalized to the context in which Hilbert $A$-modules (with $A$ a $C^*$-algebra) instead of Hilbert spaces and adjointable, Hilbert module maps (which, in particular, are $A$-linear with respect to the $A$-action) instead of bounded, linear operators are considered obtaining a description of Roe algebras \textit{with coefficients in $A$}. In this way, it is also possible to give a "real-analogue" of the description above, i.e., for real Hilbert spaces. The corresponding Roe algebras with coefficients in $A$ will subsequently be denoted as $C^{\ast}(X_{\Gamma}; A)^{\Gamma}$, $D^{\ast}(X_{\Gamma}; A)^{\Gamma}$...
\end{remark}

\subsection{KO-homology index class in the wedge setting}\label{k-homologyindexclass}

We now describe how an index class in equivariant $KO$-homology can be defined in the context of spin stratified pseudomanifolds. In partcular we will focus on the case of $(L,G)$-singularities, introduced in Section \ref{L,Gsetting}.

First of all, we need to clarify what kind of Dirac type operator we can consider starting from a wedge metric. In particular, we will follow the treatment contained in \cite{AlbinGell}.
	
Let then $M$ be a spin stratified pseudomanifold of depth $1$, endowed with a wedge metric $g$.
Consider its resolution $M_r$ as its blowup with the induced metric $g$ (see Remark \ref{wedgemetricresolution}) denote by $\leftindex^{w}TM_r$ its wedge tangent bundle.  Recall that a wedge metric $g$ can be considered a bundle metric on the wedge cotangent bundle. Thus, we can construct the \textit{wedge Clifford bundle} $\leftindex^{w}T^*M$ associated to this metric :

\[
\leftindex^{w}Cl(M_r,g) = \sum_k \otimes^k \left(\leftindex^{w}TM_r \right) / \langle v \otimes w + w \otimes v + 2 g(v,w) \rangle
\]
i.e. by taking the Clifford algebra on each fiber of $\leftindex^{w}TM_r$. Now, the orthonormal frame bundle of $\leftindex^{w}TM_r$ gives an extension of the orthonormal frame bundle of the interior of $M_r$ to the boundary (recall what said about the frame (\ref{wedgeframe}) in Section \ref{secwedge}), then denote by $\slashed{\mathfrak{S}}_g(M_r):= P_{Spin(n)}(M_{r}) \times_{l} Cl_n$ the bundle already introduced in Section \ref{secwedgediracoperator},  where $P_{Spin(n)}$ is the spin structure of the resolution of $M_r$, $Cl_n \equiv Cl_{n,0}$ denotes the real Clifford algebra and $l\colon Spin(n) \to Hom(Cl_n, Cl_n)$ the representation given by left multiplication. Again, $\slashed{\mathfrak{S}}_g(M_r)$ is $\mathbb{Z}_2$-graded and it has a canonical fiberwise right $Cl_n$-action which makes it a bundle of rank-$1$ right $Cl_n$-modules.

The Clifford multiplication extends smoothly to sections of $\leftindex^{w}TM_r$, meaning that:

\[c \colon \Gamma(\leftindex^{w}TM_r) \to \Gamma(End(\slashed{\mathfrak{S}}_g(M_r))),\]

Moreover, in \cite[2.1]{AlbinGell} performed a detailed study of the behaviour of the Levi-Civita connection of $g$ on $\leftindex^{w}TM_r$, especially on a collar neighborhood of the boundary of $M_r$. In particular, if $\nabla$ denotes the spin connection induced by this, we get the $Cl_n$-linear Atiyah-Singer operator $\slashed{\mathfrak{D}}^g$, acting on sections of $\slashed{\mathfrak{S}}_g(M_r)$ defined as usual as:

\[\slashed{\mathfrak{D}}^g=\sum_i c(e_i) \cdot \nabla_{e_i} \]

\begin{remark}
	We recall that $\slashed{\mathfrak{D}}_g$ is odd, commuting with the right $Cl_n$-action on $\slashed{\mathfrak{S}}_g(M_r)$. In particular, for this operator the \textit{Schrödinger-Lichnerowicz formula} holds:
	
	\[(\slashed{\mathfrak{D}}^g)^2=\nabla^*\nabla + \frac{1}{4}k_g,\]
where $k_g$ denotes the scalar curvature of $g$ and $\nabla^*\nabla$ is the connection Laplacian in which $\nabla^*$ denotes the formal adjoint of $\nabla$ (\cite[Theorem 8.8]{LM}).	
\end{remark}

Observe that the wedge metric $g$ induces a splitting on the fibration at the boundary of $M_r$ given by $T(\partial M_r / \beta M) \oplus \pi_r^*(T(\beta M))$, there $T(\partial M_r / \beta M)$ is the vertical tangent bundle, which consequently induces the following decomposition:

\[
\slashed{\mathfrak{S}}(\partial M_r) \simeq \slashed{\mathfrak{S}}(\partial M_r / \beta M) \widehat{\otimes} \slashed{\mathfrak{S}}(\beta M)
\]

In particular, omitting the metric $g$ in the notation, \cite[Lemma 2.2]{AlbinGell} states that on a collar neighborhood of the boundary (i.e. equivalently near the singular stratum) this operator has the following form:

\begin{equation}\label{wedgediracdec}
	\slashed{\mathfrak{D}}=c(\partial_r)\cdot\left(\partial_r + \frac{l}{2r}\right) + \frac{1}{r} \slashed{\mathfrak{D}}_{\partial M / \beta M} \widehat{\otimes} Id + Id \widehat{\otimes} \slashed{\mathfrak{D}}_{\beta M} + \mathfrak{B},\end{equation}

where $r$ is a boundary defining function of $\partial M_r$, $\mathfrak{B}$ indicates an endomorphism of $\slashed{\mathfrak{S}}(M_r)$, i.e. an element of $\Gamma(End(\slashed{\mathfrak{S}}(M_r)))$ and $\slashed{\mathfrak{D}}_{\partial M / \beta M}$ denotes the family of vertical Atiyah-Singer operators.

Factoring out $r^{-1}$ in (\ref{wedgediracdec}), it emerges that $\slashed{\mathfrak{D}} = r^{-1} \slashed{\mathfrak{D}}_e$, with $\slashed{\mathfrak{D}}_e$ being an edge differential operator. Thus, in particular, $\slashed{\mathfrak{D}}$ turns out to be a wedge differential operator of order $1$, which can be considered initially with domain equal to the compactly supported smooth sections $\Gamma_c(\slashed{\mathfrak{S}}(M_r))$ restricted to the interior $\mathring{M_r}$. In particular, $\slashed{\mathfrak{D}}$ can be seen as an unbounded operator in $\Gamma_c(\slashed{\mathfrak{S}}(M_r))|_{\mathring{M_r}} \subset L^2(\slashed{\mathfrak{S}}(M_r))$. 

The following result gives sufficient conditions to the existence of a self-adjoint extensions of this operator in $L^2(\slashed{\mathfrak{S}}(M_r))$.

\begin{theorem}\label{essentialyselfadj}
	[\cite{BPR1}, Theorem 3.5]
Let $M$ be a spin stratified pseudomanifold of depth $1$, $M_r$ be its resolution and $g$ be a wedge metric on $M$. Using the notations above, let $\slashed{\mathfrak{D}}_L$ indicate the generic operator induced on each fiber over each point of $\beta M$ by the vertical family $\slashed{\mathfrak{D}}_{\partial M / \beta M}$.  Then the following hold:
\begin{itemize}
	\item If for each fiber $L$ of the fibration $\partial M_r \to \beta M$,
	\begin{equation}\label{specl2}
		spec_{L^2}(\slashed{\mathfrak{D}}_L)\cap (-1/2,1/2)=\emptyset,
\end{equation}	
	 then the operator $\slashed{\mathfrak{D}}$  with domain $\Gamma_c(\slashed{\mathfrak{S}}(M_r))|_{\mathring{M_r}} \subset L^2(\slashed{\mathfrak{S}}(M_r))$ is essentialy self-adjoint and its unique self-adjoint extension defines a $Cl_n$-linear Fredholm operator;
	\item If the vertical metric $g_{\partial M/\beta M}$ induces a metric of positive scalar curvature on each vertical fiber $L$, then:
	\begin{equation}\label{specl22}
		spec_{L^2}(\slashed{\mathfrak{D}}_L)\cap (-\epsilon,\epsilon)=\emptyset,
	\end{equation}	
	meaning that, up to rescaling the wedge metric $g$ on the vertical tangent bundle of the boundary fibration, one can achieve condition (\ref{specl2}).
\end{itemize}

\end{theorem}

\begin{definition}
	Let $(M,g)$ be a spin stratified pseudomanifold of depth $1$ with a adapted wedge metric. We say that $M$ is \textbf{geometric Witt} if (\ref{specl2}) holds, while it is \textbf{psc-Witt} if the vertical metric $g_{\partial M_r/\beta M}$ induces psc metrics on the vertical fibers of the boundary of the resolution, and then satisfying condition (\ref{specl22}).
\end{definition}

\begin{remark}The main observation here is that, when dealing with spin pseudomanifolds with $(L,G)$-singularities, a well-adapted wedge metric always matches condition (\ref{specl22}) and then we can always assume the existence of such unique self-adjoint extension. 
\end{remark}	

Recall that these are in particular smoothly stratified spaces of depth $1$ where both the resolutions and the depth-$1$ stratum are spin manifolds such that the link over such stratum is a fixed homogeneous spaces $L=G/K$, with $G$ a connected, semisimple Lie group. 

We recall briefly the definition of $KO$-theory groups for a real unital $C^*$-algebra. 
Firstly, one define the following monoid:

\[V(A):=\{[p] \ \ | \ \ p=p^*=p^2 \in \mathbb{M}_\infty(A)\},\]
where $\mathbb{M}_\infty(A)$ denotes the $C^*$-algebra of square matrices with values in $A$ of infinite order, i.e. $\mathbb{M}_\infty(A):=\bigcup_{n \in\mathbb{N}} \mathbb{M}_n(A)$ and $[p]$ denotes an equivalence class of projections as follows. Two projections $p,q \in \mathbb{M}_\infty(A)$ are equivalent if there is a $v \in \mathbb{M}_\infty(A)$ such that $p=v^*v$ and $q=vv^*$.
The addition operation in $V(A)$ is given by:

\[[p]+[q]:=[diag(p,q)],\]
where $diag(p,q)$ denotes the block diagonal matrix made by $p$ and $q$.
\begin{definition}
	Let $A$ be a real, unital $C^*$-algebra.
	The group $KO_0(A)$ is defined as the Grothendieck group of the monoid $V(A)$.
	The $n$-th $KO$-theory group of $A$ is defined as:

	\[KO_n(A):=KO_0(A \otimes Cl_n),\]
	where $Cl_n$ denotes the real Clifford algebra. 

\end{definition}

\begin{remark}
An element in the $KO_0$-group can be seen as a formal difference of equivalence classes of projections in the real $C^*$-algebra $A \otimes \mathbb{K}$.

\end{remark}

Now let $\pi\colon M_{\Gamma} \to M$ be a Galois $\Gamma$-cover of a spin pseudomanifold with $(L,G)$-singularities as discussed in Section \ref{secroealgebras}. $M_\Gamma$ is automatically spin, and endow it with the lifted $\Gamma$-equivariant metric $g_{\Gamma}$ of a well-adapted wedge metric $g$ on $M$, which is $\Gamma$-equivariant, and consider the $\Gamma$-equivariant, $Cl_n$-linear (wedge) Atiyah-Singer operator $\slashed{\mathfrak{D}}_{\Gamma}$.

\begin{proposition} \label{fundamentalclass}
	Let $\pi \colon M_\Gamma \to M$ be a Galois $\Gamma$-covering as before, with $M$ a spin pseudomanifold with $(L,G)$-singularities of dimension $n$. 
	Then there is a well defined fundamental class associated to $\slashed{\mathfrak{D}}_\Gamma$, i.e. a $\Gamma$-equivariant real K-homology class $[\slashed{\mathfrak{D}}_\Gamma] \in KO^{\Gamma}_{n}(M_{\Gamma})$.
\end{proposition}

\begin{proof}
	Note, in particular, that since $\slashed{\mathfrak{D}}_\Gamma$ is the $\Gamma$-equivariant lift of $\slashed{\mathfrak{D}}$, the induced operators on the vertical fibers are the same. Consequently, since $M$ is psc-Witt by construction, Theorem \ref{essentialyselfadj} can be applied to $M_\Gamma$: thus we assume the existence of a unique self-adjoint extension of $\slashed{\mathfrak{D}}_\Gamma$ in $L^2(\slashed{\mathfrak{S}}((M_\Gamma)_r))$.
	
	 Next, we can consider $L^2(\slashed{\mathfrak{S}}((M_\Gamma)_r))$ as an ample, covariant $M_\Gamma$-module using the $\ast$-representation of $C_0(M_{\Gamma})$ given by: 
	
	\[
	\begin{tikzcd}
		C_{0}(M_{\Gamma}) \arrow[r] & C_{0}((M_{\Gamma})_r) \arrow[r, "\rho"] & \mathcal{L}(L^2(\slashed{\mathfrak{S}}((M_\Gamma)_r))),
	\end{tikzcd}
 \]

	where the first arrow is simply the restriction of a function $f \in C_0(M_{\Gamma})$ to $(M_{\Gamma})_r$, while $\rho$ is  the usual representation of $C_0((M_{\Gamma})_r$).

	Then, let $\varphi \in C_0(\mathbb{R})$ and $\chi\colon \mathbb{R} \to [-1,1]$ be a chopping function, i.e. an odd continuous function such that $\chi(x) \to 1$ as $x \to \infty$. From \cite[Lemma 3.1]{Roebaumconnes}, it turns out that $\varphi(\slashed{\mathfrak{D}}_\Gamma)$ and $\chi(\slashed{\mathfrak{D}}_\Gamma)$, defined by using the functional calculus, are elements of $C^{\ast}(M_{\Gamma};Cl_n)^{\Gamma}$ and $D^{\ast}(M_{\Gamma};Cl_n)^{\Gamma}$ respectively.
	As reported in \cite[Lemma 3.1]{Roebaumconnes}, $\chi(\slashed{\mathfrak{D}}_\Gamma)$ defines an element in: \[KO_1(D^{\ast}(M_{\Gamma};Cl_n)^{\Gamma} / C^{\ast}(M_{\Gamma};Cl_n)^{\Gamma})\simeq KO_{n+1}(D^{\ast}(M_{\Gamma};\mathbb{R})^{\Gamma} / C^{\ast}(M_{\Gamma};\mathbb{R})^{\Gamma})\]
	
		Finally, one uses the isomorphism (see for example \cite[Lemma 2.2]{Roebaumconnes}):
	
	\[KO_{n+1}(D^{\ast}(M_{\Gamma};\mathbb{R})^{\Gamma} / C^{\ast}(M_{\Gamma};\mathbb{R})^{\Gamma}) \simeq KO_n^\Gamma(M_\Gamma),\]
	
	to obtain the desired fundamental class $[\slashed{\mathfrak{D}}_\Gamma]$.
	
%	The function $\chi^2-1$ is by construction an element of $C_0(\mathbb{R})$, then:
%	 \[(\chi^2-1)(\slashed{\mathfrak{D}}_\Gamma) \in C^{\ast}(M_{\Gamma};\mathbb{R})^{\Gamma},\] 
%	 meaning that
%	  \[[\chi(\slashed{\mathfrak{D}}_\Gamma)] \in D^{\ast}(M_{\Gamma};\mathbb{R})^{\Gamma} / C^{\ast}(M_{\Gamma};\mathbb{R})^{\Gamma}\] 
%	  defines an involution.
%	
%	If the dimension $n$ of the resolved manifold $(M_{\Gamma})_r$ is odd, then such involution defines an element \[[\frac{1}{2}(\chi(\slashed{\mathfrak{D}}_\Gamma))] \in KO_{0}(D^{\ast}(M_{\Gamma};\mathbb{R})^{\Gamma} / C^{\ast}(M_{\Gamma};\mathbb{R})^{\Gamma}) \simeq KO_{n+1}(D^{\ast}(M_{\Gamma};\mathbb{R})^{\Gamma} / C^{\ast}(M_{\Gamma};\mathbb{R})^{\Gamma})\]
%	
%	If $n$ is even, one can uses the $\mathbb{Z}_2$-grading of $L^2(\slashed{\mathfrak{S}}((M_\Gamma)_r))$ and the fact that the Atiyah-Singer operator is odd with respect to this grading in order to define an unitary in $D^{\ast}(M_{\Gamma};\mathbb{R})^{\Gamma} / C^{\ast}(M_{\Gamma};\mathbb{R})^{\Gamma}$ and then an element in $KO_{n+1}(D^{\ast}(M_{\Gamma};\mathbb{R})^{\Gamma} / C^{\ast}(M_{\Gamma};\mathbb{R})^{\Gamma})$.
%	
%	Finally, one uses the isomorphism (see for example \cite[Lemma 2.2]{Roebaumconnes}):
%	
%	\[KO_{n+1}(D^{\ast}(M_{\Gamma};\mathbb{R})^{\Gamma} / C^{\ast}(M_{\Gamma};\mathbb{R})^{\Gamma}) \simeq KO_n^\Gamma(M_\Gamma),\]
%	
%	to obtain the desired fundamental class $[\slashed{\mathfrak{D}}_\Gamma]$.
	
\end{proof}

\begin{definition}\label{roeindex}
	In the above context, we define the \textbf{coarse index class} associated to $\slashed{\mathfrak{D}}_{\Gamma}$, as the image of $[\slashed{\mathfrak{D}}_{\Gamma}]$ under the coarse assembly map of the (real) Higson-Roe surgery sequence (\ref{higsonroeseq}):
	
	\[Ind(\slashed{\mathfrak{D}}_{\Gamma}):=\mu([\slashed{\mathfrak{D}}_{\Gamma}]) \in KO_n(C^{\ast}(M_{\Gamma}; \mathbb{R})^{\Gamma})\]
\end{definition}	

\begin{remark}\label{invindmetric}
	Note that the fundamental class just introduced, relative to the Atiyah-Singer operator in the wedge context, is equally well-defined when considering spin pseudomanifolds of depth 1 equipped with wedge metrics that are geometric Witt, i.e., satisfying the first condition of Theorem \ref{essentialyselfadj}. As a special case of \cite[Theorem 2.17]{BPR2}, if $g(t)$, for $t \in [0,1]$, is a family of geometric Witt metrics for every $t \in [0,1]$, then both the fundamental class and the coarse index class remain invariant (and the same for well-adapted wedge metrics). Therefore, since such a part can always be found between two well-adapted wedge metrics, the classes above in the $(L,G)$ context do not depend on the particular metric chosen, which is why we have omitted $g$ in the notation.
\end{remark}

\subsection{Graded Real $C^{\ast}$-algebras}

In this section, we proceed to introduce a new category of $C^*$-algebras, specifically those that are Real and equipped with a grading, for which we will subsequently define the respective $K$-theory groups. The advantage of working with these algebras lies in their greater generality and the fact that they allow us to work in both the complex and real cases simultaneously. 
Specifically, Real C*-algebras, as described below, are complex C*-algebras in the usual sense but endowed with an additional involution. With respect to this involution, the subalgebra given by the fixed points is a real C*-algebra, while the complexification of a real C*-algebra naturally admits a structure of a Real C*-algebra. In this way, the two categories are equivalent. A version of K-theory for Real C*-algebras will be introduced so that, by means of the forgetful functor that disregards the Real structure, it yields the usual K-theory for complex C*-algebras, while by restricting to the fixed points, it yields the KO-theory for real C*-algebras. For a more general introduction to the theory of Real $C^*$-algebras, we refer to \cite{Schroeder}.

\begin{definition} Let $A$ be a complex $C^{\ast}$-algebra:
	\begin{enumerate}[label=(\roman*)]
		\item $A$ is a \textbf{Real C$^*$-algebra} if it carries a $\ast$-isometric, antilinear, involutive automorphism $^- \colon a \mapsto \overline{a}$. 
		\item A $\ast$-morphism $\varphi\colon A \to B$ between two Real $C^*$-algebras is said Real if it preserves the Real-structures, meaning that $\overline{\varphi(a)}=\varphi(\overline{a})$, for each $a \in A$.
		\item A \textbf{$\mathbb{Z}_2$-grading} on a Real $C^{\ast}$-algebra $A$ is a Real, involutive $\ast$-automorphism $\alpha$. 
	\end{enumerate}
\end{definition}

Observe that the involution defining the Real structure preserves the order of multiplication. Moreover, a $\mathbb{Z}_2$-grading induces a decomposition:
\[A= A^{(0)} \oplus A^{(1)}, \quad A^{(i)}A^{(j)} \subset A^{(i+j)},\]
where $i,j \in \mathbb{Z}_2$. In particular, $\alpha(a)=(-1)^ia$ if $a \in A^{(i)}$.

A mapping $\varphi\colon A \to B$ between two graded $C^{\ast}$-algebras is \textit{graded} if preserves the grading, i.e. $\varphi(A^{(i)})\subset B^{(i)}$.

\begin{definition}\label{realification}
	If $A$ is a Real $C^*$-algebra, then its fixed point subalgebra with respect to this automorphism is naturally a real $C^{\ast}$-algebra, called the \textbf{realification} of $A$.
\end{definition} 

\begin{remark} 
	Most examples of Real $C^{\ast}$-algebras are given by complexification of real $C^{\ast}$-algebras $B$, i.e. $B \otimes_{\mathbb{R}} \mathbb{C}$ with involution given by complex conjugation. Of course, the realification in this case coincides with the real $C^*$-algebra $B$. Then one obtains that the category of real $C^*$-algebras and of Real $C^*$-algebras are equivalent. 
\end{remark}

\begin{definition}\label{maxgradedtensorproduct}
	Given two graded $C^{\ast}$ algebras $A$ and $B$, their \textbf{maximal graded tensor product} $A \widehat{\otimes} B$ is defined as the completion of their tensor product, but with product, grading, $\ast$-involution and $Real$-structure, given by:
	
	\[(a_1 \widehat{\otimes} b_1)(a_2 \widehat{\otimes}b_2)=(-1)^{ij} a_1a_2 \widehat{\otimes}b_1b_2, \quad b_1 \in B^{(i)}, \quad a_2 \in A^{(j)}, \]

	\[a \widehat{\otimes} b \in (A \widehat{\otimes} B)^{(i+j)}, \quad a \in A^{(i)}, \quad b \in B^{(j)},\]

	\[(a \widehat{\otimes} b)^{\ast}=(-1)^{ij}a^{\ast} \widehat{\otimes} b^{\ast}, \quad a \in A^{(i)}, \quad b \in B^{(j)},\]
	
	\[\overline{a \widehat{\otimes} b}=\overline{a} \widehat{\otimes} \overline{b}.\]
	
\end{definition}

\begin{definition}\label{RealCliffordalgebra}
	The \textbf{Real Clifford algebra} $Cl_{p,q}$ (here, with a small abuse of notation, we are using the same notation used for the real Clifford algebra in Section \ref{secwedgediracoperator} because of the following remark) is the unital, complex algebra generated by the real generators $\{e_1, \ldots, e_p, \epsilon_1, \ldots, \epsilon_q\}$ satisfying the relations:
	
	\[e_i^2=-1, \quad \epsilon_j^2=1,\] 
	\[\ e_ie_j=-e_je_i, \quad \epsilon_i\epsilon_j=-\epsilon_j\epsilon_i, \quad \text{if} \quad i \neq j,\]
	\[e_i\epsilon_j=-\epsilon_je_i, \quad \forall i,j.\]
	
	This become graded, Real $C^{\ast}$-algebras by assuming all generators are in $Cl_{p,q}^{(1)}$, i.e. they are odd, and the $\ast$-involution is:
	
	\[e_i^*=-e_i, \quad \epsilon_j^*=\epsilon_j, \quad \forall i,j\]

	Moreover, the Real structure is given by requiring that the involution $^-$ is the identity on the real generators and it is extended compatibly with the complex conjugation.
	
\end{definition}

\begin{remark} \label{Realreal}
	
Observe that $Cl_{p,q}$ corresponds, without considering the Real structure, to the complex Clifford algebra $\mathbb{C}l_{p,q}:=Cl(\mathbb{C}^{p+q}, \eta)$, where $\eta$ is the quadratic form:

\[\eta=\sum_{i=1}^p x_i^2 - \sum_{j=p+1}^{p+q} x_j^2.\]

However, if $p+q=n$, all quadratic forms on $\mathbb{C}^{n}$ are equivalent, then it suffices to consider $\mathbb{C}l_{p+q,0}$, which can be identified:

\[\mathbb{C}l_{p+q,0} \simeq Cl_{n,0} \otimes_{\mathbb{R}} \mathbb{C}.\]

Then, by choosing different Real structures, i.e. different involutions $^-$, one can obtain the real Clifford algebra $Cl_{r,s}$, where $r+s=n$ as its realification. In particular, with the involution defined in Definition \ref{RealCliffordalgebra}, one gets exactly the real Clifford algebras $Cl_{p,q}$.
\end{remark}
\begin{definition}\label{Realhilbertspace}
	A complex Hilbert space $H$ is called Real if it is equipped with a $\mathbb{C}$-antilinear involution $^-$. The complex C*-algebra $\mathcal{L}(H)$ of bounded operators on $H$ inherits an involution from that of $H$ by requiring that, if $T \in \mathcal{L}(H)$:
	
	\[\overline{T}(h):=\overline{T(\overline{h})}, \quad \forall h \in \mathcal{H}.\]
	
\end{definition}

\begin{remark}
	As reported in \cite[Appendix B]{HigsonRoe}, the K-theory groups of a Real C*-algebra $A$ are defined as the K-theory groups of $A$ (considered as a complex C*-algebra) where only the elements fixed by the Real structure $^-$ of $A$ are considered.
	
	In a similar manner, one can define the K-homology groups for a Real C*-algebra $A$ by considering Fredholm modules $(\rho, H, F)$, where $\rho$ is a representation of $A$, $H$ is a Real Hilbert space, and $F$ is an operator satisfying the usual conditions, along with the requirement that $\overline{F}=F$.
\end{remark}

\subsection{Localization algebras} \label{seclocalizationalgebras}

Now, we introduce a variant approach to the Roe algebras following \cite{Yu} and in \cite{Zeidler}: the localization algebras. These make easier to prove the next results.

 We need firstly to replace the Hilbert space in the definition of $X_{\Gamma}$-module with a graded Real $Cl_n$-Hilbert module (recall also Remark \ref{roecoefficients}), where $Cl_n$ is the Real Clifford algebra of Definition \ref{RealCliffordalgebra}. 

\begin{remark}
If $A$ is a Real $C^*$-algebra, we mean by a Real $A$-Hilbert module a $A$-Hilbert module $H$ such that it is endowed with a $\mathbb{C}$-antilinear 
 involution $^-$ compatible with that of $A$ (denoted again $^-$) in the following sense:
 
 \[\overline{w\cdot a}= \overline{w} \cdot \overline{a}, \quad \overline{\left(v,w\cdot a\right)}=\overline{\left(v,w\right)}\cdot \overline{a},\] 
 for all $v, w \in H$ and $a \in A$, where $\cdot$ denotes the right $A$-module operation and $\left(-,- \right)$ the $A$-valued product of $H$
\end{remark}

 Moreover, we require that all the representations are by even, $Cl_n$-linear, bounded, adjointable operators. We call such a module a covariant $(X_{\Gamma},Cl_n)$-module. It follows that all the above defined Roe algebras of Section \ref{secroealgebras}  have an analogue in this context (see also Remark \ref{roecoefficients}): we denote them $D^{\ast}(X_{\Gamma},H;Cl_n)^{\Gamma}$, $C^{\ast}(X_{\Gamma},H;Cl_n)^{\Gamma}$, $\ldots$
 Observe that, as specified in Definition \ref{Realhilbertspace}, such algebras become Real with involution induced by that of $H$.

\begin{definition}
	Let $H$ be an ample, covariant $(X_\Gamma, Cl_n)$-module. We introduce the following \textbf{localization algebras}.
	
	\begin{enumerate}[label=(\roman*)]
		\item $C_L^{\ast}(X_{\Gamma},H; Cl_n)^{\Gamma}$ is the $C^{\ast}$-subalgebra of the continuous functions from $[1,\infty)$ to $C^{\ast}(X_{\Gamma}, H;Cl_n)^{\Gamma}$ generated by bounded and uniformly continuous functions $L$ such that the propagation of $L(t)$ is finite and tends to $0$ as $t \to \infty$.
		
		\item If $Z \subset X_{\Gamma}$ is a $\Gamma$-invariant closed subset, then $C^{\ast}_L(Z \subset X_{\Gamma}, H; Cl_n)$ is the ideal in $C_L^{\ast}(X_{\Gamma},H; Cl_n)^{\Gamma}$ made by all the functions $L(t)$ such that $supp(L(t)) \subset B_{R(t)}(Z \times Z)$, where $R\colon [1,\infty) \to \mathbb{R}_+$ goes to zero as $t \to \infty$.
		
		\item There is a well defined surjective map 
		\[ev_1 \colon C_L^{\ast}(X_{\Gamma},H; Cl_n)^{\Gamma} \to C^{\ast}(X_{\Gamma},H;Cl_n)^{\Gamma}, \quad  L(t) \mapsto L(1).\] 
		We denote:
		\[C_{L,0}^{\ast}(X_{\Gamma},H;Cl_n)^{\Gamma}:= \ker{\left(ev_1\right)}, \quad  C_{L,Z}^{\ast}(X_{\Gamma},H; Cl_n)^{\Gamma}:=\left(ev_1\right)^{-1}\left(C^{\ast}(Z \subset X_{\Gamma},H; Cl_n)\right)\]
		
	\end{enumerate}
\end{definition}

All these are Real $C^*$-algebras with the obvious involutions induced by the localization algebras.
Similarly, definitions for the localization algebras $D_L^{\ast}(X_{\Gamma},H;Cl_n)^{\Gamma}$ are obtained. As for the case of Roe algebras, these algebras are functorial with respect to uniformly continuous and coarse maps. This means again that, at the $K$-theory level, such induced mappings become canonical (see \cite[Lemma 3.4]{Yu}, \cite[3.2]{Zeidler} for details). Therefore, we won't specify the choice of the ample module.

\begin{remark}\label{reallocalization}
	 All the above definitions still hold without a Clifford algebra. This means that instead of a $Cl_n$-Hilbert module, one consider a Hilbert space, real or complex. In that case one simply denotes the above algebras omitting $Cl_n$. The definitions still remain the same, except that the localization algebras are defined in terms of functions with values to the Roe algebras with the respective coefficients.
	
For example, take a complex Hilbert space $H$, which is also an ample covariant $X_\Gamma$-module. Then, the respective localization algebras are defined in terms of Roe algebras defined in \ref{secroealgebras} and are denoted $C_L^{\ast}(X_{\Gamma})^{\Gamma}$, $C^{\ast}_{L,0}(X_{\Gamma})^{\Gamma}$, \ldots 

Complex conjugation in $H$ defines a Real structure, and $H$ is isomorphic to the complexification of the real HIlbert space $H_\mathbb{R}$ given by fixed point of this involution. The operators fixed by the induced involution, i.e. the elements of the realification, are exactly the real operators on $H_\mathbb{R}$, and then one obtains that the realification of the complex localization algebras are exactly those with real coefficients.
\end{remark}	

The Yu's localization algebras introduced offer in particular an alternative approach to $K$-homology. In particular, refering to the complex case, there exists an isomorphism $Ind_L^{\Gamma}\colon K^{\Gamma}_{\ast}(X_{\Gamma}) \simeq K_{\ast}(C_L^{\ast}(X_{\Gamma})^{\Gamma})$ which is compatible with the coarse assembly map in the sense that the following diagram commutes (\cite[Section 4]{Yu}):

\begin{equation} \label{indassembly}
	\begin{tikzcd}
		K_{\ast}(C_L^{\ast}(X_{\Gamma})^{\Gamma}) \arrow[r,"(ev_1)_{\ast}"]  & K_{\ast}(C^{\ast}(X_{\Gamma})^{\Gamma})   \\
		K_{\ast}^{\Gamma}(X_{\Gamma})   \arrow[ur,"\mu"] \arrow[swap, u,"\sim" labl, "Ind_L^{\Gamma}"]
	\end{tikzcd}
\end{equation}

The above localization algebras fit into the following short exact sequence:

\begin{equation}\label{higsonroelocalizedshort}
\begin{tikzcd}
	0 \arrow[r] & C^{\ast}_{L,0}(X_{\Gamma})^{\Gamma} \arrow[r] & C^{\ast}_L(X_{\Gamma})^{\Gamma} \arrow[r, "ev_1"] & C^{\ast}(X_{\Gamma})^{\Gamma}  \arrow[r] & 0.
\end{tikzcd}
\end{equation}

The long exact sequence in $K$-theory associated to it can be identified to the Higson-Roe surgery sequence (\ref{higsonroeseq}) by
using the diagram (\ref{indassembly}) and the fact that the $K$-theory group $ K_{\ast}(C^{\ast}_{L,0}(X_{\Gamma})^{\Gamma})$ is isomorphic to $K_{\ast+1}(D^{\ast}(X_{\Gamma})^{\Gamma})$. One then obtains the following commutative diagram (see \cite[Proposition 6.1]{XieYu}):

\begin{equation} \label{higsonroelocalized}
\begin{tikzcd}[sep=small, row sep=1.5em]
	\ldots \arrow[r] & K_{n+1}(C^{\ast}(X_{\Gamma})^{\Gamma}) \arrow[r, "\delta"] \arrow[d,equal]& K_n(C^{\ast}_{L,0}(X_{\Gamma})^{\Gamma}) \arrow[r] & K_n(C^{\ast}_L(X_{\Gamma})^{\Gamma}) \arrow[r,"(ev_1)_{\ast}"] & K_n(C^{\ast}(X_{\Gamma})^{\Gamma})  \arrow[r] \arrow[d,equal] & \ldots \\
	\ldots \arrow[r] & K_{n+1}(C^{\ast}(X_{\Gamma})^{\Gamma}) \arrow[r] & K_{n+1}(D^{\ast}(X_{\Gamma})^{\Gamma}) \arrow[u, "\simeq"labl] \arrow[r] & K^{\Gamma}_{n}(X_{\Gamma}) \arrow[r, "\mu"] \arrow[swap, u,"\sim" labl, "Ind_L^{\Gamma}"] & K_{n}(C^{\ast}(X_{\Gamma})^{\Gamma}) \arrow[r] & \ldots
\end{tikzcd}	
\end{equation}

\begin{remark}
	The same results are obtained in $KO$-theory for real $C^*$-algebras considering the Roe algebras (and consequently the localization algebras) with real coefficients.
\end{remark}

\subsection{K-theory of graded Real $C^{\ast}$-algebras}

Now we will describe an alternative, and equivalent, approach to the $K$-theory ($KO$) of a complex (real) $C^*$-algebra due to Trout (see \cite{Trout}), which is alternative to the usual one in terms of projections and unitaries. This approach will simultaneously take into account any $\mathbb{Z}_2$-grading and Real structure of a $C^*$-algebra. Recall that a $C^*$-algebra is Real if it is complex and equipped with an involution (which plays the role of a "complex conjugation"): it is clear then that where this structure is trivial, what we consider are simply $C^*$-algebras over $\mathbb{C}$.

First of all, in order to define the $K$-theory groups, it is necessary to introduce the following graded, Real $C^*$-algebras.

\begin{definition} Define the following:
\begin{itemize}

	\item the graded, Real $C^{\ast}$-algebra $\mathbb{K}$ of compact operators acting on a graded, $Real$, countably infinite dimensional Hilbert space $\mathcal{H}=\mathcal{H}^{(0)}\oplus\mathcal{H}^{(1)}$ with grading given by the decomposition into diagonal and off-diagonal matrices and Real-structure given as in Definition \ref{Realhilbertspace};

	\item the graded, Real $C^{\ast}$-algebra $\mathcal{S}$ of functions in $C_0(\mathbb{R}, \mathbb{C})$ with grading given by even and odd functions and Real-structure by complex conjugation.
\end{itemize}
\end{definition}

If $B$ and $C$ are two graded, Real $C^{\ast}$-algebras, then we denote by $Hom(B,C)$ the set of homotopy classes of Real, graded, $\ast$-homomorphisms $\phi\colon B \to C$ with the \textit{point-norm topology}, meaning that $\phi_{\alpha} \to \phi$ if and only if $\phi_{\alpha}(b) \to \phi(b)$, for each $b \in B$, in the norm topology of $A \widehat{\otimes} \mathbb{K}$.

\begin{definition}
	We define the \textbf{$K$-theory groups} of a graded, Real $C^{\ast}$-algebra $A$ by:
	
	\[\widehat{K}_n(A):=\pi_n(Hom(\mathcal{S}, A \widehat{\otimes} \mathbb{K})), \quad \forall n \geq 0,\]
	where $\pi_n(Hom(\mathcal{S}, A \widehat{\otimes} \mathbb{K}))$ denotes the $n$-th homotopy group of $Hom(\mathcal{S}, A \widehat{\otimes} \mathbb{K})$ with respect to the point-norm topology and with zero map as base point.
\end{definition}

\begin{remark}\label{definingelementinK}
	Observe that any graded, Real, $\ast$-homomorphism $\varphi\colon \mathcal{S} \to A$ defines an element in the group $\widehat{K}_0(A)$. In fact, one can consider the element $[\varphi]$ represented by $\varphi \widehat{\otimes} e_{1,1}$, where $e_{1,1}$ denotes the even, rank-$1$ projection in $\mathbb{K}$.
\end{remark}

For each $n \geq 0$, $\widehat{K}_n(A)$ have an abelian group structure by considering the direct sum of two mappings and the identifications $(A \widehat{\otimes} \mathbb{K}) \oplus (A \widehat{\otimes} \mathbb{K}) \simeq A \widehat{\otimes} (\mathbb{K} \oplus \mathbb{K}) \simeq A \widehat{\otimes} \mathbb{K}$. These agree with the homotopy groups operations for $n \geq 1$.

\begin{remark}\label{Realreal2}
	Observe that for complex $C^*$-algebras, i.e. without a Real-structure, one can simply ignore all the Real structures and the Real condition on the $Hom$ set obtaining a complex analogue of the above definition.
	
On the contrary, if the Real structure is considered, the groups just introduced correspond to the $KO$-theory groups of real $C^*$-algebras in the following sense. Recalling what is meant by the realification of a Real $C^*$-algebra (see Definition \ref{realification}), it is observed that if one considers the fixed points of $\mathcal{S}$ with respect to its Real structure, one obtains the $C^*$-algebra $C_0(\mathbb{R}, \mathbb{R})$ of real-valued functions that vanish at infinity. Since Real morphisms preserve the Real structures, and therefore the fixed point spaces with respect to them, one can consider their restrictions to the corresponding realifications of the algebras. Consequently, the $K$-theory group of a Real $C^*$-algebra $A$ considered will correspond to the $KO$-theory group of the realification of the $C^*$-algebra $A$.
\end{remark}	 

\begin{remark}\label{K-KO}
	Recall that every $C^{\ast}$-algebra can be considered $\mathbb{Z}_2$-graded when paired with the trivial grading. Then, the definition above applies to any $C^*$-algebra. However, when $A$ is ungraded, its $K$-theory groups will be denoted simply by $K_{\ast}(A)$.
	
	 Observe that, when $A$ is trivially graded and unital, then $A \widehat{\otimes} \mathbb{K} \simeq M_2(A \otimes \mathbb{K})$ with grading given induced by $\mathbb{K}$, i.e. that given by diagonal and off-diagonal matrices. Consequently:
	 
	 \[\widehat{K}_n(A) = K_n(A)=\pi_n(Hom(\mathcal{S}, A \otimes \mathbb{K})), \quad \forall n \geq 0\]
	 
\end{remark}

It can be verified that $\widehat{K}_n(A)\simeq \widehat{K}_0(\Sigma^nA)$ where $\Sigma^nA$ denotes the $n$-th suspension of $A$ defined by $\Sigma^nA:=C_0(\mathbb{R}^n) \widehat{\otimes} A$, with $C_0(\mathbb{R}^n)$ is considered with trivial grading.

The $C^{\ast}$-algebra $\mathcal{S}$ is characterized by the presence of a $\ast$-homomorphism $\Delta\colon \mathcal{S} \to \mathcal{S} \widehat{\otimes} \mathcal{S}$ which gives to $\mathcal{S}$ a coalgebra structure (see \cite[Section 1.3]{HigsonGuentner} for details). 
We mention it because it induces an external product in $K$-theory:

 \[\widehat{K}_n(A) \otimes \widehat{K}_m(B) \to \widehat{K}_{n+m}(A \widehat{\otimes} B).\]
 
   This is defined by firstly making explicit what happens in the case when $n=m=0$. 
   Consider two classes $[\phi] \in \widehat{K}_0(A)$ and $[\psi] \in \widehat{K}_0(B)$ and define their product as the following class expressed in terms of their representatives and the above comultiplication:
    \[[\phi] \times [\psi] :=[(\phi \widehat{\otimes} \psi) \circ \Delta] \in \widehat{K}_0(A \widehat{\otimes} B).\] 
    
    This product is associative, commutative and functorial, meaning that $\varphi\colon A \to A'$ and $\chi\colon B \to B'$ are $Real$, graded $\ast$-homomorphisms, then: 
\[(\varphi \widehat{\otimes} \chi)_{\ast}([\phi] \times [\psi])=\varphi_{\ast}([\phi]) \times \chi_{\ast}([\psi]).\]

Finally, the product for each $n, m$ is induced by this (\cite[Section 1.7]{HigsonGuentner}).
\begin{example}
	Consider an unbounded, odd, self-adjoint operator $D$ on a graded Hilbert space $H$, i.e.:
	
	\[D=\begin{pmatrix}
		0 & D_{-} \\
		D_{+} & 0
	\end{pmatrix}
\]

Assuming that $D$ has a compact resolvent, then the functional calculus applied to $D$ defines a mapping $\psi_D\colon \mathcal{S} \to \mathbb{K}$, hence an element of $\widehat{K}_0(\mathbb{C})$. For example, $D$ can be any Dirac operator on a compact manifold. 

Under the isomorphism $\widehat{K}_0(\mathbb{C}) \simeq \mathbb{Z}$, to $\psi_D$ is associated the Fredholm index of $D_{+}$.

Now, if $A=B=\mathbb{C}$ and $\psi_{D_1}$, $\psi_{D_2}$ are defined as above, then their product $[\psi_{D_1}] \times [\psi_{D_2}]$ is the functional calculus associated to the self-adjoint operator $D_1 \widehat{\otimes} \mathbb{1} + \mathbb{1} \widehat{\otimes} D_2$, which is an operator whose Fredholm index is exactly the product of the indices of $D_1$ and $D_2$.
\end{example}

Let us consider $C_0(\mathbb{R}^n) \widehat{\otimes} Cl_{0,n}$, i.e. the graded, $C^{\ast}$-algebra of continuous functions vanishing at infinity from $\mathbb{R}^n$ to $Cl_{0,n}$. The natural inclusion of $\mathbb{R}^n$ in $Cl_{0,n}$ is, of course, a function non-vanishing at infinity, but by applying the functional calculus using a function in $\mathcal{S}$, one gets for each $f \in \mathcal{S}$ the following mapping:

\begin{equation}\label{bottelement}
	 \mathbb{R}^n \ni v \mapsto f(v) \in Cl_{0,n}, 
\end{equation}	
giving an element in $C_0(\mathbb{R}^n) \widehat{\otimes} Cl_{0,n}$: (\ref{bottelement}) then defines an element in $Hom(\mathcal{S},C_0(\mathbb{R}^n) \widehat{\otimes} Cl_{0,n})$. In this way one obtains the so called \textit{Bott element} $b_n \in \widehat{K}_0(C_0(\mathbb{R}^n) \widehat{\otimes} Cl_{0,n})$ (recall Remark \ref{definingelementinK}). The Bott element is of particular interest because of the following (see \cite[Theorem 1.14]{HigsonGuentner}).

\begin{theorem} \label{bottperiodicity}
	For any graded, $Real$, $C^{\ast}$-algebra A and for each $n \in \mathbb{N}$, the following is an isomorphism:
	
	\[\beta\colon \widehat{K}_0(A) \to \widehat{K}_0(A \widehat{\otimes} C_0(\mathbb{R}^n) \widehat{\otimes} Cl_{0,n}), \quad x \mapsto x \times b_n\]
	where $\times$ denotes the external product defined previously.
	
\end{theorem}

In particular, from this Theorem and the fact that $Cl_{n,0} \widehat{\otimes} Cl_{0,n} \simeq Cl_{n,n} \simeq M_{2n}(\mathbb{C})$, one obtains the following identifications:

\begin{equation}\label{bott}
	\widehat{K}_0(A \widehat{\otimes} Cl_n) \simeq \widehat{K}_0(A \widehat{\otimes}  C_0(\mathbb{R}^n) \widehat{\otimes} Cl_n \widehat{\otimes} Cl_{0,n}) \simeq \widehat{K}_0(\Sigma^n A) \simeq \widehat{K}_n(A).
	\end{equation}

Considering the Real structures, this corresponds to the usual $8$-periodicity of $KO$-theory of real $C^*$-algebras (recall Remark \ref{Realreal} and \ref{K-KO}), while it corresponds to the $2$-periodicity of $K$-theory in the complex case without considering them. 

Now, let us focus on the case of the localization algebras introduced in \ref{seclocalizationalgebras}. An element in $\widehat{K}_0(C^{\ast}_L(X_\Gamma; Cl_n)^{\Gamma})$ will be given by a homotopy class of a morphism:

\[\varphi\colon \mathcal{S} \to C^{\ast}_L(X_\Gamma; Cl_n)^{\Gamma} \widehat{\otimes} \mathbb{K}\]

Since the localization algebras are Real, $\varphi$ has to be considered Real too and its restriction to the fixed point set induces a mapping between the realifications. It then defines an element in the $KO$-theory of the realification of $C^{\ast}_L(X_\Gamma; Cl_n)^{\Gamma}$.

Recall that localization algebras are functorial with respect to $\Gamma$-equivariant, coarse maps. This implies that there is no dependence on the choice of the module. Then, let $H$ be a real Hilbert space and suppose that $H$ is an ample, covariant real $X_{\Gamma}$-module. Of course, $H_{\mathbb{C}}=H \otimes_{\mathbb{R}} \mathbb{C}$ is a complex analogue, with Real structure given by complex conjugation, while $H_{\mathbb{C}} \widehat{\otimes} Cl_n$ is an ample, covariant $(X_\Gamma, Cl_n)$-module whose Real structure is given by combining those of $H_{\mathbb{C}}$ and $Cl_n$ (recall Definition \ref{maxgradedtensorproduct}). 
			
This module makes explicit that $C^{\ast}_L(X_\Gamma; Cl_n)^{\Gamma}=C^{\ast}_L(X_\Gamma)^{\Gamma} \widehat{\otimes} Cl_n$ and its realification is the graded tensor product of the localization algebra with real coefficients and the real Clifford algebra $Cl_n$ (recall Remark \ref{Realreal} and \ref{reallocalization}).

 Then, by using (\ref{bott}), we obtain that:

\begin{equation}\label{grrealcorr}
	\widehat{K}_0(C^{\ast}_L(X_\Gamma; Cl_n)^{\Gamma}) \simeq K_n(C^{\ast}_L(X_\Gamma)^{\Gamma}),
\end{equation}	

which corresponds to $KO_n(C^{\ast}_L(X_\Gamma; \mathbb{R})^{\Gamma})$ by restricting to fixed points sets (recall Remark \ref{Realreal2}).

\subsection{The local index class of $\slashed{\mathfrak{D}}_\Gamma$ and the $\rho^\Gamma$ secondary invariant} \label{seclocalindex}

We now consider again the setting already introduced in Section \ref{k-homologyindexclass} with a minor modification, namely that we use instead of the real Clifford algebra, its Real counterpart. We denote again the bundle obtained $\slashed{\mathfrak{S}}(M_r)$.

 When $\pi\colon M_\Gamma \to M$ is a Galois $\Gamma$-cover, observe that the space of $L^2$ sections of the bundle $\slashed{\mathfrak{S}}((M_\Gamma)_r)$ over $(M_\Gamma)_r$, namely $L^2(\slashed{\mathfrak{S}}((M_\Gamma)_r))$, thanks to its right $Cl_n$ action, can be turned into a Hilbert $Cl_n$-module. Moreover, since $Cl_n$ denotes the Real Clifford algebra of Definition \ref{RealCliffordalgebra}, then it becomes a Real Hilbert $Cl_n$-module. More specifically, this is an ample, covariant $(X_\Gamma, Cl_n)$-module with graded and Real structures, with obvious Real structure induced by that of $Cl_n$.

We now introduce the local index class associated to the Atiyah-Singer operator $\slashed{\mathfrak{D}}_{\Gamma}$, following \cite[Chapter 4]{Zeidler}.
This will be given by a homomorphism $\varphi_{\slashed{\mathfrak{D}}_{\Gamma}}$ which makes use of the functional calculus to define the following mapping:
\[  \mathcal{S}\ni f \mapsto f(\frac{1}{t}\slashed{\mathfrak{D}}_{\Gamma}), \quad  t \in [1, \infty)\] 

By well known results (see, for example, \cite[Sections 5.3, 5.5]{HigsonRoe}), the operator $f(\frac{1}{t}\slashed{\mathfrak{D}}_{\Gamma})$ is locally compact for each $f \in \mathcal{S}$ and its propagation goes to zero as $t$ goes to infinity. Moreover, by general properties of the functional calculus and since $\slashed{\mathfrak{D}}_{\Gamma}$ is an odd operator, we obtain the following Real, graded, $\ast$-homomorphism:

\[\varphi_{\slashed{\mathfrak{D}}_{\Gamma}}\colon \mathcal{S} \to C^{\ast}_L(M_\Gamma; Cl_n)^{\Gamma}.\]

\begin{definition}
	The \textbf{local index class} $Ind^{\Gamma}_L(\slashed{\mathfrak{D}}_{\Gamma})$ associated to the operator $\slashed{\mathfrak{D}}_{\Gamma}$ is the $K$-theory element associated to $\varphi_{\slashed{\mathfrak{D}}_{\Gamma}}$, i.e.:
	
	\[Ind^{\Gamma}_L(\slashed{\mathfrak{D}}_{\Gamma}):=[\varphi_{\slashed{\mathfrak{D}}_{\Gamma}}] \in \widehat{K}_0(C^{\ast}_L(M_\Gamma; Cl_n)^{\Gamma})\]
	
In particular, using the correspondence (\ref{grrealcorr}), we obtain:

\[Ind^{\Gamma}_L(\slashed{\mathfrak{D}}_{\Gamma}) \in KO_n(C^{\ast}_L(M_\Gamma; \mathbb{R})^{\Gamma})\]	
\end{definition}

\begin{remark}The map induced in $K$-theory by the evaluation at $1$, namely  $(ev_1)_{\ast}$, sends  $Ind_L^{\Gamma}(\slashed{\mathfrak{D}}_{\Gamma})$ to the class $Ind(\slashed{\mathfrak{D}}_{\Gamma}) \in KO_n(C^{\ast}(M_{\Gamma}; \mathbb{R})^{\Gamma})$ of Definition \ref{roeindex}, accordingly to the diagram (\ref{indassembly}). In fact, the mapping $f \mapsto f(\slashed{\mathfrak{D}})$ corresponds the to the equivariant coarse index (see \cite[Proposition 5.3]{Trout}).
\end{remark}

Now, we introduce the localized analogue of the secondary invariant $\rho$ associated to a metric of uniformly positive scalar curvature (see \cite[Section 1.3]{PiazzaSchick}).

 Firstly, if $Z \subset M_{\Gamma}$ is a closed $\Gamma$-invariant subset, we say that a Riemannian metric $g$ (in our case,a well-adapted wedge metric $g$) has \textit{uniformly positive scalar curvature outside Z} if there exists $\varepsilon > 0$ such that the scalar curvature of $g$ is bounded by below by $\varepsilon$ on $M_{\Gamma} \setminus Z$.
 
Now, by \cite[Lemma 2.3]{RoePosCur}, if the Riemannian metric $g$ has uniformly positive scalar curvature with $\varepsilon = 4\tau^2$, for some $\tau>0$, then the restriction of the homomorphism $\varphi_{\slashed{\mathfrak{D}}_{\Gamma}}$ to those functions in $\mathcal{S}$ such that their support are contained in $(-\tau, \tau)$ takes value in $C^{\ast}_{L,Z}(M_{\Gamma}; Cl_n)^{\Gamma}$.

By \cite[Lemma 2.3]{Zeidler}, the inclusion $\iota$ of the set of functions supported in $(-r,r)$ into $\mathcal{S}$ is an homotopy equivalence of graded algebras (unique up to homotopy) for all $r>0$. Then, we set $r=\tau$ and we choose an homotopy inverse $\iota^{-1}$ to such inclusion, obtaining the following definitions.

\begin{definition}\label{partialsecondaryclass}
	Let $Z \subset M_{\Gamma}$, $g$ a well-adapted wedge metric on $M_\Gamma$ with uniformly positive scalar curvature outside $Z$ by a certain $\varepsilon$. Then the \textbf{partial secondary local index class} of the Atiyah-Singer operator $\slashed{\mathfrak{D}}_{\Gamma}^{g}$ associated to $g$ is defined as:
	
	\[Ind^{\Gamma}_{L,Z}(\slashed{\mathfrak{D}}_{\Gamma}^{g}):=[\varphi_{\slashed{\mathfrak{D}}_{\Gamma}}\circ \iota^{-1}] \in \widehat{K}_0(C^{\ast}_{L,Z}(M_{\Gamma}; Cl_n)^{\Gamma})\simeq KO_n(C^{\ast}_{L,Z}(M_{\Gamma};\mathbb{R})^{\Gamma}).\]
	
	In particular, if $Z=\emptyset$, i.e. $g$ is of uniformly positive scalar curvature, then the partial secondary local index class takes value in $C^{\ast}_{L,\emptyset}(M_{\Gamma}; Cl_n)^{\Gamma}$, which corresponds to $C^{\ast}_{L,0}(M_{\Gamma}; Cl_n)^{\Gamma}$. In this case, we define the \textbf{$\rho$-invariant} of $g$ as:
	
	\[\rho^{\Gamma}(g):=Ind^{\Gamma}_{L,\emptyset}(\slashed{\mathfrak{D}}_{\Gamma}^g) \in KO_n(C^{\ast}_{L, 0}(M_{\Gamma};\mathbb{R})^{\Gamma}).\]
\end{definition} 

\begin{remark}\label{partialtolocal}
	Observe that the local index class and its partial secondary counterpart are related as follows. Let $j\colon C^{\ast}_{L,Z}(M_{\Gamma};\mathbb{R})^{\Gamma} \hookrightarrow C^{\ast}_L(M_\Gamma; Cl_n)^{\Gamma}$ be the inclusion, then the induced map between their $KO_n$-groups sends exactly $Ind^{\Gamma}_{L,Z}(\slashed{\mathfrak{D}}_{\Gamma}^{g})$ to $Ind^{\Gamma}_L(\slashed{\mathfrak{D}}_{\Gamma})$.

This is quite obvious since, by construction: \[j_*(Ind^{\Gamma}_{L,Z}(\slashed{\mathfrak{D}}_{\Gamma}^{g}))=[j \circ \varphi_{\slashed{\mathfrak{D}}_{\Gamma}}\circ \iota^{-1}]=[ \varphi_{\slashed{\mathfrak{D}}_{\Gamma}}]=Ind^{\Gamma}_L(\slashed{\mathfrak{D}}_{\Gamma}),\]
since $\iota^{-1}$ is homotopic to the identity on $\mathcal{S}$.
\end{remark}	

\begin{remark}
	Note that in the definition of both the partial secondary index class and the rho invariant, we have emphasized the metric $g$. This is because the argument made for the coarse index class, discussed in Remark \ref{invindmetric}, cannot be repeated since it is not guaranteed that two metrics with uniformly positive scalar curvature outside a closed subset $Z$ can be connected by a continuous path of metrics with the same curvature condition.
\end{remark}

\subsection{The mapping theorem}

We are now finally able to describe the main result of this chapter, i.e. how the $(L,G)$-Stolz sequence of Theorem \ref{l,gstolzseq} can be mapped to the Higson-Roe surgery sequence in terms of localized algebras, as in \ref{higsonroelocalized}.

If $\Gamma$ is a countable, discrete group, then any proper, complete metric space $X_\Gamma$ with a free, cocompact $\Gamma$-action admits a coarse map into a universal space $E\Gamma$ for free $\Gamma$-actions (a contractible $CW$-complex with free $\Gamma$-action). In particular, this map is unique up to $\Gamma$-homotopy of coarse maps. Then, one defines the following \textit{universal} algebras: 

\[KO_{\ast}(C^{\ast}_{\Gamma}):=\varinjlim_{X_{\Gamma} \subset E\Gamma} KO_{\ast}(C^{\ast}(X_{\Gamma}; \mathbb{R})^{\Gamma}) \simeq KO_{\ast}(C^{\ast}_{r,\mathbb{R}}\Gamma),\]

where $C^{\ast}_{r,\mathbb{R}}\Gamma$ is the real reduced $C^{\ast}$-algebra of the group $\Gamma$, and the limit is performed along the family of all the spaces $X_\Gamma$ with the hypothesis above. 
The isomorphism easily follows from \cite[Lemma 5.14]{Roeindextheory} (which extends to the real case), which implies that this direct limit is obtained along canonical isomorphisms.

\begin{remark}
	The real \( C^* \)-algebra \( C^{\ast}_{r,\mathbb{R}} \Gamma \) is defined exactly as in the complex case. In particular, it is defined as the norm closure of the real group algebra:
	
	\[ \mathbb{R} \Gamma := \left\{ \sum_{g \in \Gamma} r_g g \mid g \in \Gamma, r_g \in \mathbb{R}, r_g = 0 \text{ for all but finitely many } r_g \right\}, \]
	
	which is viewed as a real operator on the real Hilbert space \( l^2(\Gamma) \) acting by left multiplication.
\end{remark}

Similarly, one obtains the following objects:

\[KO_{\ast}(C^{\ast}_{L;\Gamma}):=\varinjlim_{X_{\Gamma} \subset E\Gamma} KO_{\ast}(C^{\ast}_L(X_{\Gamma};\mathbb{R})^{\Gamma}),\] \[KO_{\ast}(C^{\ast}_{L,0;\Gamma}):=\varinjlim_{X_{\Gamma} \subset E\Gamma} KO_{\ast}(C^{\ast}_{L,0}(X_{\Gamma};\mathbb{R})^{\Gamma}).\]

Observe that the mappings in the Higson-Roe surgery sequence (and its localized version) (\ref{higsonroelocalized}) commute with all those of the direct family used to construct the above direct limits. Therefore, there is an induced canonical universal exact sequence:

\begin{equation}\label{higsonroelocalizedu}
	\begin{tikzcd}[sep=small, row sep=1.5em]
		\ldots \ar[r] & KO_{n+1}(C^{\ast}_{r,\mathbb{R}}\Gamma) \ar[r] & KO_n(C^{\ast}_{L,0;\Gamma}) \ar[r] & KO_n(C^{\ast}_{L;\Gamma}) \ar[r] & KO_n(C^{\ast}_{r, \mathbb{R}}\Gamma) \ar[r] & \ldots
\end{tikzcd}
\end{equation}

\begin{theorem}\label{mapstolzhigsonthm} Denoting by $B\Gamma=E\Gamma/\Gamma$ the classifying space for Galois $\Gamma$-coverings, then the $(L,G)$-Stolz sequence \ref{l,gstolzseq} with $B\Gamma$ as reference space maps to the universal, localized Higson Roe surgery sequence (\ref{higsonroelocalizedu}).
This means that the following diagram is commutative:

\begin{equation} \label{mapstolzhigson}
	\begin{tikzcd}[sep=small, row sep=1.5em]
		\ldots \ar[r, "\iota"] & R^{spin, (L,G)}_{n+1}(B\Gamma) \ar[r, "\partial"] \ar[d, "Ind^{\Gamma}_{rel}"] & Pos^{spin, (L,G)}_{n}(B\Gamma)\ar[r, "\varphi"] \ar[d, "\rho^{\Gamma}"] & \Omega^{spin, (L,G)}_{n}(B\Gamma)\ar[r, "\iota"] \ar[d, "Ind^{\Gamma}_L"] & R^{spin, (L,G)}_{n}(B\Gamma)\ar[r, "\partial"] \ar[d, "Ind^{\Gamma}_{rel}"] & \ldots \\
		\ldots \ar[r] & KO_{n+1}(C^{\ast}_{r,\mathbb{R}}\Gamma) \ar[r] & KO_n(C^{\ast}_{L,0;\Gamma}) \ar[r] & KO_n(C^{\ast}_{L;\Gamma}) \ar[r] & KO_n(C^{\ast}_{r,\mathbb{R}}\Gamma) \ar[r] & \ldots
	\end{tikzcd}	
\end{equation}

\end{theorem}

Before proving the theorem, we now describe what are the vertical maps appearing in the above diagrams. 
\begin{itemize}
	\item Let $[M]=[M, \partial M, g_{\partial M}, f\colon M \to B\Gamma] \in R^{spin, (L,G)}_{n+1}(B\Gamma)$, where $g_{\partial M}$ is a well-adapted wedge metric of positive scalar curvature defined on the boundary $\partial M$. \\
	Firstly, extend $g_{\partial X}$ to a well-adapted wedge metric $g$ on $M$ (which is, of course, not necessarily of positive scalar curvature). Then, let $\pi\colon M_{\Gamma} \to M$ be the Galois $\Gamma$-cover classified by $f$ with the $\Gamma$-equivariant lift of $g$. 
	Define:
	
	 \[M_{\Gamma}^{\infty}:= M_{\Gamma} \cup_{\partial M_{\Gamma}} (\partial{M_{\Gamma}} \times [0, + \infty)),\]

	which is a non compact spin pseudomanifold with $(L,G)$-singularities without boundary of dimension $n+1$. \\
	 Since the metric on $\partial M_{\Gamma}$, i.e. the $\Gamma$-invariant lift of $g_{\partial M}$, is of positive scalar curvature, then by extending it along the cylinder $\partial{M_{\Gamma}} \times [0, + \infty)$, one gets a metric $g_{\Gamma}$ on $M_{\Gamma}^{\infty}$ which of uniformly positive scalar curvature outside the closed $\Gamma$-invariant subset $M_{\Gamma}$ (by construction and by how the $\Gamma$-action is trivially extended on $M_{\Gamma}^{\infty}$).\\
	  Then, we get a well defined partial secondary local index class of Definition \ref{partialsecondaryclass} associated to it:

	\[Ind^{\Gamma}_{L,M_{\Gamma}}(\slashed{\mathfrak{D}}_{\Gamma}^{g_{\Gamma}})\in KO_{n+1}(C^{\ast}_{L,M_{\Gamma}}(M_{\Gamma}^{\infty};\mathbb{R})^{\Gamma})\]

	Now, we define:
	
	\begin{equation}\label{indrelrel}
		Ind^{\Gamma}_{M_{\Gamma}}(\slashed{\mathfrak{D}}_{\Gamma}^{g_{\Gamma}}):=(ev_1)_{\ast}\left(Ind^{\Gamma}_{L,M_{\Gamma}}(\slashed{\mathfrak{D}}_{\Gamma}^{g_{\Gamma}})\right)\in KO_{n+1}(C^{\ast}(M_{\Gamma} \subset M_{\Gamma}^{\infty};\mathbb{R})^{\Gamma}),
		\end{equation}
	
	which can be also seen as an element in the $KO$-theory of the real Roe $C^*$-algebra of $M_\Gamma$ since by Lemma \ref{inclisomk} we have the following isomorphism:
	 \[KO_{n+1}(C^{\ast}(M_{\Gamma} \subset M_{\Gamma}^{\infty};\mathbb{R})^{\Gamma}) \simeq KO_{n+1}(C^{\ast}(M_{\Gamma};\mathbb{R})^{\Gamma})\]
	Finally, using the mapping $\widehat{f}\colon M_{\Gamma} \to E\Gamma$ which cover the classifying map $f$ of the $\Gamma$-cover, we finally define:
	
	\begin{equation} \label{indrel}
		Ind^\Gamma_{rel}([M]):=\widehat{f}_{\ast}(Ind^{\Gamma}_{M_{\Gamma}}(\slashed{\mathfrak{D}}_{\Gamma}^{g_{\Gamma}})) \in KO_{n+1}(C^{\ast}_{r,\mathbb{R}}\Gamma).
	\end{equation}

	\begin{remark}
	Note that the class (\ref{indrel}) just introduced can be more generally defined whenever the metric has positive scalar curvature outside a closed subset and the space is equipped with a proper map to $B\Gamma$.
	\end{remark}

	The bordism invariance of the class (\ref{indrel}) can be traced back to standard results in the literature. For example, the index classes introduced in the coarse context can be related to Mishchenko-Fomenko index classes (see \cite[Theorem 2.11]{BPR2}) defined for spaces with cylindrical ends. The general idea is to apply the results contained in \cite{Bunke}, and in particular its $K$-theoretic relative index theorem \cite[Theorem 1.2]{Bunke} (see also the proof of Proposition \ref{inddifffact}, where these results are discussed more in detail). Following the notation of Section \ref{secbordismgroups}, where we firstly introduced the $R$-groups for pseudomanifolds with fibered $L$-singularities, consider a bordism $W$ between two cycles, say $[M, g_{\partial M}]$ and $[M', g'_{\partial M'}]$.
	Its boundary \(\partial W\) is made by a bordism \(Z\) between $\partial M$ and $\partial M'$, equipped with a positive scalar curvature metric, with $M$ (respectively $M'$) attached along $\partial M$ (respectively $\partial M'$). Then \(\partial W\) clearly has positive scalar curvature outside a compact set (in particular, outside \(M \sqcup M'\)). One can proceed by applying the K-theoretic relative index theorem, cutting this space along the hypersurface made by $\partial M \sqcup \partial M'$ and to the cylinder constructed on \(\partial M \sqcup \partial M'\), with the obvious extended metric, noting that the index of the latter space is zero as it is equipped with a psc metric.
	Finally, since $\partial W$ has a vanishing local index class by the bordism invariance of the index class, which is still valid in the context of pseudomanifolds with $(L,G)$-singularities (see \cite[Theorem 2.17]{BPR2}), one obtains that 
	\[	Ind^\Gamma_{rel}([M]) = 	Ind^\Gamma_{rel}([M'])\]
	Observe that this idea is exactly the same used to prove \cite[Theorem 1.17]{Bunke}

	\item Let $[M, g, f\colon M \to B\Gamma] \in Pos^{spin, (L,G)}_{n}(B\Gamma)$, where $g$ is a well-adapted psc metric. Then, considering the $\Gamma$-covering $\pi\colon M_{\Gamma} \to M$ classified by $f$ and endowed with the lifted metric $g_{\Gamma}$, we can then defined the $\rho$-invariant associated to $g_{\Gamma}$ since it has positive scalar curvature on all $M_\Gamma$:

\[\rho^{\Gamma}(g_{\Gamma}) \in KO_n(C^{\ast}_{L,0}(M_{\Gamma};\mathbb{R})^{\Gamma})\]

Again, via the mapping $\widehat{f}$ covering $f$, we get:

\[\rho^{\Gamma}([M, g, f\colon M \to B\Gamma]) := \widehat{f}_{\ast}(\rho^{\Gamma}(g_{\Gamma})) \in KO_n(C^{\ast}_{L,0;\Gamma})\]

For the discussion regarding the well-definedness of this map, see the Remark \ref{wdrho} below.

	\item Let $[M, f\colon M \to B\Gamma] \in \Omega_n^{spin, (L,G)}(B\Gamma)$ and $\pi\colon M_{\Gamma} \to M$ be the Galois cover classified by $f$. Then consider a well adapted metric $g$ on $M$ and its $\Gamma$-invariant lift $g_\Gamma$ on $M_\Gamma$: as discussed in the previous section, we have a well defined local index class $Ind^{\Gamma}_L(\slashed{\mathfrak{D}}_{\Gamma}) \in KO_n(C^{\ast}_L(M_{\Gamma};\mathbb{R})^{\Gamma})$. \\
	Finally, via the mapping $\widehat{f}$ covering $f$, one obtains: 
\[Ind^{\Gamma}_L([M, f\colon M \to B\Gamma]):=\widehat{f}_{\ast}(Ind^{\Gamma}_L(\slashed{\mathfrak{D}}_{\Gamma})) \in KO_{n}(C^{\ast}_{L,\Gamma}).
\]

\end{itemize}

Once defined the vertical mappings appearing in the diagram (\ref{mapstolzhigson}), it remains to show that all the three squares there are commutative. 
In particular, the commutativity of the first square will be given directly by a corollary of an important result which we are going to state.

Let $[M, \partial M, g_{\partial M}, f\colon M \to B\Gamma] \in R^{spin, (L,G)}_{n+1}(B\Gamma)$, $\pi\colon M_{\Gamma} \to M$ and $M_{\Gamma}^{\infty}$ be endowed with the extended, $\Gamma$-invariant metric $g_{\Gamma}$, which is of psc outside $X_{\Gamma}$, as discussed in the above definition of $Ind^\Gamma_{rel}$.

 Since the $\Gamma$-action on $M_\Gamma$ is cocompact and preserving the boundary by construction, then the inclusion $i_{\partial}\colon\partial M_{\Gamma} \hookrightarrow M_{\Gamma}$ is a coarse equivalence, hence inducing isomorphism in $KO$-theory. Then, considering the class $Ind^{\Gamma}_{M_{\Gamma}}(\slashed{\mathfrak{D}}_{\Gamma}^{g_{\Gamma}})$ introduced in (\ref{indrelrel}), we get:

\[\left(i_{\partial}\right)_{\ast}^{-1}\left(Ind^{\Gamma}_{M_{\Gamma}}(\slashed{\mathfrak{D}}_{\Gamma}^{g_{\Gamma}})\right) \in KO_{n+1}(C^{\ast}(\partial M_{\Gamma}; \mathbb{R})^{\Gamma}).\]

However, since the lifted metric on the boundary $g_{\partial M}^{\Gamma}$ is of psc, there is also  defined a $\rho$-class associated to it: 
\[\rho^{\Gamma}(g_{\partial M}^{\Gamma}) \in K_n(C^{\ast}_{L, 0}(\partial M_{\Gamma};\mathbb{R})^{\Gamma}).\]

From \cite[Proposition 6.4.3]{HigsonRoe}, a suitable distance function can be defined on the cylinder $\partial M_\Gamma^\infty := \partial M_{\Gamma} \times [0, +\infty) \subset M_{\Gamma}^{\infty}$ in order to have that $KO_{\ast}(C^{\ast}(\partial M_\Gamma^\infty;\mathbb{R})^{\Gamma})=0$. 

Again, since in this case the inclusion of the boundary is a coarse equivalence, this implies $KO_{\ast}(C^{\ast}( M_{\Gamma}^\infty;\mathbb{R})^{\Gamma})=0$.
Therefore, by exactness of the Higson-Roe surgery sequence in terms of (real) localization algebras, there exists a unique class \[(i)_{\ast}^{-1}(Ind^{\Gamma}_{L}(\slashed{\mathfrak{D}}_{\Gamma})) \in KO_{n+1}(C^{\ast}_{L,0}(M_{\Gamma}^\infty;\mathbb{R})^{\Gamma}),\] where $i_{\ast}\colon KO_{n+1}(C^{\ast}_{L,0}( M_{\Gamma}^\infty;\mathbb{R})^{\Gamma}) \to KO_{n+1}(C^{\ast}_{L}( M_{\Gamma}^\infty;\mathbb{R})^{\Gamma})$ is induced by the inclusion $C^{\ast}_{L,0}( M_{\Gamma}^\infty;\mathbb{R})^{\Gamma} \hookrightarrow C^{\ast}_{L}( M_{\Gamma}^\infty;\mathbb{R})^{\Gamma}$.

Finally, we define:

\[\rho^{\Gamma}(M_{\Gamma}):= \partial_{MV}\circ (i)_{\ast}^{-1}(Ind^{\Gamma}_{L}\left(\slashed{\mathfrak{D}}_{\Gamma})\right) \in KO_n(C^{\ast}_{L,0}(\partial M_{\Gamma};\mathbb{R})^{\Gamma}),\]

where $\partial_{MV}\colon KO_{n+1}(C^{\ast}_{L,0}( M_{\Gamma}; \mathbb{R})^{\Gamma}) \to KO_n(C^{\ast}_{L,0}(\partial M_{\Gamma};\mathbb{R})^{\Gamma})$ is the boundary map of the Mayer-Vietoris exact sequence associated to the decomposition of $M_{\Gamma}^{\infty}$ into the union of $M_{\Gamma}$ and $\partial M_\Gamma^\infty$ along the common boundary $\partial M_{\Gamma}$, which represents their intersection.
For a detailed discussion on the Mayer-Vietoris sequences in the localization algebra settings and their properties, we refer entirely to \cite[Chapter 5]{Zeidler}.

\begin{theorem}[Delocalized APS Index Theorem] \label{delocalizedAPS}
	Let $M_{\Gamma}$, $M_{\Gamma}^{\infty}= M_\Gamma \bigcup_{\partial M_\Gamma} \partial M_\Gamma^\infty$ and $g^\Gamma_{\partial M}$ be as above. Then:
	
	\begin{equation}\label{apsindexthm}
		\delta \circ (i_\partial)_{\ast}^{-1}\left(Ind^{\Gamma}_{M_{\Gamma}}(\slashed{\mathfrak{D}}^{g_{\Gamma}}_{\Gamma})\right)=\rho^{\Gamma}(g_{\partial M}^{\Gamma})-\rho^{\Gamma}(M_{\Gamma}) \in KO_n(C^{\ast}_{L,0}(\partial M_{\Gamma};\mathbb{R})^{\Gamma}),
	\end{equation}	
	
	where $\delta\colon  KO_{n+1}(C^{\ast}(\partial M_{\Gamma};\mathbb{R})^{\Gamma}) \to KO_n(C^{\ast}_{L,0}(\partial M_{\Gamma};\mathbb{R})^{\Gamma})$ is the boundary map in the long exact sequence in $KO$-theory associated to:
	
	\[\begin{tikzcd}
		0 \arrow[r] & C^{\ast}_{L,0}(\partial M_{\Gamma};\mathbb{R})^{\Gamma} \arrow[r] & C^{\ast}_L(\partial M_{\Gamma};\mathbb{R})^{\Gamma} \arrow[r, "ev_1"] & C^{\ast}(\partial M_{\Gamma};\mathbb{R})^{\Gamma}  \arrow[r] & 0.
	\end{tikzcd}\]
\end{theorem}	

By applying to $(i_\partial)_{\ast}$ to (\ref{apsindexthm}), one obtains in particular the following.

\begin{corollary} \label{apsindexthmcor}
	In the above hypothesis, one gets:
	
	\begin{equation}\label{cordelocalizedapr}
		\delta\left(Ind^{\Gamma}_{M_{\Gamma}}(\slashed{\mathfrak{D}}^{g_{\Gamma}}_{\Gamma})\right)=(i_{\partial})_{\ast}\left(\rho^{\Gamma}(g_{\partial M}^{\Gamma})\right) \in KO_n(C^{\ast}_{L,0}(M_{\Gamma};\mathbb{R})^{\Gamma}),
	\end{equation}	
	where $\delta\colon  KO_{n+1}(C^{\ast}( M_{\Gamma};\mathbb{R})^{\Gamma}) \to KO_n(C^{\ast}_{L,0}( M_{\Gamma};\mathbb{R})^{\Gamma})$ is the boundary map of the long exact sequence in $KO$-theory associated to (\ref{higsonroelocalizedshort}).
\end{corollary}

These results are treated in \cite[Theorem 6.5]{Zeidler} (see also \cite{PiazzaSchick} for a different approach) in the smooth case. However, all the arguments adapt to our context. Observe that in (\ref{cordelocalizedapr}) there is only one term on the right: in fact, $(i_\partial)_{\ast}(\rho^{\Gamma}(M_{\Gamma}))$ is vanishing by construction since the Mayer-Vietoris sequence associated to $M^\infty_\Gamma$ is exact.

\begin{remark} \label{wdrho}
	We have postponed the discussion regarding the well-definedness of the map \(\rho^\Gamma\), and in fact its bordism invariance, as we require the just stated Delocalized APS Index Theorem for localization algebras. \\
	In fact, a bordism in the group $Pos^{spin, (L,G)}_{n}(B\Gamma)$ is represented by a compact, spin pseudomanifold with $(L,G)$-singularities whose boundary is the disjoint union of two representatives of the bordism class. Then, since the $\rho$-invariant is additive with respect to disjoint union and changes sign for reverse spin structures, one gets that the difference between the two $\rho^\Gamma$ of the boundary components is zero using \ref{delocalizedAPS} and the fact that the localized index class of the bordism is zero since it has positive scalar curvature everywhere.
\end{remark}

We can now proceed and prove Theorem \ref{mapstolzhigsonthm}.
\begin{proof}
	
	As said above, it remains only to show that the three squares in (\ref{mapstolzhigson}) are commutative.
	
	Let $[M]=[M, \partial M, g_{\partial M}, f\colon M \to B\Gamma] \in R^{spin, (L,G)}_{n+1}(B\Gamma)$ and $\delta$ be the boundary map of Corollary \ref{apsindexthmcor}: call $\delta_{\Gamma}\colon KO_{n+1}(C^{\ast}_{r,\mathbb{R}}\Gamma) \to KO_n(C^{\ast}_{L,0;\Gamma})$ the induced map induced between the respecive direct limits by universality. 
	
	\begin{remark}
		We make the following specification about the induced map between direct limits, taking $\delta_\Gamma$ as an example. As we said above, the \textit{universal} $KO$-theory groups are defined via a direct limit construction, i.e. passing through direct families and a universality condition. In our case, direct families are given by the $KO$-theory groups of localization algebras of proper, complete metric spaces with a free cocompact $\Gamma$-action for a discrete group $\Gamma$ and inclusion maps as $\Gamma$-cocompact subsets. 
		
		Assume $X_\Gamma$ and $Y_\Gamma$ to be two such spaces and $\iota_{X,Y}\colon X_\Gamma \to Y_\Gamma$ a map of the family. Then, using the naturality property of the maps induced on $KO$-theory, meaning the fact that the following diagram commutes:
		
		\begin{equation}
			\begin{tikzcd}
				KO_{n+1}(C^{\ast}(X_{\Gamma};\mathbb{R})^{\Gamma}) \ar[d, "(\iota_{X,Y})_*"] \ar[r, "\delta"] & KO_n(C^{\ast}_{L,0}( X_{\Gamma};\mathbb{R})^{\Gamma}) \ar[d, "(\iota_{X,Y})_*"] \\
				KO_{n+1}(C^{\ast}(Y_{\Gamma};\mathbb{R})^{\Gamma}) \ar[r, "\delta"] & KO_n(C^{\ast}_{L,0}( Y_{\Gamma};\mathbb{R})^{\Gamma})
			\end{tikzcd}
		\end{equation}
	together with the universal property of the direct limit, to have an induced map $\delta_\Gamma$ which commutes with each universal map sending an element of the family to the direct limit. 
		\end{remark}
	
	By all the above definitions, we have that:
	
	\[\delta_{\Gamma}\circ Ind^{\Gamma}_{rel}([M])=\delta_{\Gamma} \circ \widehat{f}_\ast \left(Ind^{\Gamma}_{M_{\Gamma}}(\slashed{\mathfrak{D}}_{\Gamma}^{g_{\Gamma}})\right)=\widehat{f}_\ast \circ \delta \left(Ind^{\Gamma}_{M_{\Gamma}}(\slashed{\mathfrak{D}}_{\Gamma}^{g_{\Gamma}})\right), \]
	
	Therefore, we can apply Corollary \ref{apsindexthmcor} in order to obtain:
	
	\[\widehat{f}_\ast \circ \delta \left(Ind^{\Gamma}_{M_{\Gamma}}(\slashed{\mathfrak{D}}_{\Gamma}^{g_{\Gamma}})\right)= \widehat{f}_\ast \circ (i_{\partial})_{\ast}\left(\rho^{\Gamma}(g_{\partial M}^{\Gamma})\right)=(\widehat{f}|_{\partial M_\Gamma})_\ast \left(\rho^{\Gamma}(g_{\partial X}^{\Gamma})\right)   \in KO_n(C^{\ast}_{L,0;\Gamma})\]
	
	However, recalling that $\partial([M])=[\partial M, g_{\partial M}, f|_{\partial M} \colon \partial M \to B\Gamma] \in Pos^{spin, (L,G)}_{n}(B\Gamma)$, we have obtained that:
	
	\[\delta_{\Gamma}\circ Ind^{\Gamma}_{rel}\left([M]\right)=\rho^{\Gamma} \circ \partial \left([X]\right),\]
	which gives exactly the commutativity of the first square.
	
	Let $[M]=[M, g, f\colon M \to B\Gamma] \in Pos^{spin, (L,G)}_{n}(B\Gamma)$.
	
	 Denoting by $j\colon KO_n(C^{\ast}_{L,0}(M_{\Gamma};\mathbb{R})^{\Gamma}) \to KO_n(C^{\ast}_L(M_{\Gamma};\mathbb{R})^{\Gamma})$, and by $j_\Gamma\colon KO_n(C^{\ast}_{L,0;\Gamma}) \to KO_n(C^{\ast}_{L;\Gamma})$ the canonical, universal induced map, observe that by functorality:
	
	\[j_\Gamma \circ \rho^\Gamma([M]):=j_\Gamma \circ \widehat{f}_\ast\left(\rho^{\Gamma}(g_{\Gamma})\right)=\widehat{f}_\ast \circ j\left(\rho^{\Gamma}(g_{\Gamma})\right).\]

	However, observe that by the definition \ref{partialsecondaryclass} of the partial secondary local index class, the inclusion $j$ maps the $\rho$-invariant to the local index class. Then the commutativity of the second square follows by the following, when we also use what observed in Remark \ref{partialtolocal}:
	
	\[\widehat{f}_\ast \circ j\left(\rho^{\Gamma}(g_{\Gamma})\right)=\widehat{f}_\ast \left(Ind^{\Gamma}_L(\slashed{\mathfrak{D}}_{\Gamma})\right)=Ind^{\Gamma}_L \circ \varphi([M]).\]
	
	Let $[M]=[M, f\colon M \to B\Gamma] \in \Omega_n^{spin, (L,G)}(B\Gamma)$, and recall that $\iota([M])=[M, \emptyset, 0, f\colon M \to B\Gamma]$, i.e. includes $M$ with empty boundary. 
	
	Denoting by $(ev_1)_\Gamma$ the universal mapping induced by $(ev_1)_{\ast}$:
	
	\[(ev_1)_{\Gamma} \circ Ind^{\Gamma}_L([M]):=(ev_1)_{\Gamma} \circ \widehat{f}_\ast\left(Ind^{\Gamma}_L(\slashed{\mathfrak{D}}_{\Gamma})\right)=\widehat{f}_\ast \circ (ev_1)_{\ast}\left(Ind^{\Gamma}_L(\slashed{\mathfrak{D}}_{\Gamma})\right).\]
	
	Since $M$ has empty boundary, the space $M_{\Gamma}^{\infty}$ coincides with $M_{\Gamma}$, which implies that the partial secondary local index class lies in $KO_n(C^{\ast}_{L,M_{\Gamma}}(M_{\Gamma};\mathbb{R})^{\Gamma}) \simeq KO_n(C^{\ast}_{L}(M_{\Gamma};\mathbb{R})^{\Gamma})$ (the isomorphism follows by the obvious $C^{\ast}(M_{\Gamma} \subset M_{\Gamma};\mathbb{R})^{\Gamma} \simeq C^{\ast}(M_{\Gamma}; \mathbb{R})^{\Gamma}$).
	
	 In particular, this means that $Ind^{\Gamma}_{L,M_{\Gamma}}(\slashed{\mathfrak{D}}_{\Gamma}^{g_{\Gamma}}) \equiv Ind^{\Gamma}_L(\slashed{\mathfrak{D}}_{\Gamma})$.
	 
	Finally, recalling the definition of the mapping $Ind^{\Gamma}_{rel}$, we get:
	
	\[\widehat{f}_\ast \circ (ev_1)_{\ast}\left(Ind^{\Gamma}_L(\slashed{\mathfrak{D}}_{\Gamma})\right)=Ind^{\Gamma}_{rel}\left(\iota([M])\right), \]
	 which proves the desired commutativity of the third square.

\end{proof}

\newpage

\section{An estimate for the (L,G)-Pos groups} \label{estimate}

In this section, we will address an application of the results achieved in the previous sections, particularly Theorem \ref{mapstolzhigsonthm}, which provided a way to map the Stolz sequence into the Higson-Roe sequence in the context of spaces with \((L, G)\) singularities. Specifically, the aim is to provide a lower bound estimate on the rank of the bordism groups \(Pos^{spin, (L,G)}_*\). To proceed, we will need to introduce some spaces of metrics, particularly those of well-adapted wedge metrics, and the index-difference homomorphism in its formulation adapted to the wedge context.

\subsection{The space of well-adapted wedge metrics}

We now consider the space $\mathcal{R}(M)$ of all Riemannian metrics on a smooth manifold $M$. Recall that when $M$ is compact this is a convex subspace of the Fréchet space of smooth symmetric $2$-tensors on the tangent bundle $TM$. In particular, we are interested in the subspace $\mathcal{R}^+(M) \subset \mathcal{R}(M)$ of all Riemannian metrics with positive scalar curvature.

In the case in which $M$ has non empty boundary $\partial M \neq \emptyset$, one can consider metrics compatible with a collar neighborhood of $\partial M$. Recalling that a collar neighborhood is an open subset $\partial M \times [0,1) \subset M$, then for some $c \in (0,1)$, we ask $\mathcal{R}(M, \partial M)^c$ to be the space of all Riemannian metrics on $M$ which, on a collar neighborhood of the boundary $\partial M \times [0,c]$, is of the form $g_\partial+dx^2$, where $g_\partial$ is a Riemannian metric on the boundary $\partial M$. Again, $\mathcal{R}^+(M, \partial M)^c$ will denote its subset of positive scalar curvature Riemannian metrics (and observe that also $g_\partial$ will be of psc).

\begin{remark}
	The dependence on $c$ can be eliminated easily observing that if $b > c$, $\mathcal{R}^+(M, \partial M)^b \subset \mathcal{R}^+(M, \partial M)^c$ and these inclusion maps are homotopy equivalences. Then one can perform a direct limit, obtaining the space $\mathcal{R}^+(M, \partial M)$.
\end{remark}

\begin{remark} There is an obvious restriction of a metric which product near the boundary to the boundary itself, which then gives rise to a map 
	\[res \colon \mathcal{R}^{+}(M, \partial M) \to \mathcal{R}^{+}(\partial M)\] 
	
	As stated in \cite[Theorem 1.1]{EbertFrenck}, this map is in particular a Serre fibration (meaning that it satisfies the homotopy lifting property for each $CW$-complex), but obviously not a surjective map.
\end{remark}	

Similarly, we now proceed to introduce spaces of metrics in the singular context, particularly in the \((L,G)\)-singular context with well-adapted wedge metrics.

\begin{definition} Let us denote as $\mathcal{R}_w(M)$ the space of well-adapted wedge metrics as defined in \ref{welladaptedwedgemetric} on a compact pseudomanifold $M$ with $(L,G)$-singularities and its subspace $\mathcal{R}^+_w(M)$ of positive scalar curvature metrics.
\end{definition}	

\begin{remark} In particular, using also Theorem \ref{wedgethm}, recall that $g \in \mathcal{R}^+_w(M)$ consists of a triple $(g_{\beta M}, g_{M_r}, \nabla)$, where:
\begin{itemize}
\item $g_{\beta M}$ and $g_r$ are Riemannian metrics of psc on the depth-$1$ stratum $\beta M$ and on the resolution $M_r$;
\item $\nabla$ is a $G$-connection on the $G$-principal bundle over $\beta M$ inducing a splitting on the tangent bundle of the tubular neighborhood $N(\beta M)$;
\item $g_{r}$ is a product metric near the boundary of the resolution $\partial M_r$;
\item on the tubular neighborhood $N(\beta M)$, g takes the form of:

\[(dr^2+r^2g_{\partial M_r/\beta M})\oplus \pi^*(g_{\beta M}),\]
where $g_{\partial M_r/\beta M}$ is a metric on the vertical tangent bundle inducing on each fiber a fixed metric $g_L$ with scalar curvature $k_L=l(l-1)$, with $l=dim(L)$.
\end{itemize}
\end{remark}

The particular structure of these well-adapted wedge metrics allows us to immediately obtain the following maps.

\begin{enumerate}[label=(\roman*)]
\item \[\iota_L\colon \mathcal{R}^+(\beta M) \to \mathcal{R}^+(\partial M_r),\]
 which assings to each psc metric $g_{\beta M}$ the lifted Riemannian submersion metric , on the total space of $\pi_r\colon \partial M_r \to \beta M$. In particular, as already discussed in Section \ref{L,Gsetting}, this is performed thanks to the associated connection to $\nabla$ (which gives a splitting on the tangent bundle of $\partial M_r$), by imposing on the vertical tangent bundle the metric that restricts to the fixed metric \(g_L\) on each fiber. \\
 With a small abuse of notation regarding the vertical metric, then $g$ takes the form: 
 \[\iota_L(g_{\beta M}) = g_L+\pi_r^*g_{\beta M}\]
 
Observe that $\iota_L(g_{\beta M})$ is uniquely determined by $g_{\beta M}$, hence $\iota_L$ is an injective map. 
  
  \item \[res_{M_r}\colon \mathcal{R}^+_w(M) \to \mathcal{R}^+(M_r, \partial M_r),\]
  which assigns to a well-adapted metric $g$ its restriction to the resolution $M_r$.
  
  \item \[res_{\beta M}\colon \mathcal{R}^+_w(M) \to \mathcal{R}^+(\beta M),\] 
  which, similarly to the above, is a restriction map assigning to $g$ its restriction to the depth-1 stratum $\beta M$.

\end{enumerate}

\begin{remark}
	As previously noted, we observe here that we will associate to the resolutions of smoothly stratified spaces metrics that are obtained by restriction from wedge metrics. This should not be confused with the consideration that leads to considering a metric on the resolution (viewed as a blowup) starting from a wedge metric, using the diffeomorphism between the interior of the resolution and the regular stratum. Obviously, the two metrics in question are distinctly different.
\end{remark}		

Recall that since $M$ is a compact pseudomanifold with $(L,G)$-singularities, then its resolution $M_r$ is a compact smooth manifold with boundary. Therefore we can consider its spaces of Riemannian metrics introduced in the beginning of this section.

Let us denote by: 
 
\[\mathcal{R}^{+}(M_r, \partial M_r)_{g_\partial}:=res^{-1}(g_\partial)\] 
 \[\mathcal{R}^+_w(M)_{g_{\beta M}}:=res_{\beta M}^{-1}(g_{\beta M})\]
  the preimages of two restriction maps above, and assume that $\iota_L(g_{\beta M})=g_\partial$. In this situation, we obtain the following commutative diagram:

\begin{equation}
	\begin{tikzcd}
		\mathcal{R}^+_w(M)_{g_{\beta M}} \ar[d] \ar[r] & \mathcal{R}^+_w(M) \ar[d, "res_{M_r}"] \ar[r, "res_{\beta M}"] & \mathcal{R}^+(\beta M) \ar[d, "\iota_L"] \\
		\mathcal{R}^{+}(M_r, \partial M_r)_{g_\partial} \ar[r] & \mathcal{R}^{+}(M_r, \partial M_r) \ar[r, "res"] & \mathcal{R}^{+}(\partial M_r)
	\end{tikzcd}
\end{equation}

We now make the following important observations.

Assume that a psc-metric $g_{\beta M} \in \mathcal{R}^+(\beta M)$ is such that $\iota_L(g_{\beta M}) \in  \mathcal{R}^+(\partial M_r)$ lies in the image of $res$. This means that there is a positive scalar curvature metric on the resolution $M_r$ extending it on its interior. 

Applying the second point of Theorem \ref{wedgethm}, then a well-adapted wedge metric on the tubular neighborhood is obtained. More precisely (check the proof of that Theorem for details), one may need to scale the metric $g_{\beta M}$ by a constant factor to obtain it. Therefore, we conclude that one can obtain a lift in $\mathcal{R}^+_w(M)$ of a rescaled metric from $g_{\beta M}$.

 Similarly, each continuous path in $g_{\beta M}(t)$ of psc-metrics in $\mathcal{R}^+(\beta M)$, with the property that $g_{\beta M}(0)=res_{\beta M}(g_0)$ for some metric $g_0 \in \mathcal{R}^+_w(M)$, can be lifted to a path $g(t)$ in $\mathcal{R}^+_w(M)$ such that $res_{\beta M}(g(t))=g_{\beta M}(t)$, for each $t$. 

Now take into account the vertical arrow on the left, i.e. the map: 

\begin{equation}\label{maphomeometrics}
	\mathcal{R}^+_w(M)_{g_{\beta M}} \to \mathcal{R}^+(M_r, \partial M_r)_{g_\partial},
\end{equation}	
which assigns to $g=(g_{\beta M}, g_{r}, \nabla) $ the metric $g_{r}$ with the property that on the boundary $\partial M_r$ is equal $\iota_L(g_{\beta M})$, i.e. it takes the form 
\begin{equation}\label{gwboundary}
Cg_L \oplus \pi_r^*(g_{\beta M}).
\end{equation} 
Obviously, this map has an inverse, since a metric in $\mathcal{R}^+(M_r, \partial M_r)$ which on the boundary takes the form (\ref{gwboundary}) can be extended to a psc metric on the tubular neighborhood $N(\beta M)$ by construction. 
In particular, the map (\ref{maphomeometrics}) is an homeomorphism, hence a homotopy equivalence, inducing for each $q \geq 0$ the following isomorphisms on the homotopy groups:

\begin{equation}
	\pi_q\left(\mathcal{R}^+_w(M)_{g_{\beta M}}\right) \cong \pi_q(\mathcal{R}^+(M_r, \partial M_r)_{\iota_L(g_{\beta M})}).
	\end{equation}

\subsection{The wedge-index-difference homomorphism}

We now want to introduce a wedge analogue of the index difference homomorphism, which we are going to define below. This, in particular, can be seen as an obstruction for two metrics of positive scalar curvature on a smooth compact manifold to be concordant (recall Definition \ref{concordant}). 
We start then by reviewing the results of our interest in the smooth case.

Let $M$ be a compact manifold, possibly with non-empty boundary $\partial M$. Since the index-difference is defined on homotopy groups, we need to fix a base point in the spaces of metrics introduced in the previous section by which the homomorphism will depend. Let $g_0$ be a psc-metric on $M$ and $h_0$ be its restriction on the boundary $\partial M$. 

If $\partial M \neq 0$, we assume that $g_0$ is collared on the boundary for $h_0$, i.e. of the form $dt^2+h_0$ on a collar neighborhood of $\partial M$, meaning that $g_0 \in \mathcal{R}^+(M, \partial M)_{h_0}$, following the notations introduced before. Now, let us consider both cases, with or without boundary.

\begin{definition}
	Let $M$ be a closed spin manifold and $g_0$ be a fixed psc-metric on $M$. Assuming that $\Gamma$ is a discrete group, and that there exists a proper map $f \colon M \to B\Gamma$. 
	Then by \textbf{index-difference} we mean the homomorphisms:
	
	\[Inddiff^\Gamma_{g_0}\colon \pi_q(\mathcal{R}^+(M)) \to KO_{n+1+q}(C^\ast_{r, \mathbb{R}}\Gamma),\]
	defined below for each $q \geq 0$.
	
	If $M$ is a compact manifold with boundary $\partial M$, fix as before an element $g_0 \in \mathcal{R}^+(M, \partial M)_{h_0}$, and then one obtains for each $q \geq 0$:
	
	\[Inddiff^\Gamma_{g_0}\colon \pi_q\left(\mathcal{R}^+(M, \partial M)_{h_0}\right) \to KO_{n+1+q}(C^\ast_{r, \mathbb{R}}\Gamma).\]
	
In both cases, the spaces of Riemannian metrics are considered as pointed spaces with basepoint $g_0$.
\end{definition}

We now sketch how these homomorphisms are constructed. However, observe that for our purposes we will only need the case where $q =0$.
Consider for simplicity the case when $\partial M = \emptyset$. 

A class in $[\varphi] \in \pi_q(\mathcal{R}^+(M))$ is represented by a continuous map $\varphi \colon S^q \to \mathcal{R}^+(M)$ from the $q$-sphere $S^q$ with base-point $p_0$. This means that $\varphi$ defines a continuous family $g_p$ of psc-metrics on $M$, depending continuously on the parameter $p \in S^q$, such that $g_{p_0}=g_0$. 

Next, consider the product $S^q \times M$ with the product metric on $(p, x) \in S^q \times M$ given by $(g_S(p), g_p(x))$, where $g_S$ denotes the standard metric on the sphere (flat for $q \leq 1$ and round for $q \geq 2$). By \cite[Theorem 3]{GL}, this metric can be assumed to be of positive scalar curvature up to homotopy, hence without changing the class $[\varphi]$. 

We now extend the metric on the $q+1$ disk, whose boundary is identified with $S^q$, in a way that a metric on $D^{q+1} \times M$ which is collared on the boundary is obtained. Of course, such extension will not have in general positive scalar curvature. Then, attach a cylinder along the boundary, obtaining:

\[(D^{q+1} \times M) \cup_{S^{q} \times M} (S^{q} \times M \times \mathbb{R}_+)\simeq\mathbb{R}^{q+1} \times M.\]

By trivially extending the psc metric defined on $S^{q} \times M$ along the cylinder, we have then obtained a space with a metric which is psc outside the compact set $D^{q+1} \times M$.

 Then, in a similar way of Definition \ref{partialsecondaryclass}, by passing to the $\Gamma$-covering classified by $f$, this defines a partial secondary local index class and then an index in $KO_{n+1+q}(C^{\ast}_{r, \mathbb{R}}\Gamma)$ (see (\ref{indrel})).
 
Observe that the above can be extended to the case of a manifold with boundary: however, it is necessary to ask for a class $[\varphi]$ to be in $\pi_q(\mathcal{R}^+(M, \partial M)_{h_0})$, hence to fix a psc metric $h$ on the boundary $\partial M$, in order to get a psc metric on the cylinder attached to $D^{q+1} \times \partial M$. This provides  again a psc metric outside the compact set $D^{q+1} \times M \subset \mathbb{R}^q+1 \times M$, thence an index class in $KO_{n+1+q}(C^{\ast}_{r, \mathbb{R}}\Gamma)$.

Now, let us focus on the case $q=0$: fixed the base point $g_0$, a class in $\pi_0(\mathcal{R}^+(M))$ is represented by a psc metric $g_1 \in \mathcal{R}^+(M)$. Then, consider the interpolation:
\[(1-t)g_0+tg_1, \quad t \in [0,1]\] 
which lies entirely in $\mathcal{R}(M)$ by convexity. By identifying the parameter $t$ with the coordinate on $[0,1]$, this gives a Riemannian metric on the cylinder $M \times [0,1]$ of product type near the boundary and restricting to $g_0$ on $M \times \{0\}$ and to $g_1$ on $M \times \{1\}$.

Attaching cylinders along the two boundary components and trivially extending the metric gives $M \times \mathbb{R}$ with a Riemannian metric which is of psc outside the compact set $M \times [0,1]$: in particular, this metric is equal to $g_0+dt^2 $ on $M \times (-\infty, 0]$ and to $g_1 + dt^2$ on $M \times [1, \infty)$.

The class $Inddiff^\Gamma_{g_0}([g_1]) \in KO_{n+1}(C^{\ast}_{r,\mathbb{R}}\Gamma)$ can be regarded as an obstruction for $g_0$ and $g_1$ to be concordant psc metrics.
In fact, if $g_0$ and $g_1$ are concordant, then there is a psc-metric on the cylinder $M \times [0,1]$ restricting to $g_0$ and $g_1$ on the two boundary components. Then, we have that this metric is of psc outside the emptyset and then its partial secondary local index class is zero.

For more details about the index-difference construction and the proof of the following result, we refer to \cite{BERW} and \cite[Theorem C]{ERW}.

\begin{theorem}\label{inddiffsurj}
	Let $M$ be a compact spin manifold with boundary $\partial M$ with $dim(M)=n\geq 6$ and assume $g_0 \in \mathcal{R}^+(M, \partial M)_{h_0} \neq \emptyset$. Moreover, assume that $M$ admits a mapping $f \colon M \to B\Gamma$, for $\Gamma$ a discrete group, such that the induced map on fundamental groups $f_\ast \colon \pi_1(M) \to \Gamma$ is split-surjective. 
	
	Then, if:
	\begin{itemize}
		
	\item $\Gamma$ satisfies the rational Baum-Connes conjecture;
	\item $\Gamma$ is torsion free and has finite rational homological dimension $d$;
	\end{itemize}

	\[
	Inddiff^{\Gamma}_{g_0} \otimes Id_{\mathbb{Q}}\colon \pi_q\left(\mathcal{R}^+(M, \partial M)_{h_0}\right) \otimes \mathbb{Q} \to KO_{n+1+q}(C^\ast_{r, \mathbb{R}}\Gamma) \otimes \mathbb{Q}
	\]
	
	is surjective whenever $q \geq d - n -1$, meaning that it generates the target as a $\mathbb{Q}$-vector space.
\end{theorem}

Now, assume that $M$ is a compact spin pseudomanifold with $(L,G)$-singularities of dimension $n$ and $f\colon M \to B\Gamma$ is a proper map. 

If $g_0$ and $g_1$ are two well-adapted wedge metrics of positive scalar curvature in $\mathcal{R}^+_w(M)$, then we can follow similarly the construction in the case of a closed manifold and obtain a index-difference homomorphism also in the wedge case.

In fact, consider $M \times [0,1]$: as already discussed, this is a spin pseudomanifold with $(L,G)$-singularities with boundary $M \sqcup M$: we can of course define a well-adapted metric (not in general of psc) which is collared near the boundary and restricting to $g_0$ and $g_1$ on $M \times \{0\}$ and $M \times \{1\}$ respectively.

As before, by attaching cylinder along the two boundary components, one obtains the non compact pseudomanifold with $(L,G)$-singularities $M \times \mathbb{R}$. In particular, the trivial extension on this cylindrical end gives a well-adapted wedge metric on $M \times \mathbb{R}$ which is of psc outside the compact $M \times [0,1]$. Then, by considering the Galois $\Gamma$-covering of over $M$ classified by the map $f$, one associates to $M\times \mathbb{R}$ with such extended metric its partial secondary local index class of Definition \ref{partialsecondaryclass}, and then in particular the class in $KO_{n+1}(C^{\ast}_{r, \mathbb{R}}\Gamma)$ as in (\ref{indrel}).

\begin{remark}The same construction given before for the closed case and $q>0$ can be adapted similarly in the context of pseudomanifolds with $(L,G)$-singularities and well-adapted wedge metrics. In any case, since the treatment presented here will not use the case with \( q > 0 \), we will not report the details.
\end{remark}	

\begin{definition} In the hypothesis above, we define the \textbf{wedge-index-difference homomorphism} as the following homomorphism, defined up to a choice of a well-adapted wedge psc-metric $g_0 \in \mathcal{R}^+_w(M)$ as base-point:

\[\leftindex^{w}Inddiff^\Gamma_{g_0}\colon \pi_q\left(\mathcal{R}^+_w(M)\right) \to KO_{n+1+q}(C^\ast_{r, \mathbb{R}}\Gamma),\]

which works as just discussed. 
\end{definition}

Now we state the following straightforward result.

\begin{lemma} \label{inddifffact2}
	Let $M$ be a compact, spin stratified pseudomanifold with $(L,G)$ singularities and $f\colon M \to B\Gamma$ be a proper map. Once fixed $g_0 \in \mathcal{R}^+_w(M)$, then there exists a homomorphism $\l_{g_0}\colon \pi_0(\mathcal{R}^+_w(M)) \to R_{n+1}^{spin, (L,G)}(B\Gamma)$ such that the following diagram is commutative:

	\begin{equation}
		\begin{tikzcd}[sep=huge]
			\pi_0(\mathcal{R}^+_w(M))  \ar[r,"\leftindex^{w}Inddiff^\Gamma_{g_0}"] \ar[d,"\l_{g_0}"] &  KO_{n+1}(C^\ast_{r, \mathbb{R}}\Gamma) \\
			   R_{n+1}^{spin, (L,G)}(B\Gamma) \ar[ur, "Ind^\Gamma_{rel}"] &  
		\end{tikzcd},
	\end{equation}

	where $Ind^\Gamma_{rel}$ indicates the homomorphism already introduced in Theorem \ref{mapstolzhigsonthm}.
\end{lemma}

\begin{proof}
	We define $l_{g_0}$ as the mapping which sends a class $[g_1] \in \pi_0(\mathcal{R}^+_w(M))$ to the class $[M \times [0,1], M \sqcup M, g_0 \sqcup g_1, \overline{f}\colon M \times [0,1] \to B\Gamma]$, where $\overline{f}(x,t)=f(x)$.
	
	 Then, it is an immediate consequence of the definition of the $Ind^\Gamma_{rel}$ map defined in the previous section to check that:
	
	\[Ind^\Gamma_{rel}\left(l_{g_0}[g_1]\right)=\leftindex^{w}Inddiff^\Gamma_{g_0}([g_1]),\]
	
	which is exactly the commutativity property we wanted to prove.
\end{proof}

For what follows, it suffices to consider a restriction of the wedge-index-difference to a fiber of the map $res_{\beta M}\colon \mathcal{R}^+_w(M) \to \mathcal{R}^+(\beta M)$ defined previously. We indicated it as $\mathcal{R}^+_w(M)_{g_{\beta M}}$, which means that here we are restricting to those psc well-adapted wedge metrics which induces on the depth-$1$ stratum $\beta M$ the psc metric $g_{\beta M}$. If one chooses $g_0 \in \mathcal{R}^+_w(M)_{g_{\beta M,0}}$ as a base point, then in the same way as as before one obtains:

\begin{equation}\label{wedgeinddiff}
	\leftindex^{w}Inddiff^\Gamma_{g_0}\colon \pi_q(\mathcal{R}^+_w(M)_{g_{\beta M,0}}) \to KO_{n+1+q}(C^\ast_{r, \mathbb{R}}\Gamma).
\end{equation}

Observe that in this particular case the interpolation can be always chosen such that the metric on the stratum $\beta M$ is kept fixed, and hence also on the tubular neighborhood and on the boundary of the resolution $M_r$.

What we want want to prove now is an analogue of Theorem \ref{inddiffsurj} in the wedge context, i.e. about the rational surjectivity of the wedge-index-difference map (\ref{wedgeinddiff}). In particular, this can be easily achieved by reducing (\ref{wedgeinddiff}) to a composition of a homotopy equivalence and the ordinary index-difference for manifolds with boundary. This is contained in the following result.

\begin{proposition} \label{inddifffact}
	Let M be a compact, spin stratified pseudomanifold with $(L,G)$-singularities, $M_r$ be its resolution and $f\colon M \to B\Gamma$ as above. Then, if $g_0 \in \mathcal{R}^+_w(M)_{g_{\beta M,0}}$ is fixed, the mapping $\leftindex^{w}Inddiff^\Gamma_{g_0}$ factors as:
	
	\begin{equation}\label{inddifffactdiagram}
		\begin{tikzcd}[sep=huge]
			\pi_0(\mathcal{R}^+_w(M)_{g_{\beta M,0}})  \ar[r,"\leftindex^{w}Inddiff^\Gamma_{g_0}"] \ar[d,"\Huge{\sim}"labl] &  KO_{n+1}(C^\ast_{r, \mathbb{R}}\Gamma) \\
			\pi_0(\mathcal{R}^+(M_r, \partial M_r)_{g_{\partial,0}}) \ar[ur, swap, "Inddiff^\Gamma_{g_{\partial,0}}"] &  
		\end{tikzcd},
	\end{equation}

	where $g_{\partial,0}=\iota_L(g_{\beta M,0}) \in \mathcal{R}^+(\partial M_r)$.
\end{proposition}

For the proof of this Proposition we will need the following result which relates the index associated to a well-adapted wedge metric on $M$ and the indices of the restrictions to the resolution $M_r$ and to the tubular neighborhood $N(\beta M)$. This is based crucially on the $K$-theoretic relative index theorem due to Bunke in \cite{Bunke}. 

Assume that $M$ is a compact spin stratified pseudomanifold with $(L,G)$-singularities with boundary $\partial M$ and $f\colon M \to B\Gamma$ is a proper map. We know that $M$ can be considered as:
 \[M=M_r \bigcup_{\partial M_r} N(\beta M),\]
  i.e. as the union of its resolution $M_r$, which is a smooth manifold with corners, and the tubular neighborhood $N(\beta M)$ of its depth-$1$ stratum $\beta M$. Moreover, assume that $g$ is a well-adapted wedge metric, such that on both $\partial M$ and $ \partial M_r$ restricts to a psc metric (recall that $g$ is of product type near them). 
  
  We begin by attacching a cylinder on the boundary of $M$ and extending trivially the metric, i.e. obtaining:

\[M^\infty:=M\bigcup_{\partial M} (\partial M \times [0,	\infty)).\]

Then, observe that in particular this is a non compact spin stratified pseudomanifold with $(L,G)$-singularities without boundary, with a well-adapted wedge metric of positive scalar curvature outside the compact $M$. 

Moreover, $M^\infty$ can be considered as the union of the following spaces with cylindrical ends, i.e.:

\[M^\infty=M^\infty_r \bigcup_{\partial M^\infty_r} N^\infty,\] 
where:

\[M_{r}^\infty:=M_r \bigcup \left( (M_r \cap \partial M) \times [0,\infty)\right),\]
\[N^{\infty}:= N(\beta M) \bigcup \left((N(\beta M) \cap \partial M) \times [0,\infty) \right).\]

In particular, both are non compact spaces with equal boundary: 

\[\partial M^{\infty}_r=\partial M_r \bigcup \left((\partial M_r \cap \partial M) \times [0, \infty)\right).\]

By construction, both $M^{\infty}_r$ and $N^\infty$ have a psc metric of product type near their boundary $\partial M^\infty_r$. Then, by attaching two cylinders on both these spaces and extending the psc metric trivially, one obtains $M^\infty_{r,cyl}$ and $N^\infty_{cyl}$ and call the respective metrics $g_r^\infty$ and $g_N^\infty$, which are both of psc outside a compact set. Observe that the first is a non compact, complete manifold without bundary, while the second is still a non compact pseudomanifold with $(L,G)$-singularities without boundary.

\begin{proposition} \label{bunkewedge}
	Let $M$, $M^\infty_{r,cyl}$ and $N^\infty_{cyl}$ as just discussed. Then:
	
	\begin{equation}\label{bunkerrelformula}
		Ind^\Gamma_{rel}(M^\infty, f)=Ind^\Gamma_{rel}(M^\infty_{r,cyl}, f|_{M^r})+Ind^\Gamma_{rel}(N^\infty_{cyl}, f|_{N(\beta M)}) \in KO_n(C^{\ast}_{r, \mathbb{R}}\Gamma),
	\end{equation}
	
	where the above indices are the ones as defined in (\ref{indrel}) and where we emphasized the mappings used to define the above classes.
	
\end{proposition}	

\begin{proof}
	The idea is to strictly follow the proof of \cite[Theorem 1.2]{Bunke}, see also \cite[Proposition 6.6]{AlbinPiazza}. In particular, following the notations there, we call $W_1:=M^\infty_r$, $V_1:=N^\infty$ which implies that:

	\[M_1:=M^\infty=M^\infty_r \bigcup_{\partial M^\infty_r} N^\infty,\]

	i.e. the union of $W_1$ and $V_1$ along their common boundary. Moreover, let $W_2:= \partial M^\infty_r \times [0,\infty)$ and $V_2:= \partial M^\infty_r \times (-\infty,0]$ and similarly:
	
	\[M_2:=\partial M^\infty_r \times \mathbb{R}.\]
	
	Both spaces have metric with positive scalar curvature outside a compact set by what discussed above: in particular, $M_2$ has psc everywhere. Moreover, observe that $M_2$ is smooth, while $M_1$ is stratified with $(L,G)$-singularities and both are non compact.

	Now, define:
	\[M_3:=W_1 \bigcup_{\partial M^\infty_r} V_2=M^\infty_{r,cyl},\]
	 \[M_4:=W_2 \bigcup_{\partial M^\infty_r} V_1=N^\infty_{cyl}.\]
	 
In \cite{Bunke} it is required that all these $M_i$, for $i=1,\ldots, 4$, are partitioned by a compact hypersurface.	 
However, in our case this hypersurface is $\partial M^{\infty}_r$ for both $M_1$ and $M_2$, which is non-compact, but by construction, since it is the product of a compact space and the real line, a suitable tubular neighborhood can be determined inside all $M_i$.

Then one obtains the thesis by direcly applying \cite[Theorem 1.2]{Bunke} substituting the spaces $H^l$, $l \geq 0$, with the respective self-adjoint domains of the Atiyah-Singer operators of each spaces. As also remarked in \cite[Proposition 6.6]{AlbinPiazza} (in the context of the signature operator on stratified spaces), the assumption of the existence of a positive, smooth, compactly supported function $f$ such that $(D^2+f)^{-1}$ exists, is satisfied since $(D^2+Id)$ is invertible by positivity of the operator $D^2$ and approximating the constant function one by suitable non negative, smooth functions compactly supported in the respective regular parts ($D$ in our case will be the Atiyah-Singer operators on the spaces above). 

Finally, formula (\ref{bunkerrelformula}) is obtained observing that the index for the space $M_2$ vanishes since $M_2$ the metric is of positive scalar curvature everywhere.

\end{proof}

Now we are ready to proceed to the proof of  \ref{inddifffact}.

\begin{proof}[Proof of Proposition \ref{inddifffact}]
	First of all, we observe that the vertical homotopy equivalence $\pi_0(\mathcal{R}^+_w(M)_{g_{\beta M,0}}) \to \pi_0(\mathcal{R}^+(M_r,\partial M_r)_{g_{\partial,0}})$ in the diagram is simply given by taking the class represented by the restriction of a representative to $M_r$.
	
	By definition of the wedge-index-difference of $[g_1] \in \pi_0(\mathcal{R}^+_w(M)_{g_{\beta M,0}})$, one firstly construct the stratified space $M \times [0,1]$, endowing it with a well-adapted wedge metric of product type near the two boundary components and restricting there to $g_0$ and $g_1$ respectively. 
	
	In particular, by construction such metric is induced by an interpolation $g_t$ between $g_0$ and $g_1$ which keeps fixed $g_{\beta M,0}$ on $\beta M$, i.e. such that for each $t \in [0,1]$, $res_{\beta M}(g_t)=g_{\beta M,0}$. 
	
	Then, for each $t$, this path keeps fixed the metric $g_{\partial,0}=\iota_L(g_{\beta M,0})$ on $\partial M_r$ and on the tubular neighborhood $N(\beta M)$ too.
	
	Therefore, $M \times [0,1]$ is a compact, spin, stratified pseudomanifold with $(L,G)$-singularities with boundary, whose resolution is obviously given by $M_r \times [0,1]$, hence endowed with a metric which on its boundary $\partial M_r 	\times [0,1]$ restricts to the psc metric $g_{\partial,0}+dt^2$. 
	
	Then we are exactly in the situation of Proposition \ref{bunkewedge}. However, by what discussed before, on the tubular neighborhood of the depth-$1$ stratum of $M \times [0,1]$, the metric is of psc everywhere. Then, there is no index contribution due to it, implying that:
	
	\[Ind^\Gamma_{rel}(M \times \mathbb{R}, f \times Id)=Ind^\Gamma_{rel}((M\times[0,1])^\infty_{r,cyl},f|_{M^r}\times Id)\in KO_{n+1}(C^\ast_{r, \mathbb{R}}\Gamma).\]

	Looking at this equation, the term on the left is by construction the $KO$-theory class of $\leftindex^{w}Inddiff^\Gamma_{g_0}([g_1])$, while the term on the right is exactly $Inddiff^\Gamma_{g_{\partial,0}}([g_1|_{M_r}])$.
\end{proof}

We emphasize that, to prove the previous proposition, the particular structure of the well-adapted wedge metrics near the boundary of the resolution is necessary, namely that there they are product-like.

\subsection{Main result}

As already pointed out, our goal was to present an analogue of Proposition \ref{inddiffsurj} for the index-difference in the wedge setting. This can be easily done in virtue of Proposition \ref{inddifffact}, since there is proved that the wedge-index-difference of a pseudomanifold reduces to that of its resolution, i.e. for a smooth, compact, spin manifold with boundary for which \ref{inddiffsurj} holds.
Then, we summarize this in the following result.

\begin{proposition} \label{winddiffsurj}
	Let $M$ be a compact, spin stratified pseudomanifold with $(L,G)$-singularities, with link $L$ simply connected. Moreover, let $dim(M)=n\geq6$ and $g_0 \in \mathcal{R}^+_w(M) \neq \emptyset$. Assume that:
	
	\begin{itemize}
		\item there is a mapping $f\colon M \to B\Gamma$, $\Gamma$ a discrete group, such that $f_{\ast}\colon 	\pi_1(M) \to \Gamma$ is split-surjective;
		\item $\Gamma$ satisfies the rational Baum-Connes conjecture;
		\item $\Gamma$ is torsion-free and has finite rational homological dimension $d$.
		\end{itemize}
	
	Then: 
	
	\[\leftindex^{w}Inddiff^\Gamma_{g_0} \otimes Id_{\mathbb{Q}} \colon \pi_0\left(\mathcal{R}^+_w(M)\right) \otimes \mathbb{Q}\to KO_{n+1}(C^\ast_{r, \mathbb{R}}\Gamma) \otimes \mathbb{Q}\]
	is surjective whenever $d\leq n+1$.
	
\end{proposition}

\begin{proof}
	
	As discussed above, this follows from \ref{inddiffsurj}, by applying Proposition \ref{inddifffact}. The only point which needs to be discussed here regards the split surjectivity of the map $f|_{M_r}$, which is needed in order to apply Proposition \ref{inddiffsurj} for the index-difference map of the resolution.  
	
	This justifies the simply connectedness assumption on the link $L$: in this case, the fibration $\partial M_r \to \beta M$ induces an isomorphism between the fundamental group of the boundary of the resolution and the one of the depth-$1$ stratum $\beta M$, while by construction the tubular neighborhood has a deformation retract to $\beta M$. 
	
	Then, by Van Kampen's Theorem:
	
	\[\pi_1(M) \simeq \pi_1(M_r) \ast_{\pi_1(\beta M)} \pi_1(\beta M) \simeq \pi_1(M_r),\]
		and in particular the inclusion of the resolution $M_r$ into $M$ induces an isomorphism on the fundamental groups. Therefore, the induced map $(f|_{M_r})_*$ of the restriction of $f$ to $M_r$ is a composition of an isomorphism and $f_*$, which is by hypothesis split surjective, and then is a split surjective map too.
	
	Finally, by Proposition \ref{inddiffsurj} the mapping $Inddiff^\Gamma_{g_{\partial, 0}} \otimes Id_{\mathbb{Q}}$ is surjective, and then by (\ref{inddifffactdiagram}) the surjectivity of the restriction of the wedge-index-difference to \[\pi_0\left(\mathcal{R}^+_w(M)_{g_{\beta M,0}}\right) \otimes \mathbb{Q}, \quad  \text{where:} \ \ g_{\beta M,0}=res_{\beta M}(g_0),\]
	 is obtained.
\end{proof} 

\begin{corollary} \label{indrelsurj}
	With the same hypothesis of proposition \ref{winddiffsurj}, the map:
	
	\[Ind^\Gamma_{rel} 	\otimes Id_{\mathbb{Q}}\colon R_{n+1}^{spin,(L,G)}(B\Gamma) \otimes \mathbb{Q} \to KO_{n+1}(C^\ast_{r, \mathbb{R}}\Gamma) \otimes \mathbb{Q}\]
			
	is surjective whenever $d\leq n+1$.		
\end{corollary}

\begin{proof}
	The surjectivity of the rational wedge-index-difference map of proposition \ref{winddiffsurj} implies the thesis using the commutativity of the diagram of lemma \ref{inddifffact2}.
\end{proof}

Now, let $M$ be a compact, spin stratified pseudomanifold with $(L,G)$-singularities of dimension $n$, and assume $\pi_1(M)=\Gamma$ be a discrete group. It is well known that $M$ admits a canonical map $f\colon M \to B\Gamma$ up to homotopy, classifying its universal cover, say $\tilde{M}$. 

By general properties, this map is $2$-connected, which we recall this means that $f$ induces isomorphisms on the $i$-th homotopy groups for $i=0,1$, while it induces a surjective map for $i=2$. Then, we know that $f$ functorially induces maps:

\[f_*\colon R_k^{spin,(L,G)}(M) \to R_k^{spin,(L,G)}(B\Gamma),\]

which are isomorphisms for all $k\geq0$ from Theorem \ref{2conrgroup}. 

Moreover, we have shown in \ref{indrelsurj} that under certain hypothesis on the group $\Gamma$, the map $Ind^\Gamma_{rel}$ is rational surjective.
In particular, if $\Gamma$ is a torsion-free group satisfying the rational Baum-Connes conjecture, with rational homological dimension $d$ such that $d-n-1\leq0$, then one obtains the following diagram, in which each map is assumed to be tensorialized with the identity on $\mathbb{Q}$:

\begin{equation} \label{rankdiagram}
	\begin{tikzcd}[sep=small, row sep=2.5em, scale cd=0.85]
		\ldots \ar[r,"\varphi"] & \Omega_{n+1}^{spin,(L,G)}(M) \otimes \mathbb{Q} \ar[r, "\iota"] \ar[d,"f_* "] & R^{spin,(L,G)}_{n+1}(M)\otimes \mathbb{Q}  \ar[r, "\partial "] \ar[d, "f_* ","\Huge{\sim}"labl] & Pos^{spin,(L,G)}_{n}(M)\otimes \mathbb{Q} \ar[r, "\varphi"] \ar[d, "f_* "] &  \ldots \\
		\ldots \ar[r,"\varphi"] & \Omega_{n+1}^{spin,(L,G)}(B\Gamma)\otimes \mathbb{Q}  \ar[r, "\iota"] \ar[d,"Ind^{\Gamma}_L"] & R^{spin,(L,G)}_{n+1}(B\Gamma) \otimes \mathbb{Q} \ar[r, "\partial"] \ar[d, "Ind^{\Gamma}_{rel}", twoheadrightarrow] & Pos^{spin,(L,G)}_{n}(B\Gamma)\otimes \mathbb{Q} \ar[r, "\varphi "] \ar[d, "\rho^{\Gamma}"]  &  \ldots \\
		\ldots \ar[r] & KO_{n+1}(C^{\ast}_{L;\Gamma}) \otimes \mathbb{Q} \ar[r] \ar[d,"\Huge{\sim}"labl] &  KO_{n+1}(C^{\ast}_r\Gamma) \otimes \mathbb{Q} \ar[r] & KO_{n}(C^{\ast}_{L,0;\Gamma}) \otimes \mathbb{Q} \ar[r] & \ldots \\
		KO_{n+1}(M) \otimes \mathbb{Q} \ar[r, "f_{\sharp}"] \ar[from=ruuu, bend right=21, crossing over] & KO_{n+1}(B\Gamma) \otimes \mathbb{Q} \ar[swap, ur, "\mu", "\rotatebox{-60}{\(\sim\)}"labl] & & & 
	\end{tikzcd}	
\end{equation}

Observe that the curved vertical arrow on the left represents the fact that the composition $Ind_L^\Gamma \circ f_*$, can be thought as the composition:

\[\begin{tikzcd}
	\Omega_n^{spin,(L,G)}(M) \otimes \mathbb{Q} \ar[r] & KO_{n+1}(M) \otimes \mathbb{Q} \ar[r, "f_{\sharp}"] & KO_{n+1}(B\Gamma) \otimes \mathbb{Q},
	\end{tikzcd}
\]

where the first arrow associated to a bordism class the pushforward of the fundamental class in $KO$-homology (see proposition \ref{fundamentalclass}) along the map with range in $M$ of the bordism class, while the second arrow is the one functorially induced by $f$. Moreover, observe that since all the objects here are abelian groups, hence $\mathbb{Z}$-modules, by tensorizing with $\mathbb{Q}$ one annihilates the torsion and obtains a $\mathbb{Q}$-vector space. 

We now state and prove the following theorem, which is inspired by a result due to Schick and Zenobi (see \cite[Theorem 1.1]{SZ}).

\begin{theorem}\label{rankpos}
	Let $M$ be a compact, spin stratified pseudomanifold with $(L,G)$-singularities of dimension $n \geq 6$ with link $L$ simply connected. Assume that $\Gamma=\pi_1(M)$ satisfies all the hypothesis of proposition \ref{winddiffsurj} and denote by $f\colon M \to B\Gamma$ the 2-connected classifying map of the universal cover of $M$. 
	
	If:
	
	 \[k:=\dim\left(Coker\left(f_{\sharp}\colon KO_{n+1}(M) \otimes \mathbb{Q} \to KO_{n+1}(B\Gamma) \otimes \mathbb{Q}\right)\right),\]
	 
    then the following estimate holds:
    
    \[rk\left(Pos^{spin,(L,G)}_n(M)\right)\geq k\]
\end{theorem}

\begin{proof}
	As already remarked, all the hypothesis are needed to obtain a diagram like (\ref{rankdiagram}). 
	We want to show that using its properties, the rational bordism group $Pos^{spin,(L,G)}_n(M)$ is large at least as the Cokernel of $f_\sharp$.
	
	 Take a non zero element of $c\in Coker(f_\sharp)\subset KO_{n+1}(B\Gamma) \otimes \mathbb{Q}$ and consider its image under the assembly map $\mu(c) \neq 0$ (recall that $\mu$ is an isomorphism since $\Gamma$ is required satisfying the rational Baum-Connes conjecture). 
	 
	 Since both $Ind^\Gamma_{rel}$ and $f_\ast$ are surjective, we can consider the preimage of $\mu(c)$ along their composition $r \in R_{n+1}^{spin,(L,G)}(M) \otimes \mathbb{Q}$. Now, we want to show that $r$ is mapped injectively in $Pos_n^{spin,(L,G)}(M)$ using the boundary map $\partial$. This happens because $r \notin  Im(\iota)$. In fact, if $r$ lies in the image of $\iota$, then take a preimage of $r$ in $\Omega_{n+1}^{spin,(L,G)}(M) \otimes \mathbb{Q}$: its image under the composition $Ind^\Gamma_{rel} \circ f_\ast \circ \iota$ is by consturction $\mu(c)$, but then by commutativity of (\ref{rankdiagram}), this would imply that $c$ is in the image of $f_\sharp$, which is a contradiction.
	 
	 Then, since $Im(\iota)=Ker(\partial)$ by exactness, $r$ is mapped injectively in $Pos_n^{spin,(L,G)}(M) \otimes \mathbb{Q}$. 
	
	Finally, the thesis is obtained by repeating this argument for all basis elements of $Coker(f_\sharp)$ and by recalling that the rank of an abelian group $G$ (i.e. a $\mathbb{Z}$-module) coincides with the dimension of the $\mathbb{Q}$-vector space $G \otimes \mathbb{Q}$.
\end{proof}

\newpage

	\addcontentsline{toc}{section}{References}
	\bibliographystyle{plain}
	\nocite{*}
	\bibliography{bibliography}

\end{document}